\documentclass[a4paper, 11pt, twoside, leqno]{amsart}

\usepackage[T1]{fontenc}
\usepackage[utf8x]{inputenc}
\usepackage[english]{babel}
\usepackage{amsmath, amsthm, amssymb, stmaryrd}
\usepackage{mathrsfs, esint, bbm}   
\usepackage{fancyhdr}               
\usepackage{fixltx2e}               
\usepackage{enumitem}               
\usepackage[all,cmtip]{xy}          
\usepackage{graphicx}
\usepackage{hyperref}
\usepackage{cancel}

\newtheorem{theorem}{Theorem}            
\newtheorem{corollary}[theorem]{Corollary}
\newtheorem{lemma}[theorem]{Lemma}
\newtheorem{prop}[theorem]{Proposition}

\newtheorem{mainthm}{Theorem}          

\theoremstyle{definition}              
\newtheorem{definition}{Definition}

\theoremstyle{remark}                  
\newtheorem{step}{Step}

\newtheorem{remark}{Remark}

\newcommand{\N}{\mathbb{N}}                 
\newcommand{\Z}{\mathbb{Z}}                 
\newcommand{\R}{\mathbb{R}}                 
\renewcommand{\H}{\mathscr{H}}              
\let\oldS\S                                 
\renewcommand{\S}{\mathbb{S}}               

\newcommand{\weak}{\rightharpoonup}              

\newcommand{\T}{\mathrm{T}}
\newcommand{\gf}{\mathfrak{g}}
\renewcommand{\P}{\mathbb{P}}

\renewcommand{\a}{\mathbf{a}}
\renewcommand{\b}{\mathbf{b}}
\newcommand{\be}{\mathbf{e}}
\newcommand{\db}{\mathbf{d}}
\renewcommand{\d}{\mathrm{d}}               
\newcommand{\D}{\mathrm{D}}

\newcommand{\epsi}{\varepsilon}
\newcommand{\eps}{\epsi}

\newcommand{\Harm}{\mathrm{Harm}}

\DeclareMathOperator{\Id}{Id}                                       
\DeclareMathOperator{\dist}{dist}                                   
\DeclareMathOperator{\sign}{sign}                                   
\DeclareMathOperator{\ind}{ind}                                     
\DeclareMathOperator{\tang}{tan}
\DeclareMathOperator{\curl}{\curl}                                  
\renewcommand{\div}{\mathrm{div}\,}                                 
\newcommand{\Vg}{\mathrm{vol}_g}
\renewcommand{\O}{\mathrm{O}}
\newcommand{\abs}[1]{\left| #1 \right|}                             
\newcommand{\norm}[1]{\left\| #1 \right\|}                          
\newcommand{\csubset}{\subset\!\subset}                             

\usepackage{color}
\usepackage[normalem]{ulem} 

\numberwithin{equation}{section}
\numberwithin{definition}{section}
\numberwithin{theorem}{section}
\numberwithin{remark}{section}

\begin{document}

\title[Dynamics of vortices on surfaces]
{Dynamics of Ginzburg-Landau vortices\\ for vector fields on surfaces}
\author{Giacomo Canevari}
\address{Dipartimento di Informatica, 
  Universit\`a di Verona,
  Strada le Grazie 15, \mbox{37134} Verona, Italy.}
\email[G. Canevari]{giacomo.canevari@univr.it}
\author{Antonio Segatti}
\address{Dipartimento di Matematica ``F. Casorati'',
  Universit\`a di Pavia,
  Via Ferrata 5, \mbox{27100} Pavia, Italy.}
\email[A. Segatti]{antonio.segatti@unipv.it}

\date{\today}

\begin{abstract}
In this paper we consider the gradient flow of the following Ginzburg-Landau type energy
\[
 F_\eps(u) := \frac{1}{2}\int_{M}\abs{D u}_g^2 +\frac{1}{2\eps^2}\left(\abs{u}_g^2-1\right)^2\Vg.  
\]
This energy is defined on tangent vector fields on a $2$-dimensional closed and oriented Riemannian manifold $M$  (here $D$ stands for the covariant derivative) and depends on a small parameter $\eps>0$. 
If the energy satisfies proper bounds, when $\eps\to 0$ the second term forces the vector fields to have unit length. However, due to the incompatibility for vector fields on $M$ between the Sobolev regularity and the unit norm constraint, critical points of $F_\eps$ tend to generate a finite number of singular points (called vortices) having non-zero index (when the Euler characteristic is non-zero). These types of problems have been extensively analyzed in the recent paper by R. Ignat \& R. Jerrard \cite{JerrardIgnat_full}. As in Euclidean case (see, among the others \cite{BBH}), the position of the vortices is ruled by the so-called renormalized energy. 

In this paper we are interested in the dynamics of vortices. We rigorously prove that the vortices move according to the gradient flow of the renormalized energy, which is the limit behavior when $\eps\to 0$ of the gradient flow of the Ginzburg-Landau energy.


 
\end{abstract}

\maketitle
\tableofcontents


\section{Introduction and Main Result}
\label{sec:intro_main}
A prominent example of the non-trivial interplay between PDEs models, material science, 
differential geometry and topology is the the study of the formation and the evolution of the so-called topological defects, that is regions in the material where the order parameter (in this paper, a vector field on a $2$-dimensional Riemannian manifold) experiences a very rapid change. 
To put this phenomenon in perspective, 
we let $(M,g)$ be a closed (i.e.~compact, connected and without boundary), oriented, $2$-dimensional Riemannian manifold with Euler characteristic $\chi(M)$.  
The classical Poincaré-Hopf Theorem (see \cite{AGM-VMO}) provides a precise relation between the zeros of continuous vector fields on $M$ and the topology of $M$. 
More precisely, any continuous vector field $u$ with finitely many zeros $\left\{a_1,\ldots,a_n\right\}$ must satisfy 
\begin{equation}
\label{eq:PHclassic}
\sum_{j=1}^n\text{ind}\left(u,\partial B_\delta (a_j)\right)\, = \,\chi(M), 
\end{equation}
where~$\ind(u, \, \partial B_\delta(a_j))$ is a local topological invariant
--- the index --- computed on the boundary of a small geodesic ball~$B_\delta(a_j)$,
with radius~$\delta$ and center~$a_j$. The radius~$\delta$ must be taken small enough,
so that the ball~$B_\delta(a_j)$ does not contain any zero of~$u$ other than~$a_j$.
(For the definition of the index, see Section~\ref{sect:stationary} below.)
An immediate consequence is that a non-vanishing continuous vector field 
exists if and only if the Euler characteristic of $M$ is equal to zero, namely only if $M$ is a torus (up to homeomorphism).
In other words, if $\chi(M)\neq 0$ then $M$ does not support a globally defined, unit-norm, continuous vector field.  
Interestingly, the very same topological obstruction holds for Sobolev vector fields. 
More precisely (see \cite{AGM-VMO} and \cite{ssvM3AS}), indicating with $D$ the Levi-Civita connection, we have that the Sobolev class
\begin{equation}
\label{eq:h1vuoto}
\left\{v:M\to TM: v(p)\in T_p M,\,\,\abs{v} =1\,\,\text{a.e. in }M,\,\,\abs{D v}_g\in L^2(M)\right\}\neq \emptyset\Longleftrightarrow \chi(M)=0.
\end{equation}  
Thus, the Dirichlet-type energy on unit-norm vector fields
\begin{equation}
\label{eq:dirichlet}
F(u):=\frac{1}{2}\int_{M}\abs{D u}_g^2 \Vg
\end{equation}
is well defined only if the Euler characteristic of $M$ is zero.
The energy $F$ is interesting both from the geometric and the modellistic point of view. 
On the one hand, the energy $F$ is usually called {\itshape Total Bending} (see \cite{wiegmink95}) and (quantitatively) measures how a globally defined unit-norm vector field $u$, if it exists, fails to be a parallel vector field (namely a vector field $u$ for which $D_v u =0$ for any $v$).  
On the other hand, the energy $F$ is related, for instance, to the modelling of nematic shells, namely a rigid particle coated by a thin film of nematic liquid crystal (see \cite{ball_survey}, \cite{ball_defects}, \cite{LubPro92}, \cite{Nelson02}, \cite{Straley71}, \cite{Virga94}) when one neglects the effects of the extrinsic geometry of the substrate on the elastic energy. We refer to \cite{NapVer12L}, \cite{NapVer12E} and \cite{ssvPRE} for models including extrinsic effects. 

When $\chi(M) \neq 0$ the energy $F \equiv +\infty$, therefore its minimization requires at first a relaxation of at least one of the constraints in the Sobolev class of vector fields in \eqref{eq:h1vuoto}.  
Looking back at the definition \eqref{eq:h1vuoto} we see that there are many ways to relax one of the constraints and obtain 
a non-empty class on which the energy $F$ is finite. 
A common strategy (and the one we will follow in this paper) consists in relaxing the unit norm constraint and studying the asymptotic behavior of the Ginzburg-Landau energy 
\begin{equation}
\label{eq:GL_intro}
F_\eps(u) := \frac{1}{2}\int_{M}\abs{D u}_g^2 +\frac{1}{2\eps^2}\left(\abs{u}_g^2-1\right)^2\Vg.  
\end{equation}
This problem has been addressed by Ignat \& Jerrard in the recent paper \cite{JerrardIgnat_full}. Note that this type of approximation is completely intrinsic as it makes no reference to any embedding of the manifold. 

There are however other possible strategies to obtain a relaxed version of the energy $F$. 
For instance, if our $2$-dimensional manifold is an embedded surface in $\mathbb{R}^3$, 
a possible strategy consists in replacing $M$ with a suitable triangulation of $M$ and the energy $F$ (actually its extrinsic version where the covariant derivative is replaced with the derivative of $\mathbb{R}^3$) with a discrete energy defined on unit norm vector fields sitting in the vertices of the triangulation of $M$ (see \cite{gamma-discreto}). 
This approach, contrary to the Ginzburg-Landau approximation, is extrinsic in the sense that it depends on how the manifold $M$ sits in $\mathbb{R}^3$. 
However, these two approaches have the common feature of giving information, in a suitable asymptotic regime, about the location, and the charge of the vortices. 

Since our work is very much related with the work of R. Jerrard \& R. Ignat, we briefly recall some of the main results of~\cite{JerrardIgnat_full} regarding \eqref{eq:GL_intro}.

 Let~$i$ be an almost 
complex structure on~$M$, that is,
an operator~$TM\to TM$ which restricts to a linear map
on each tangent plane~$T_p M$ and satisfies~$i^2+1 = 0$
(where~$1$ denotes the identity on~$TM$). For any (sufficiently regular)
vector field~$u$ on~$M$, we define the~$1$-form 
\begin{equation}
\label{eq:prejac_intro}
 j(u):= (D u, iu)_g
\end{equation}
and the $2$-form
\begin{equation}
\label{eq:vorticity_intro}
 \omega(u):= \d j(u) + \kappa \Vg,
\end{equation}
where~$\kappa$ is the Gauss curvature of~$M$.
The form~$\omega(u)$ is called the vorticity of~$u$.
In the Euclidean setting, when $u\colon\R^2\to\R^2$, 
the vorticity~$\omega(u)$ essentially reduces to the 
Jacobian determinant of~$u$, i.e.~$\omega(u) = 2(\det \nabla u) \, \d x$. 
As in the Euclidean case, if a sequence~$u_\eps$ satisfies
a logarithmic energy bound such as $F_\eps(u_\eps) \leq C \abs{\log\eps}$
(for some constant~$C$ that does not depend on~$\eps$)
then the sequence of vorticities~$\omega(u_\eps)$ is compact.
Controlling the vorticity of~$u_\eps$ is key to
proving a sharp lower bound for the energy of~$u_\eps$
and obtain compactness results for the minimizers of~$F_\eps$.
However, when the genus $\gf = \frac{2-\chi(M)}{2}\neq 0$, it is
also necessary to control the $L^2$-projection of the $1$-form $j(u_\eps)$ on the space of harmonic $1$-forms, $\Harm^1(M)$.
$\Harm^1(M)$ is a real vector space of dimension $2\gf$.
Let~$\P$ be the $L^2$-projection onto~$\Harm^1(M)$.
Ignat and Jerrard~\cite{JerrardIgnat_full} showed that,
if a sequence of vector field~$u_\eps$ satisfies 
suitable energy bounds, then there exist distinct points
$a_1$, \ldots, $a_n$ in~$M$, integers~$d_1$, \ldots, $d_n$,
a form~$\xi\in\Harm^1(M)$ and a (non-relabelled) subsequence,
such that
\begin{align*}
&\d j(u_\eps)\xrightarrow{\eps\to 0} 2\pi \sum_{j=1}^n d_j\delta_{a_j}\qquad \hbox{ in } W^{-1,p}(M) \quad
 \textrm{for any } p\in (1, \, 2) \\
& \mathbb{P}j(u_\eps) \xrightarrow{\eps\to 0}\xi
\end{align*}
The integers~$d_j$ represent the charges of
topological singularities which arise at the points~$a_j$.
They satisfy
\[
 \sum_{j=1}^n d_j = \chi(M).
\]
The harmonic $1$-form $\xi\in \Harm^1(M)$ depends non-trivially
on~$\a:=\left(a_1, \, \ldots, \, a_n\right)$
and~$\db:=(d_1, \, \ldots, \, d_n)$, as it must satisfy
the constraint $\xi\in\mathcal{L}(\a, \, \db)$
where~$\mathcal{L}(\a, \, \db)$ is a suitable subset of~$\Harm^1(M)$
defined by the position and the charges of the singularities.
(See~\cite[Section~2.1]{JerrardIgnat_full} and Section~\ref{sec:gradient_reno}
for the definition of~$\mathcal{L}(\a, \, \db)$.)
If the~$u_\eps$'s are minimizers of~$F_\eps$, then~$n = \abs{\chi(M)}$,
$d_j = \sign(\chi(M))$ for any~$j$ and 
\begin{equation} \label{eq:F-Gamma}
 F_\eps(u_\eps) = \pi n\abs{\log\eps}
  + W(\a, \, \db, \, \xi) + n\gamma + \mathrm{o}(1)
  \qquad \textrm{as }\eps\to 0.
\end{equation}
Here, the function $W = W(\a, \, \db, \, \xi)$ is the so-called
renormalized energy, which accounts for the interaction between
the singular points~$a_j$. It can be characterised in terms
of the Green function for the Laplace-Beltrami operator on~$M$ 
and the Gauss curvature of~$M$ (see~\cite[Proposition~2.4]{JerrardIgnat_full} and section~\ref{sec:gradient_reno} below). When the~$u_\eps$'s 
are minimizers of~$F_\eps$, $(\a, \, \db, \, \xi)$ minimize the
renormalized energy~$W$. The constant~$\gamma > 0$ in~\eqref{eq:F-Gamma}
is the so-called core energy, which accounts for the energy
in a suitably small neighbourhood of each singular point~$a_j$
(we refer to~\cite[Section~2.3]{JerrardIgnat_full} for
the definition of~$\gamma$ --- in Ignat and Jerrard's notation,
$\gamma = \iota_F$). Moreover, again up to a 
to a subsequence, minimizers converge to a unit-norm vector field~$u_*$,
called a canonical harmonic vector field,
which is smooth in $M\setminus \left\{a_1, \, \ldots, \, a_n\right\}$
and is uniquely determined by~$(\a, \, \db, \, \xi)$ up to a global rotation.
These results represent the extension of the classical two-dimensional results of Bethuel, Brezis and Hélein \cite{BBH} 
to the case of vector fields on a Riemannian manifold (see also \cite{AlicandroPonsiglione} for a
$\Gamma$-convergence result).

In this paper, we discuss the dynamics of the vortices on a two dimensional closed and oriented Riemannian manifold $(M,g)$. We show that 
 the vortices move according to the gradient flow of the  renormalized energy. 
The gradient flow of the renormalized energy emerges as limit dynamics of the (properly rescaled) gradient flow of the Ginzburg-Landau energy \eqref{eq:GL_intro}. 
More precisely, 
we consider the asymptotic behavior when $\eps\to 0$ of the Ginzburg Landau dynamics
\begin{equation}
\label{eq:GL-intro}
\begin{cases}
\displaystyle\frac{1}{\vert \log\eps\vert}\partial_t u_\eps -\Delta_g u_\eps + \frac{1}{\eps^2}\left(\vert u_\eps\vert^2-1\right)u_\eps = 0,\,\,\,\,\hbox{ a.e. in }M\times (0,T),\\
\\
\displaystyle u_\eps(x,0) = u_0^{\eps}\,\,\,\,\hbox{ a.e. in }M.
\end{cases}
\end{equation}
The operator $\Delta_g$ is the so called rough Laplacian (see \eqref{eq:rough} for the definition) while
the factor $\frac{1}{\abs{\log\eps}}$ in front of the time derivative of $u_\eps$ comes from scaling and indicates that the isolated zeros of $u_\eps$ with nonzero index move with velocities of order $\abs{\log\eps}^{-1}$.

Our results confirm in the Riemannian setting the classical planar results
of Rubistein \& Sternberg \cite{Rubistein_Sternberg95}, Jerrard \& Soner 
\cite{jerrard_soner_dyna}, F.H. Lin
\cite{lin1}, \cite{lin2} and Sandier \& Serfaty
\cite{SS-GF}.
Quite recently Chen \& Sternberg \cite{chen-stern2014} have proved the convergence of the (rescaled) heat flow of the Ginzburg-Landau energies 
defined on maps $u:M\to \mathbb{C}$ where $M$ is a two-dimensional compact simply connected Riemannian manifold without boundary. 
Our setting and our results are however different for some crucial points. 
As we deal with vector fields, the topology of~$M$ triggers the formation of the vortices via the Poincaré-Hopf Theorem, while in \cite{chen-stern2014} the vortices arise thanks to the choice of initial conditions and the scaling of the equation. 
Another difference lies in the energy or, more precisely, in the Dirichlet part of the energy. 
The energy $\frac{1}{2}\int_{M}\abs{\nabla u}^2_g \Vg$ defined on maps $u:M\to \mathbb{C}$ (in this case $\nabla$ stands for the Riemannian gradient) is conformally invariant. This is not the case for the energy $\frac{1}{2}\int_{M}\abs{D u}^2_g \Vg$ defined on vector fields (as before $D$ is the covariant derivative). 
The conformal invariance combined with the Uniformization Theorem permits (as in the ``static'' situation considered in \cite{baraket}) to transfer the analysis from the manifold $M$ to a reference domain in $\mathbb{R}^2$. 
As a result, the renormalized energy is modeled on the Laplacian in $\mathbb{R}^2$ with corrections due to the conformal change. In particular, the interaction between the vortices is ruled by the Green function of the Laplacian in $\mathbb{R}^2$. 
On the contrary, when dealing with the energy $\frac{1}{2}\int_{M}\abs{D u}^2_g \Vg$, we have to deal with the Laplace-Beltrami operator on $M$ and with its Green function. 
The resulting renormalized energy inherits the complexity of the Green function (see \cite{steiner2005})
 and couples the Gauss curvature of $M$ with the position and with the charge of the vortices. Moreover, it explicitly depends on the harmonic $1$-form $\xi$.
Finally,  
the interaction term between the positions of the vortices behaves logarithmically only when the geodesic distance between the two vortices is small enough. 

We will consider initial conditions $u_\eps^0$ satisfying the following set
of hypothesis.
Given~$n\in\Z$, $n\geq 1$,
we consider a~$n$-uple of \emph{distinct}
points~$\a^0 = (a^0_1, \, \ldots, \, a^0_n)$ in~$M$
and a~$n$-uple of integers~$\db = (d_1, \, \ldots, \, d_n)$
such that
\[
 \sum_{j=1}^n d_j = \chi(M), \qquad \abs{d_j} = 1 \quad \textrm{for any } j.
\]
Moreover, we fix an harmonic $1$-form $\xi^0$ such that
$\xi^0\in \mathcal{L}(\a^0, \, \db)$. 
Once we have chosen $\a^0$, $\db$, $\xi^0$ as above, 
we consider a sequence of vector fields $u^0_{\eps}\in H^1_{\tang}(M)$ (see \eqref{eq:W1p} for the definition of this Sobolev space) such that 
 \begin{align}
 & \omega (u^{0}_\eps)\xrightarrow{\eps \to 0} 2\pi \sum_{j=1}^n d_j \delta_{a_j^{0}}\, \,\,\,\text{ in } W^{-1,p}(M)\label{eq:initial_vorticity_intro}\\
 & F_\eps(u_\eps^{0})\le \pi n\vert \log\eps\vert + W(\a^0, \, \db, \, \xi^0) + n\gamma + \mathrm{o}_{\eps\to 0}(1),\label{eq:well_prepared_intro} \\
 & \norm{u^0_\eps}_{L^\infty(M)} \leq 1. \label{eq:initial_Linfty_intro}
 \end{align} 
Energetically speaking we are saying that the sequence $u^0_{\eps}$ is a recovery sequence for $W(\a^0, \, \db, \, \xi^0)$. 
A sequence~$u^0_\eps$ that satisfies~\eqref{eq:initial_vorticity_intro}, \eqref{eq:well_prepared_intro}, \eqref{eq:initial_Linfty_intro}
will be called a sequence of \emph{well-prepared initial conditions}.
If~$u^0_\eps$ is a sequence of well-prepared
initial conditions, then for any $\eps>0$ there exists a smooth solution $u_\eps$ of \eqref{eq:GL-intro}.
The main result of the paper is the following
\begin{mainthm}
\label{th:main1}
Let $(M,g)$ be a two dimensional closed and oriented Riemannian manifold. 
Let $u_\eps$ be a sequence of solutions of~\eqref{eq:GL-intro}
with $u_\eps^{0}\in H^1_{\tang}(M)$ a sequence of well-prepared initial conditions. 
Then there exists a time $T^*\in (0,T)$ (the first collision time)  such that for any $t\in (0,T^*)$
\begin{align*}
& \omega(u_\eps(t))\xrightarrow{\eps\to 0}  2\pi\sum_{j=1}^n d_j\delta_{a_j(t)}\qquad \hbox{ in } W^{-1,p}(M) \quad
 \textrm{for any } p\in (1, \, 2) \\
& \mathbb{P}\left(j(u_\eps(t))\right) \xrightarrow{\eps\to 0}\xi(t),
\end{align*}
where $\a(t):=(a_1(t), \, \dots, \, a_n(t))$ is 
an~$n$-uple of distinct points such that~$\a\in H^1(0, \, T^*; \, M^n)$, $\a(0) = \a^0$.
Moreover, $\xi\in H^{1}(0,T^*;\Harm^1(M))$ and 
\begin{equation} \label{eq:xi_mainth}
 \xi(t)\in \mathcal{L}\left(\a(t), \, \db\right) 
  \quad \textrm{for any } t\in (0, \, T^*), 
  \qquad \xi(0) = \xi^0.
\end{equation}
Finally, $\a$ is a solution of the gradient flow of the renormalized energy $W$, namely
\begin{equation}
\label{eq:gradflowW_mainth}
\begin{cases}
\displaystyle\frac{\d}{\d t} \a(t) = 
-\frac{1}{\pi}\nabla_{\a} W(\a(t),\db,\xi(t)) 
\qquad \textrm{for any } t\in (0,T^*),\\[.4cm]
\a(0) = \a^0.
\end{cases}
\end{equation}
\end{mainthm}
Although the function~$W$ is differentiable,
the definition of~$\nabla_\a W$
requires some care, because the variables~$(\a, \, \db, \, \xi)$
are \emph{not} independent --- they are related to each other
by the constraint~$\xi\in\mathcal{L}(\a, \, \db)$. Therefore,
the gradient~$\nabla_\a W$ must be taken in a suitable sense.
We address this issue in Section~\ref{sec:gradient_reno}
(see, in particular, Lemma~\ref{lemma:bundle}).
The maps~$\a$, $\xi$ are uniquely identified by the
conditions~\eqref{eq:xi_mainth}, \eqref{eq:gradflowW_mainth}
--- see Corollary~\ref{cor:uniqueness} below. This justifies why the convergence in Theorem~\ref{th:main1} holds for the whole sequence of $\eps$. 

The proof of Theorem~\ref{th:main1} relies on the abstract scheme developed by Sandier \& Serfaty in \cite{SS-GF}. 
This scheme essentially consists of three main points. 
First of all, we need that the Ginzburg-Landau energy, upon rescaling, $\Gamma$-converges to the renormalized energy $W$. As anticipated, this first point is provided by the recent results of Ignat \& Jerrard (\cite{JerrardIgnat_full}).
Second of all, we need a lower bound relating the properly rescaled $L^2$ norm of $\partial_t u_\eps$ with the norm of the limit velocity $\a'$. 
Finally, we need a further lower bound relating the properly rescaled $L^2$ norm of 
the gradient of the Ginzburg-Landau energy with the norm of the gradient of the renormalized energy. 
As in the Euclidean case, the proof of the first lower bound relies on the so-called product estimates (see Sandier \& Serfaty in \cite{SS-product}) which are extended to our Riemannian setting.  
The proof of the second lower bound relies on the expression for  the gradient of the renormalized energy in terms of the canonical harmonic vector field, which is obtained as the limit of $u_\eps(t)$ along a suitable subsequence of $\eps$. This approach is in the spirit of known results in \cite{BBH} and \cite{lin1}, \cite{lin2}. 
This strategy requires a careful study of the differentiability properties of the renormalized energy $W$ and of the structure of its gradient with respect to the position of the vortices.

We conclude by recalling that the dynamics of Ginzburg-Landau type vortices for vector fields on surfaces including the effects of the extrinsic geometry of $\R^3$ has been recently addressed in the paper \cite{AG-DCDS}.

\section{Preliminary Material}
\label{sec:preliminary}

\subsection{Differential Geometry preliminaries}
\label{ssec:diffgeo}

We let $(M,g)$ be a closed (compact and without boundary), 
oriented $2$-dimensional Riemannian manifold with
Euler characteristics $\chi(M)$, Gauss curvature $\kappa$ 
and Levi Civita connection $D$. 
We let $\Vg$ denote the $2$-volume form on $M$ induced 
by the metric $g$ and we set $\abs{g}:=\sqrt{\abs{\det g}}$.
We will write~$B_r(p)$ for the open geodesic ball in~$M$
with center~$p\in M$ and radius~$r>0$.
We will use a similar notation for balls in~$\R^2$ 
as well. However, in this paper we only consider 
balls~$B_r(0)\subseteq\R^2$ centered at the origin.

We let $\displaystyle TM:=\bigsqcup_{x\in M}T_x M$ the tangent bundle of $M$. 
A vector field on $M$ is a map
\begin{equation}
\label{eq:def_vector}
v: M\to TM,\,\,\,\,\,\,v(x)\in \T_x M.
\end{equation}
We let $\abs{v}_g:=g(v,v)^{1/2}=(v,v)_g^{1/2}$ denote length of the vector $v$ with respect to the metric $g$. 

We recall that for a tensor of type $(1,1)$ its Riemannian norm is given by
\begin{equation}
\label{eq:normDv}
\abs{D v}_g^2 := g^{ij}(D_i v,D_j v)_g, 
\end{equation}
(with some abuse of notation we indicate with the same symbol $\abs{\cdot}_g$ both the norm of a vector and the norm of a tensor)
and thus
\begin{align} 
\abs{D v(x)}_g^2 &= \abs{D_{\tau_1}v(x)}_g^2 +\abs{D_{\tau_2}v(x)}_g^2, \label{normHom} \\
(D v(x), \, Dw(x))_g &= (D_{\tau_1} v(x), \, D_{\tau_1}w(x))_g +
 (D_{\tau_2} v(x), \, D_{\tau_2}w(x))_g  \label{scalprodHom}
\end{align}
for any vector fields~$v$ and~$w$ and any~$x\in T_x M$,
where $\left\{\tau_1,\tau_2\right\}$ is an orthonormal basis for $T_x M$.

As we are interested in an evolution problem, we will need to consider the cylindrical manifold $Q:=[0, \, T]\times M$, with~$T>0$. 
The tangent space $TQ = T([0, \, T]\times M)$ is isomorphic to $\mathbb{R}\oplus TM$ and thus, any vector field in $TM$ can be identified with a vector field in $TQ$. 
The cylindrical metric $g_{\text{cyl}}$ is defined as
\begin{equation}
\label{eq:cyl_metric}
g_\mathrm{cyl}:={\d t}^2 + g.
\end{equation}
Thus, for $u\in TM$ we have that $\vert u\vert_g = \vert u\vert_{g_\mathrm{cyl}}$.
We let $D^Q$ be the Riemannian connection 
on~$(Q, \, g_\mathrm{cyl})$.

\subsection{Differential Forms}
\label{ssec:differential_forms}

Given an $n$-dimensional Riemannian manifold $(N, \, g)$ 
(with or without boundary) we let $T^{k}(T'N)$
denote the bundle of covariant $k$-tensors on $N$.
We let $\Lambda^k (T'N)$ denote the subset consisting of alternating
tensors, namely the disjoint union over~$p\in N$ of the 
spaces~$\Lambda^{k}(T'_pN)$ of alternating covariant $k$-tensors
on $T'_pN$:
\[
\Lambda^k (T'N) = \bigsqcup_{p\in N}\Lambda^{k}(T'_pN).
\]
A differential $k$ form (which we will often call $k$-form)
is a section of $\Lambda^k (T'N)$. 
The space of $k$-differential forms on $N$ is denoted $\Omega^{k}(N)$.
At each point $p\in M$, $\Lambda^{k}(T'_pN)$ has a scalar product
induced by the metric~$g$. 
With some abuse of notation, we will denote it by~$(\cdot,\cdot)_g$.
The scalar products on~$TN$ and~$T'N$ induce an isometric
isomorphism bewteen~$TN$ and~$T'N$, denoted
\[
 \flat\colon TN\to T'N, \qquad \# = \flat^{-1}\colon T'N\to TN
\]
Adopting the standard convention for the indices, we can write
$\flat$, $\#$ in a local coordinate system~$\{\d x^1, \, \ldots, \, \d x^n\}$
as
\[
 X^\flat = g_{jk} \, X^j \, \d x^k, \qquad 
 \omega^\# = g^{jk} \omega_j \frac{\partial}{\partial x^k}
\]
for any vector field~$X$ and~$1$-form~$\omega$.
Here~$g_{jk}$ are the component of the metric tensor~$g$
and~$g^{jk}$ are the components of~$g^{-1}$.
These isomorphisms allow us to define the Riemannian 
gradient of a scalar function~$\varphi\colon N\to\R$,
as~$\nabla\varphi := (\d\varphi)^\#$.
 
We let $p\in N$. For any $v\in T_pN$, 
we define a linear map ${\bf i}_{v}:\Lambda^{k}(T'_p(N)) \to \Lambda^{k-1}(T'_p(N))$, called 
interior multiplication, as
\begin{equation}
\label{eq:interior_mult1}
({\bf i}_{v}\omega)\left(w_1,\ldots,w_k\right) := \omega\left(v,w_1,\ldots,w_k\right),\,\,\,\,\,\forall \omega\in \Lambda^{k}(T'_p(N)).
\end{equation}
This definition extends to differential forms and vector fields by working pointwise. More precisely, if $v$ is a vector field and $\omega$ is a $k$-differential form, we define the a $(k-1)$-differential form by
\begin{equation}
\label{eq:interior_mult}
({\bf i}_{v}\omega)_p := {\bf i}_{v(p)}\omega_p.
\end{equation}

We let
\begin{equation}
\label{eq:stardef}
\star:\Omega^{k}(N)\to \Omega^{n-k}(N)
\end{equation}
denote the Hodge dual operator. The Hodge dual operator is linear and it is uniquely defined by requiring that 
\begin{equation}
\label{eq:star}
\omega\wedge \star \eta = \left(\omega, \, \eta\right)_g \Vg
\end{equation}
for any~$\omega\in\Omega^k(N)$, $\eta\in\Omega^k(N)$.
The Hodge dual is an isometric isomorphism
$\Omega^k(N)\to\Omega^{n-k}(N)$ and it satisfies
\begin{equation} \label{Hodgeinverse}
 \star\star\omega = (-1)^{k(n-k)}\omega 
 \qquad \textrm{for any } \omega\in\Omega^k(N).
\end{equation}
Using the Hodge~$\star$ operator, we may now define
the codifferential~$\d^*\colon\Omega^{k}(N)\to\Omega^{k-1}(N)$ as
\begin{equation} \label{d*}
 \d^*\omega := (-1)^{n(k-1)+1} \star\d\star\omega
 \qquad \textrm{for any } \omega\in\Omega^k(N).
\end{equation}
The differential~$\d$ and the codifferential~$\d^*$ 
satisfy an ``integration by parts'': if
$U$ is an open bounded subset in $M$, the boundary~$\partial U$ 
is smooth (and carries the orientation induced by~$U$), 
and if~$\eta\in\Omega^{k-1}(N)$, $\zeta\in\Omega^k(N)$, then
\begin{equation}
\label{eq:in_parts_form}
 \int_{U}\left(\d \eta, \, \zeta\right)_g \Vg 
 = \int_{U}\left(\eta, \, \d^*\zeta\right)\Vg 
 + \int_{\partial U}\eta\wedge \star\zeta.
\end{equation} 
Equation~\eqref{eq:in_parts_form} follows by writing
$\d(\eta\wedge\star\zeta) = \d\eta\wedge\star\zeta 
+ (-1)^{k-1}\eta\wedge\d(\star\zeta)$ and applying Stokes' theorem.
If~$N$ is closed and~$U = N$, then~\eqref{eq:in_parts_form} reduces to
\[
 \int_N \left(\d \eta, \, \zeta\right)_g \Vg
 = \int_N \left(\eta, \, \d^*\zeta\right)_g\Vg.
\]
In particular, $\d^*$ is the~$L^2$-adjoint of~$\d$.

We will specialize these definitions when $N=M$ (where~$M$, as above,
is a closed, oriented, two-dimensional Riemannian manifold)
and when $N$ is the ``space-time'' cylinder $Q= [0,T]\times M$,
equipped with the cylindrical metric \eqref{eq:cyl_metric}. 
As we will work in the ambient space~$N=M$ most of the time,
we will denote the Hodge dual, differential and codifferential
on the manifold~$M$ simply as~$\star$, $\d$, $\d^*$.
On the other hand, the Hodge dual, differential and codifferential
on the product manifold~$Q$ are denoted as~$\star_Q$, $\d_Q$, $\d_Q^*$,
respectively.

Since~$M$ is $2$-dimensional, the Hodge dual operator
maps $0$-forms (i.e., scalar functions) into~$2$-forms and vice-versa,
as~$\star 1 = \Vg$, $\star\Vg = 1$. The action of~$\star$ on
$1$-forms may be expressed in local coordinates~$\{x^1, \, x^2\}$ as
\begin{align}
&\star\d x^1 = \frac{1}{\sqrt{\vert g\vert}}(g_{12}\,\d x^1 + g_{22}\,\d x^2)\label{eq:starMdx1}\\
&\star\d x^2 = -\frac{1}{\sqrt{\vert g\vert}}(g_{11}\,\d x^1 + g_{12}\,\d x^2) \label{eq:starMdx2}
\end{align}
The Hodge dual of a~$1$-form may also be used to define
an almost complex structure on~$M$, that is,
a map $i:TM\to TM$ that restricts to an isometry on each 
tangent plane~$\T_p M$ and satisfies
\begin{equation}
\label{eq:prop-i-bis}
 i^2 v + v = 0
\end{equation}
for any~$v\in T_pM$. We construct such a map by imposing
\begin{equation}
\label{eq:prop-i}
 (iv, \, w)_g = -(v, \, iw)_g := \Vg(v, \, w)
\end{equation}
for any~$v\in T_pM$, $w\in T_p M$. Note that $(iv,v)_g = 0$ for any $v\in T_pM$. 
Equation~\eqref{eq:prop-i} immediately implies
\begin{lemma}
\label{lem:vettore_ruotato}
 Let $(M, \, g)$ be a two-dimensional Riemannian manifold.
 For any vector field~$X$ on $M$, there holds
 \begin{equation}
 \label{eq:ruotato}
  (i X)^{\flat} = \star X^\flat.
 \end{equation}
\end{lemma}
In particular, if we take 
$X = \nabla \psi$ where $\psi:M\to \mathbb{R}$ is a smooth function,
and define $\nabla^{\perp}\psi:=i\nabla \psi$, then
\begin{equation}
\label{eq:grad_ruotato}
 (\nabla^\perp \psi)^\flat = \star \d \psi.
\end{equation}
On a $2$-dimensional manifold, the codifferential~$\d^*$
may be written as~$\d^* = -\star\d\star$. 
If~$\Psi\in\Omega^2(M)$ is a $2$-form and~$\psi:=\star \Psi$, then 
\begin{equation}
\label{eq:codifferential_explicit1}
 \d^* \Psi = -\star \d(\star \Psi) = -\star \d\psi. 
\end{equation}

Suppose now the ambient manifold is~$Q := [0, \, T]\times M$.
Given a local coordinate system~$\{t, \, x^1, \, x^2\}$,
we may write the Hodge dual of~$1$-forms on~$Q$ as
\begin{align}
&\star_{Q} \d x^1 = \frac{1}{\sqrt{\vert g\vert}}(-g_{12}\,\d t\wedge \d x^1 
   - g_{22} \, \d t\wedge \d x^2)\label{eq:starQdx1}\\
&\star_{Q} \d x^2 = \frac{1}{\sqrt{\vert g\vert}}(g_{11}\,\d t\wedge \d x^1 
   + g_{12} \, \d t\wedge \d x^2)\label{eq:starQdx2}.
\end{align}
In particular, let~$X$ be a vector field on~$Q$
of the form~$X = X^k\frac{\partial}{\partial x^k}$, with $X^k = X^k(t,x)$ for $k=1,2$.
By comparing~\eqref{eq:starQdx1}--\eqref{eq:starQdx2}
with~\eqref{eq:starMdx1}--\eqref{eq:starMdx2}, we obtain
(${\bf i}_{t}$ denotes the interior multiplication by $\frac{\partial}{\partial t}$)
\begin{equation}
\label{eq:interiordt}
 {\bf i}_{t} (\star_{Q} X^{\flat} )= -\star X^b.
\end{equation}
Let $\psi:M\to \mathbb{R}$ be a smooth function. We let~$X := \nabla \psi$
and consider a $2$-form~$\omega := \omega_k \d t\wedge \d x^k$ 
with $\omega_k =  \omega_k(t,x)$. 
Thanks to \eqref{eq:star}, \eqref{eq:interiordt}, \eqref{eq:ruotato} and \eqref{eq:interior_mult} 
we obtain the following identity:
\begin{equation}
\label{eq:star_grad_ruotato}
 \begin{split}
  \omega \wedge \d \psi &= \left( \omega, \star_{Q}X^\flat\right)_g \\
   & =\left( {\bf i}_t \omega, {\bf i}_t( \star_{Q}X^\flat)\right)_g 
   =-\left( {\bf i}_t\omega, \star X^\flat\right)_g = -\left({\bf i}_t \omega, (\nabla^{\perp}\psi)^\flat\right)_g \\
   &= -\omega\left(\frac{\partial}{\partial t},\nabla^{\perp}\psi\right).
 \end{split}
\end{equation}

\subsubsection{Differential Operators}
\label{sssect:diff_op}
We recall the definition of a few
differential operators. Throughout this subsection, we
work in the ambient manifold~$M$.
The divergence of a vector field~$v$ is defined
as the Lie derivative~$\mathcal{L}_v$ of the volume form in the 
direction of~$v$, that is
\begin{equation}
\label{eq:divergence_2}
  (\div v) \Vg := \mathcal{L}_v\Vg = \d (\mathbf{i}_{v}\Vg) 
\end{equation}
As~$\abs{v}_g^2 \Vg = v^\flat\wedge\star v^\flat$,
we may equivalently write
\[
 \div v = \star(\mathcal{L}_v\Vg)
 = \star\d(\star v^\flat) = - \d^* v^\flat
\]
Since~$\d^*$ is the $L^2$-adjoint of~$\d$, it follows that
$-\div$ is the~$L^2$-adjoint of the Riemannian gradient~$\nabla$.
Moreover, for any smooth, open set~$U\subseteq M$ and 
any smooth function~$\varphi\colon M\to\R$,
Equation~\eqref{eq:in_parts_form} implies
\begin{equation} \label{eq:ip_div}
 \int_{U} (\nabla\varphi, \, v)_g \, \Vg 
 = - \int_{U} \varphi \, (\div v) \, \Vg 
 + \int_{\partial U} \varphi \, (\nu, \, v)_g \, \d\H^1
\end{equation}
where~$\nu$ is the exterior-pointing unit normal to~$\partial U$.

Given a vector field~$v$ on~$M$, the covariant derivative~$D v$
may be regarded as (fiber-wise linear) operator~$\T M\to \T M$,
which maps a tangent vector~$w\in \T_p M$ into the 
tangent vector~$D_w v\in T_p M$. Equivalently, the operator~$D$
maps sections of~$T M$ into sections 
of~$\mathrm{Hom}(TM, \, TM) \simeq TM\otimes T'M$.
The bundle~$TM\otimes T'M$ is equipped with a scalar product,
as defined in~\eqref{normHom}-\eqref{scalprodHom}.
The~$L^2$-adjoint~$D^*$ of~$D$ maps sections of~$TM\otimes T'M$
into sections of~$TM$, i.e.~vector fields.
We define the rough Laplacian (also known as connection Laplacian or Bochner Laplacian) as
\begin{equation}
\label{eq:rough}
 -\Delta_g v := D^* D v. 
\end{equation}
By definition, the rough Laplacian satisfies 
\begin{equation} \label{eq:ip_rough}
 -\int_M (\Delta_g v, \, w)_g \Vg = \int_M (D v, \, D w)_g \, \Vg
\end{equation}
for any smooth vector fields~$v$, $w$ on~$M$.
We can characterise more explicitely $D^*$, and hence $\Delta_g$,
in terms of the covariant derivative $D$.
Let $u$, $v$, $w$ be smooth vector fields. Since the covariant derivative
is compatible with the metric, we have
\[
 \begin{split}
  (D_w u, \, v)_g = \d(u, \, v)_g (w) - (u, \, D_w v)_g
  = \div\!\left((u,\, v)_g w\right) - (u, \, v)_g \, \div w - (u, \, D_w v)_g
 \end{split}
\]
Therefore, for any smooth open set $U\subseteq M$, 
\begin{equation} \label{eq:ip_covariant}
 \int_U (D_w u, \, v)_g \, \Vg = \int_{\partial U} (u,\, v)_g \, (w, \, \nu)_g \,\d\H^1
 -\int_U  (u, \, v)_g \, \div w \, \Vg - \int_U (u, \, D_w v)_g \, \Vg
\end{equation}
In particular, $D^*_w = - D_w - \div w$.
Comparing \eqref{eq:ip_rough} with \eqref{eq:ip_covariant},
we obtain an expression for $\Delta_g$ in terms of an orthonormal 
tangent frame $\{\tau_1, \, \tau_2\}$:
\begin{equation} \label{eq:rough_frame}
 -\Delta_g u = D^*_{\tau_k} D_{\tau_k} u
 = - D_{\tau_k} D_{\tau_k} u - (\div\tau_k) D_{\tau_k} u
 = - D_{\tau_k} D_{\tau_k} u + D_{D_{\tau_k}\tau_k} u
\end{equation}
(the last equality is obtained by differentiating the orthonormality
conditions $(\tau_h, \, \tau_k)_g = \delta_{hk}$).

We will consider another notion of Laplacian, that is, the Hodge Laplacian for differential forms.
The Hodge Laplacian of a $k$-form $\omega$ is defined as 
\begin{equation}
 -\Delta \omega := \left(\d\d^* + \d^*\d\right)\omega.
\end{equation}
Thanks to~\eqref{eq:in_parts_form}, 
the Hodge Laplacian satisfies
\[
 -\int_M (\Delta\omega, \, \zeta)_g \Vg 
 = \int_M \left((\d\omega, \, \d\zeta)_g + (\d^*\omega, \, \d^*\zeta)_g
 \right) \Vg
\]
for any smooth~$k$-forms~$\omega$, $\zeta$. 
In particular, a form~$\omega$ is harmonic (i.e., $\Delta\omega=0$)
if and only if $\d\omega=0$ and~$\d^*\omega=0$.
However, the rough Laplacian and the Hodge Laplacian 
do \emph{not} coincide: in general, $(\Delta_g v)^\flat\neq \Delta v^\flat$.
(In fact, we have the Weitzenb\"ock
identity~$(\Delta_g v)^\flat - \Delta v^\flat = A v^\flat$,
where~$A$ is an operator of order zero depending on
the curvature of~$M$.)

\subsubsection{Connection $1$-form}
\label{sssec:connection1}
Let~$U\subseteq M$ be an open set and let~$\tau_1$,
$\tau_2$ be smooth vector fields on~$U$. We say 
that~$\left\{\tau_1, \, \tau_2\right\}$ is an orthonormal,
positvely oriented, tangent frame on $U$
--- or simply a \emph{moving frame} --- if 
\begin{equation}
 \left(\tau_j,\tau_k\right)_g = \delta_{jk},\qquad 
 \text{and}\qquad \Vg(\tau_1,\tau_2) =1\,\,\,\,\,\\\text{ on } U.
\end{equation}
The connection $1$-form associated with~$\{\tau_1, \, \tau_2\}$ 
is defined by
\begin{equation}
\label{eq:connection1}
\mathcal{A}(v):= \left(\tau_1, \, D_v\tau_2\right)_g.
\end{equation}
for any smooth vector field $v$ on $U$. 
Note that the following holds
\[
\mathcal{A}(v) = -\left(\tau_2,D_v\tau_1\right)_g
\]
Since $\left\{\tau_1,\tau_2\right\}$ are orthonormal, we get that 
\begin{equation}
\label{eq:prop_connection}
D_{v}\tau_1 = -\mathcal{A}(v)\tau_2\,\,\,\,\,\,\,\,\,D_v\tau_2 = \mathcal{A}(v)\tau_1.
\end{equation}
The connection~$1$-form satisfies
\begin{equation} 
\label{eq:dconnection}
 \d\mathcal{A} = \kappa\Vg,
\end{equation}
where~$\kappa$ is the Gauss curvature of~$M$
(see for instance~\cite[Proposition~2, p.~92]{doCarmo}).
We let $\mathbb{A}$ be the vector field (sometimes called spin connection) 
\begin{equation}
\label{eq:spin_connection}
\mathbb{A}:= \mathcal{A}^{\sharp},\,\,\,\,\,\,\,\hbox{ namely }\,\,\,\,\mathbb{A}^j:= g^{jk}\mathcal{A}_k.
\end{equation}
By possibily modifying the frame $\left\{\tau_1,\tau_2\right\}$,
we can always assume (see \cite[Lemma 6.1]{ssvM3AS}) that 
\begin{equation}
\label{eq:div_free_A}
\div \mathbb{A} = 0
\end{equation}

\subsection{Special coordinate systems and the Ginzburg-Landau energy in coordinate}
\label{ssec:coordinates}

\subsubsection{Normal coordinates}
\label{sssec:normal}

We recall that the exponential map 
$\mathrm{Exp}_p\colon T_pM\to M$
at a point~$p\in M$ is defined by
\[
 \mathrm{Exp}_p(v) := \gamma_{p,v}(1) 
 \qquad \textrm{for any } v\in T_p M,
\]
where~$\gamma_{p,v}\colon\R\to M$ is the unique geodesic 
such that~$\gamma_{p,v}(0) = p$, $\gamma_{p,v}^\prime(0) = v$.
As~$M$ is compact and smooth, there exists a number~$\delta_* > 0$
(the injectivity radius of~$M$) such that, 
for any~$\delta\in(0, \, \delta_*)$, $\mathrm{Exp}_p$
restricts to a diffeomorphism from~$B_\delta(0)\subseteq T_p M$
to its image, $B_\delta(p) := \mathrm{Exp}_p(B_\delta(0))\subseteq M$.
$B_\delta(p)$ is the geodesic ball in~$M$ of 
center~$p$ and radius~$\delta$. We choose an
orthonormal, positively oriented basis~$\{\hat{e}_1, \, \hat{e}_2\}$
of~$T_p M$ and consider the isomorphism
$E\colon\mathbb{R}^2\to T_pM$, $E(x^1,x^2):= x^1 \hat{e}_1 + x^2 \hat{e}_2$.
Then, we define the geodesic normal coordinates centered at~$p$ as the map
\[
 \Phi := E^{-1}\circ \text{Exp}_p^{-1} \colon B_\delta(p)\to \mathbb{R}^2.
\]
We have 
$\Phi(B_\delta(p)) = B_\delta(0)\subseteq\R^2$.
The metric tensor~$g$, written in geodesic normal coordinates, satisfies 
\begin{gather}
 g_{jk}(x) = \delta_{jk} + \O\!\left(\abs{x}^2\right)
  \,\,\,\,\,\,\hbox{ and }\,\,\,\,\,\,
  \det g(x)= 1 + \O\!\left(\abs{x}^2\right), \label{eq:basic_normal} \\
 \abs{\nabla\Phi^{-1}(x)} = \O\!\left(\abs{x}\right) \label{eq:basic_normal_nabla} 
\end{gather}
for any~$x\in B_\delta(0)$.
Moreover, 
\begin{gather}
g_{ij}(0) = \delta_{ij}\qquad \textrm{ and } \qquad\partial_j\det g(0) = 0\quad \forall j, \label{eq:det0}\\
\Gamma_{i,j}^k(0) \qquad \textrm{ and }\qquad \partial_k g_{ij}(0) = 0 \quad \forall i,j,k.\label{eq:crist0} 
\end{gather}
In particular, from \eqref{eq:det0} we deduce that the coordinates fields $\frac{\partial }{\partial x_j}$ verify
\begin{equation}
\label{eq:div0}
\div\left(\frac{\partial}{\partial x_j}\right)(0) = 0\qquad \forall j=1,2.
\end{equation}

When working with the cylindrical manifold $Q:= [0, \, T]\times M$, 
we will often use the coordinate system given by 
\begin{equation} \label{eq:cylindricalcoord}
 \Phi_Q := \left(\Id, \, \Phi\right)\colon
 [0,T]\times B_\delta(p)\to [0, \, T]\times\mathbb{R}^2,
 \qquad \Phi_Q(t, \, x) = (t, \, \Phi(x))
\end{equation}
We have ($(\Phi_Q^{-1})^*$ denotes the pullback)
\begin{equation}
\label{eq:change_volume}
 (\Phi_Q^{-1})^*(\Vg\wedge\d t) = \left(1 + \O(\abs{x}^2)\right) \, \d x\wedge\d t.
\end{equation}

\subsubsection{The Ginzburg-Landau energy in geodesic coordinates}
\label{sssec:normalGL}

Given a vector field~$u$, locally defined on an 
open subset of~$u$, we will introduce
a convenient representation of~$u$ in terms of
geodesic coordinates and express the Ginzburg-Landau
energy of~$u$ in terms of geodesic coordinates.
For simplicity, we consider time-independent vector fields,
but the same argument applies verbatim to vector fields
that depend on time as well.
Let~$B_\delta(p)$ be a geodesic ball in~$M$, with~$\delta$ small enough,
so that the geodesic coordinates~$\Phi\colon B_\delta(p)\to\R^2$ are well-defined.
Let~$\{\tau_1, \, \tau_2\}$ be an orthonormal, positively oriented tangent frame 
on~$B_\delta(p)$. We may choose the frame in such a way that the associated
connection~$1$-form~$\mathcal{A}$ (see~\eqref{eq:connection1}) satisfies
\begin{equation} \label{eq:Ap0}
 \mathcal{A}(p) = 0
\end{equation}
(see e.g.~\cite[Lemma 6.1]{ssvM3AS}). This implies
\begin{equation} \label{eq:smallA}
 \abs{\mathcal{A}(\Phi^{-1}(x))} = O(\abs{x})
 \qquad \textrm{for any } x \in B_\delta(0)\subseteq\R^2.
\end{equation}
Let~$u\in H^1_{\tan}(B_\delta(p))$ be a vector field. We write
\[
 u = u^1 \, \tau_1 + u^2 \, \tau_2
\]
for some scalar functions~$u^1$, $u^2$. We define a map
$v =(v^1,v^2)\colon B_\delta(0)\subseteq\R^2\to \mathbb{R}^2$ 
by 
\[
 v^k(x) := u^k(\Phi^{-1}(x))   
\]
for~$x\in B_\delta(0)$, $k=1, \, 2$, so that
\begin{equation} \label{uv}
 u(\Phi(x)) = v^1(x) \, \tau_1(\Phi(x)) + v^2(x) \, \tau_2(\Phi(x))
 \qquad \textrm{for any } x \in B_\delta(0).
\end{equation}
We denote with~$\abs{\cdot}$ and~$(\cdot, \cdot)$
the norm and the scalar product in~$\mathbb{R}^2$, respectively. 
Moreover, we denote with~$\nabla$ the usual (Euclidean)
derivative in $\mathbb{R}^2$. 

\begin{lemma} \label{lemma:normD}
 If~$\{\tau_1, \, \tau_2\}$ is an orthonormal tangent frame
 that satisfies~\eqref{eq:smallA}, then
 \[
  \abs{D u(\Phi^{-1}(x))}^2_g 
   = \left(1 + \O(\abs{x}^2)\right) \abs{\nabla v(x)}^2
    + \O(\abs{v(x)}^2)
 \]
 for any~$x\in B_\delta(0)$.
\end{lemma}
\begin{proof}
 We compute the covariant derivatives of~$u$ with
 respect to~$\tau_1$, $\tau_2$. Thanks to \eqref{eq:prop_connection},
 we obtain
 \begin{align*}
  \begin{split}
   D_{\tau_1} u =  D_{\tau_1}\left(u^1\tau_1 + u^2\tau_2\right) 
   = \left(\d u^1(\tau_1)+ u^2 \mathcal{A}(\tau_1)\right)\tau_1 
   + \left(\d u^2(\tau_1)-u^1\mathcal{A}(\tau_1)\right)\tau_2
  \end{split} \\
  \begin{split}
   D_{\tau_2} u =  D_{\tau_2}\left(u^1\tau_1 + u^2\tau_2\right) 
   = \left(\d u^1(\tau_2)+ u^2 \mathcal{A}(\tau_2)\right)\tau_1 
   +\left(\d u^2(\tau_2)-u^1\mathcal{A}(\tau_2)\right)\tau_2
  \end{split}
 \end{align*}
 As the frame~$\{\tau_1, \, \tau_2\}$ is orthonormal, we deduce
 \[
  \begin{split}
   \abs{D u}^2_g = \abs{D_{\tau_1} u}^2_g + \abs{D_{\tau_2} u}^2_g
   &= \abs{\d u^1(\tau_1)+ u^2 \mathcal{A}(\tau_1)}^2  
   + \abs{\d u^2(\tau_1)-u^1\mathcal{A}(\tau_1)}^2 \\
   &\qquad\qquad + \abs{\d u^1(\tau_2)+ u^2 \mathcal{A}(\tau_2)}^2  
   + \abs{\d u^2(\tau_2)-u^1\mathcal{A}(\tau_2)}^2
  \end{split}
 \]
 and, by expanding the squares,
 \begin{equation} \label{eq:normD1}
  \begin{split}
   \abs{D u}^2_g 
   = \abs{\d u^1}^2_g + \abs{\d u^2}^2_g
   + \abs{\mathcal{A}}^2_g \abs{u}^2_g
   + 2\sum_{k=1}^2\left(u^2 \, \d u^1(\tau_k) 
    - u^1\,\d u^2(\tau_k)\right)\mathcal{A}(\tau_k)
  \end{split}
 \end{equation}
 The last term in the right-hand side of~\eqref{eq:normD1}
 may be estimated as
 \[
   \abs{\sum_{k=1}^2\left(u^2 \, \d u^1(\tau_k) 
    - u^1\,\d u^2(\tau_k)\right)\mathcal{A}(\tau_k)}
    \lesssim \abs{\mathcal{A}}^2_g\left(\abs{\d u^1}_g^2 + \abs{\d u^2}^2_g\right)
    + \abs{u}^2_g
 \]
 Therefore, we obtain
 \begin{equation*} 
  \begin{split}
   \abs{D u}^2_g 
   = \left(1 + \O(\abs{\mathcal{A}}^2_g)\right)
    \left(\abs{\d u^1}^2_g + \abs{\d u^2}^2_g\right)
    + \O(\abs{u}^2_g)
  \end{split}
 \end{equation*}
 By writing the right-hand side in terms of~$v$, and
 applying~\eqref{eq:basic_normal}, \eqref{eq:basic_normal_nabla}
 and~\eqref{eq:smallA}, the lemma follows. 
\end{proof}

We consider the Ginzburg-Landau energy density 
\begin{equation} \label{eq:energy_density}
 e_\eps(u):= \frac{1}{2}\abs{D u}^2_g 
  + \frac{1}{4\eps^2}\left(\abs{u}^2_g -1\right)^2,
\end{equation}
and its Euclidean counterpart
\begin{equation} \label{eq:Euclidean_density}
 \bar{e}_\eps(v):= \frac{1}{2}\abs{\nabla v}^2 
  + \frac{1}{4\eps^2}\left(\abs{v}^2 -1\right)^2.
\end{equation}
Lemma~\ref{lemma:normD} implies 
\begin{equation} \label{eq:energy_dens_normal}
 e_\eps(u)\left(\Phi^{-1}(x)\right) 
 = \left(1 + \O(\abs{x}^2)\right)\bar{e}_\eps(v)(x) 
 + \O(\abs{v(x)}^2) 
\end{equation}
for any~$x\in B_\delta(0) = \Phi(B_\delta(p))$.
Therefore, the Ginzburg-Landau energy on the ball~$B_\delta(p)$
may be written as 
\begin{equation}
 \label{eq:energy_normal_coordinates}
 \int_{B_\delta(p)}e_\eps(u)\Vg =
 \left(1+ \O(\delta^2)\right) \int_{B_\delta(0)}
 \left( \bar{e}_\eps(v)(x) + \O(\abs{v(x)}^2) \right) \d x 
\end{equation}

\subsection{Functional spaces}
\label{ssect:functional}

For functions $u:M\to \mathbb{R}$, the Lebesgue and 
Sobolev spaces $L^p(M)$ and $W^{1,p}(M)$ are defined as in \cite{aubin98}. 
The dual space of $W^{1,p}(M)$ is 
\begin{equation*}
W^{-1,p}(M):=\left(W^{1,q}(M)\right)',\,\,\,\,\,\,\,\frac{1}{p}+\frac{1}{q}= 1.
\end{equation*}
In particular, for a (measure-valued) $2$-form we set
\[
\norm{\mu}_{W^{-1,p}(M)}:=\sup\left\{\int_{M}f \mu:\quad f\in W^{1,q}(M),\quad \norm{f}_{W^{1,q}(M)}\le 1\right\}.
\]
When dealing with vector fields, for $p\in [1,+\infty)$ we set 
\begin{align}
L^p_{\tang}(M) &:=\left\{v: M\to TM:\,\,\,\,\,\,v(x)\in T_x M,\quad \abs{v}_g \in L^p(M)\right\}, \label{eq:Lp} \\
W^{1,p}_{\tang}(M) &:=\left\{v: M\to TM:\,\,\,\,\,\,v(x)\in T_x M,\quad \abs{v}_g +\abs{D v}_g\in L^p(M)\right\},\label{eq:W1p}
\end{align}
where the covariant derivative $D$ is understood in the distributions sense. 
We set $H^1_{\tang}(M):=W^{1,2}_{\tang}(M)$ and 
we recall that smooth vector fields are indeed dense in $W^{1,p}_{\tang}(M)$, see \cite[Lemma 5.1]{JerrardIgnat_full}.

The Nash-Moser embedding \cite{Nash} gives an isometric embedding 
$j\colon\left(M, \, g\right)\to \left(\mathbb{R}^{\ell}, \, (\cdot,\cdot)_{\mathbb{R}^\ell}\right)$. Therefore, we define
\begin{equation}
\label{eq:H1tM}
H^1(a, \, b; \, M):=\left\{v\in H^1(a, \, b; \, \R^\ell): v(t)\in j(M)\quad \text{ for any }t\in (a,b)\right\}
\end{equation}
With some abuse of notation we will identify $M$ with $j(M)\subset \mathbb{R}^{\ell}$.

The Riemannian manifold $(M, \, g)$ is indeed also a metric space. Given two points $P$ and $Q$ in $M$, the (Riemannian) distance between $P$ and $Q$ is 
\begin{equation}
\label{eq:riem_dist}
\dist_g(P, \, Q):= \inf\left\{\text{L}_\gamma,\,\,\,
\gamma\colon[a,b]\to M \hbox{ is piecewise smooth,}
\,\,\gamma(a)=P,\,\,\gamma(b)=Q\right\}\!,
\end{equation}
where 
\begin{equation}
\text{L}_\gamma:= \int_{a}^b\vert \gamma'(t)\vert_g \, \hbox{d}t.
\end{equation}
is the lenght of the piecewise smooth curve $\gamma:[a,b]\to M$. 


Given a curve $\gamma:[a,b]\to M$, we define its metric derivative (see \cite{Amb-Gi-Sav}) as
\begin{equation}
\label{eq:metric_derivative_def}
\abs{\gamma'}_g(t):= \lim_{h\to 0}\frac{\dist_g\left(\gamma(t+h),\gamma(t)\right)}{\abs{h}},
\end{equation}
provided this limit exists. 
For a curve $\gamma\in H^1(a, \, b; \, M)$,
the metric derivative exists at almost every point 
and belongs to~$L^2(a, \, b)$, as expressed by 
the following
\begin{prop}[Metric derivative]
\label{prop:metric_derivative}
Let $(M,g)$ be a closed oriented $2$-dimensional Riemannian manifold. 
Let~$\gamma\colon [a, \, b]\to M$ be a measurable curve.
Then $\gamma\in H^1(a,b;M)$ if and only if there exists 
$m\in L^2(a,b)$ such that 
\begin{equation}
\label{eq:metricderiv_prop}
\dist_g\left(\gamma(t),\gamma(s)\right)\le \int_{s}^t m(r) \, \d r
\qquad \textrm{for any } s, t \ \textrm{ with } 
a< s \le t< b.
\end{equation}
In this case, the limit 
\eqref{eq:metric_derivative_def} exists for almost any $t\in (a,b)$, the function $t\mapsto \abs{\gamma'}_g(t)$ belongs to $L^2(a,b)$ and is such that 
\[
\abs{\gamma'}_g(t) \le m(t), \,\,\,\,\,\,\hbox{ for a.a. }t\in (a,b)
\]
for each function $m\in L^2(a,b)$ satisfying \eqref{eq:metricderiv_prop}.
Moreover, 
there holds 
\begin{equation}
\label{eq:metric_deriv_eq}
\abs{\gamma'}_g(t) = \abs{\gamma'(t)}_g\,\,\,\,\,\,\hbox{ a.e. in } (a,b). 
\end{equation}
\end{prop}
\begin{proof}
The proof that \eqref{eq:metricderiv_prop} implies that the metric derivative exists for almost any $t$ and belongs to $L^2(a,b)$ is in \cite[Theorem 1.1.2]{Amb-Gi-Sav}. 
It is easy to prove that any~$\gamma\in H^1(a,b;M)$ satisfies~\eqref{eq:metricderiv_prop}
(take $m(t) = \abs{\gamma'(t)}_{g}$ and use \eqref{eq:riem_dist}).
We claim that a measurable curve~$\gamma\colon [a, \, b]\to M$
that satisfies~\eqref{eq:metricderiv_prop} belongs to~$ H^1(a,b;M)$. 
Let~$\gamma\colon [a, \, b]\to M$ be a measurable curve
that satisfies~\eqref{eq:metricderiv_prop}.
We have~$\gamma\in L^2(a, \, b; \, M)$ because~$M$ is compact. 
Let~$h\in(0, \, b-a)$.
Equation~\eqref{eq:metricderiv_prop} and the H\"older inequality imply
\[
 \dist_g(\gamma(t), \, \gamma(t+h)) \leq h \int_{t}^{t+h} m^2(r) \, \d r
\]
for any~$t\in (a, \, b-h)$. By Fubini's theorem, we deduce
\[
 \begin{split}
  \int_a^{b+h} \dist_g^2(\gamma(t), \, \gamma(t+h)) \, \d t
  \leq h \int_a^{b-h} \left(\int_t^{t+h} m^2(r) \, \d r \right) \d t
  = h^2 \int_a^{b} m^2(r) \, \d r
 \end{split}
\]
Therefore, identifying~$M$ with $j(M)\subset \mathbb{R}^\ell$
we deduce that
\[
 \int_a^{b+h} \abs{\gamma(t) - \gamma(t+h)}^2_{\R^\ell} \, \d t
  \leq  h^2 \int_a^{b} m^2(r) \, \d r
\]
for any~$h>0$. This implies~$\gamma\in H^1(a, \, b; \, M)$ \cite{brezis}.

Finally, we show \eqref{eq:metric_deriv_eq}.
Let $\gamma \in H^1(a, \, b; \, M)$. Let~$s\in (a, \, b)$ be
a Lebesgue point for the distributional derivative~$\gamma^\prime$, 
such that the metric derivative $\abs{\gamma'}_g(s)<+\infty$. 
Let $\sigma>0$ be smaller than the injectivity radius of~$M$.
Let $I(s)$ be a neighborhood of~$s$ such that for $t\in I(s)$
there holds $\gamma(t)\in B_\sigma(\gamma(s))$. (Such a neighborhood 
exists because~$\gamma$ is continuous, by Sobolev embeddings.) 
We consider geodesic normal coordinates centered in $\gamma(s)$ 
and we let, for $t\in I(s)$, $\bar{\gamma}(t)\in\R^2$ 
be the representation of $\gamma(t)$ in this coordinates (see Subsection \ref{sssec:normal}). 
In particular, $\bar{\gamma}(s) = (0, \, 0)$.
By definition of geodesic normal coordinates, for any~$t\in I(s)$
there exists a unique geodesic of endpoints~$\gamma(s)$ and~$\gamma(t)$,
which is represented in normal coordinates by a straight line segment of
endpoints~$\bar{\gamma}(s) = (0, \, 0)$ and~$\bar{\gamma}(t)$.
Moreover, by Gauss' lemma, geodesic normal coordinates 
are a radial isometry. Therefore, we have 
\[
\dist_g(\gamma(t), \, \gamma(s)) = \abs{\bar{\gamma}(t)}_{\mathbb{R}^2}.
\] 
Then, 
\[
 \abs{\gamma^\prime}_g(s) 
 = \lim_{t\to s}\frac{\dist_g(\gamma(t), \, \gamma(s))}{\abs{t-s}} 
 = \lim_{t\to s}\frac{\abs{\bar{\gamma}(t)}_{\mathbb{R}^2}}{\abs{t-s}} 
 = \abs{\bar{\gamma}'(s)}_{\mathbb{R}^2}
\]
and, since geodesic normal coordinates centered at~$\gamma(s)$
map isometrically~$\T_{\gamma(s)} M$ into~$\R^2$, 
$\abs{\bar{\gamma}'(s)}_{\mathbb{R}^2} = \abs{\gamma'(s)}_g$.
This completes the proof.
\end{proof}

\subsection{Index and vorticity of a vector field}
\label{ssec:vorticity}

We recall the definitions of index and vorticity of a vector field 
(see also~\cite[Section 1.2]{JerrardIgnat_full}).
To every isolated zero $p\in M$ of a smooth vector field $u$ we can associate an integer number, called the index of~$u$ at~$p$, in the following way. 
Let~$\delta > 0$ be a small parameter, such that 
the ball~$B_\delta(p)$ does not contain any zero of~$u$ 
other than~$p$. Using normal geodesic coordinates,
we may represent~$u$ by a map~$v\colon B_\delta(0)\subseteq\R^2\to\R^2$,
as in~\eqref{uv}. We have~$\abs{v} = \abs{u}_g > 0$ on~$\partial V$
and hence, the map~$v/\abs{v}\colon\partial_\delta(0)\to\S^1$ 
is well-defined and smooth. We define the index of~$u$ at~$p$
as the topological degree of~$v/\abs{v}$ on~$\partial B_\delta(p)$,
\begin{equation}
\label{eq:index1-smooth}
 \ind(u, \, \partial B_\delta(p)) 
 := \deg\left(\frac{v}{\abs{v}}, \, \partial B_\delta(p), \, \S^1\right)
\end{equation}
We have used geodesic normal coordinates for
convenience only; the index is actually
independent of the choice of the coordinate system,
because the degree is invariant by composition with
(orientation-preserving) diffeomorphisms.
If $U\subseteq M$ is an open bounded set, $u$ is a smooth vector field
with finitely many zeros in~$U$ and~$u(x)\neq 0$ for any~$x\in\partial U$,
we define
\begin{equation}
\label{eq:index2}
\text{ind}(u,\partial U):=\sum_{p\in u^{-1}(0)} \ind(u, \, \partial B_\delta(p))
\end{equation}
where~$\delta > 0$ is such that the balls~$B_\delta(p)$ 
(with~$p\in u^{-1}(0)$) are pairwise disjoint.
The topological degree and the index of a vector field
can defined for maps and vector fields with Sobolev or VMO regularity
(see \cite{BN1}, \cite{BN2} and \cite{AGM-VMO}).

There is an equivalent definition of the index,
in terms of the so-called vorticity measure,
which is more convenient for our purposes.
We give the definition directly in terms of vector fields with Sobolev regularity. 
For a vector field $u\in H^1_{\tang}(M)$ (or, more generally, $u\in W^{1,p}_{\tang}(M)\cap L^q_{\tang}(M)$ with $\frac{1}{p}+ \frac{1}{q}=1$) we define the $1$-form
\begin{equation}
\label{eq:current}
j(u) = (Du, iu)_g.
\end{equation}
(The map~$i\colon TM\to TM$ is defined in~\eqref{eq:prop-i-bis}, \eqref{eq:prop-i}.)
We define the vorticity measure as the $2$-form
\begin{equation}
\label{eq:vorticity}
 \omega(u) = \d j(u) + \kappa \Vg,
\end{equation}
where~$\kappa$ is the Gauss curvature of~$M$.
If $U\subset M$ is open with smooth boundary,
$U'$ is a neighbourhood of~$\overline{U}$
and $u\in W^{1,p}_{\tang}(U')\cap L^q_{\tang}(U')$
is such that $\abs{u}_g\ge 1/2$ a.e.~in~a neighbourhood of~$\partial U$, 
we define the index of $u$ on~$\partial U$ as 
\begin{equation}
\label{eq:index1}
\text{ind}(u, \, \partial U):= \frac{1}{2\pi}\left(\int_{\partial U}\frac{j(u)}{\abs{u}^2_g} + \int_{U}\kappa \Vg\right)
\end{equation}
If, additionally, $\abs{u}_g=1$ on~$\partial U$,
then by Stokes' theorem we obtain 
\begin{equation}
\label{eq:index}
 \text{ind}(u, \, \partial U) = \frac{1}{2\pi}\left(\int_{\partial U}j(u)
 + \int_{U}\kappa \Vg\right) = \frac{1}{2\pi}\int_{U}\omega(u).
\end{equation}
If~$u$ is smooth with finitely many zeroes
in~$U$, the index as defined in~\eqref{eq:index1}
agrees with~\eqref{eq:index2}. Indeed, take an orthonormal,
positively oriented tangent frame $\{\tau_1, \, \tau_2\}$
defined in a neighbourhood of~$\partial U$.
(Such a frame exists, see e.g.~\cite[Lemma~6.1]{JerrardIgnat_full}.)
By lifting results, we may write~$u$ as
\[
 u = \rho\left((\cos\theta)\tau_1 + (\sin\theta)\tau_2\right)
 \qquad \textrm{on } \partial U,
\]
where~$\rho := \abs{u}_g$ and~$\theta\colon\partial U \to\R$
is a function that is smooth except for,
at most, one jump discontinuity on each connected 
component of~$\partial U$. We may compute~$j(u)$
in terms of~$\rho$ and~$\theta$, as
\begin{equation} \label{j(u)conn}
 j(u) = \rho^2 \left(\d\theta - \mathcal{A}\right)
\end{equation}
where~$\mathcal{A}$ is the connection~$1$-form of~$\{\tau_1, \, \tau_2\}$, 
defined in~\eqref{eq:connection1}. Equation~\eqref{j(u)conn}
shows that~$\d\theta$ can be extended smoothly to~$\partial U$,
even if~$\theta$ is discontinuous. The index of~$u$, as defined 
in~\eqref{eq:index1}, can be written as
\begin{equation*} 
 \begin{split}
  \text{ind}(u, \, \partial U) = 
  \frac{1}{2\pi}\left(\int_{\partial U}(\d\theta - \mathcal{A})
  + \int_{U}\kappa \Vg\right)
  &= \frac{1}{2\pi}\left(\int_{\partial U}\d\theta 
  + \int_{U}(-\d\mathcal{A} + \kappa \Vg)\right) \\
  &\hspace{-.2cm}\stackrel{\eqref{eq:dconnection}}{=}
  \frac{1}{2\pi} \int_{\partial U}\d\theta 
 \end{split}
\end{equation*}
We write~$\Gamma_1$, \ldots, $\Gamma_p$ for the
connected components of~$\partial U$; each~$\Gamma_k$
is given the orientation induced by~$U$. Then, we deduce
\[
 \text{ind}(u, \, \partial U)
 = \sum_{k=1}^p \deg(e^{i\theta}, \, \Gamma_k, \, \S^1)
\]
and~\eqref{eq:index2} follows, by additivity of the degree.

The definitions~\eqref{eq:current} and~\eqref{eq:vorticity}
apply to Euclidean vector fields as well.
If $u\colon\mathbb{R}^2\to \mathbb{R}^2$ is smooth, then 
\[
\omega(u) = \d j(u) = 2(\det\nabla u) \,\d x^1\wedge \d x^2.
\] 
and thus, if $U\subset \mathbb{R}^2$ open, simply connected,
with Lipschitz boundary and $\abs{u}=1$ on $\partial U$, we have  
\[
\text{ind}(u,\partial U) = \frac{1}{2\pi}\int_{U}\d j(u) = \frac{1}{\pi}\int_{U}\text{det}\nabla u\, \d x = \text{deg}(u, \, \partial U, \, \S^1).
\]

\section{The renormalized energy and its gradient}
\label{sec:gradient_reno}

First, we set some notation and we 
recall some results from~\cite{JerrardIgnat_full}.
Let~$n\in\Z$, $n\geq 1$.
We denote $n$-uples of points in~$M$ as~$\a = (a_1, \, \ldots, \, a_n)$
and $n$-uples of integer numbers as~$\db = (d_1, \, \ldots, \, d_n)$.
We restrict our attention to the pairs~$(\a, \, \db)$ in the `admissible class'
\begin{equation} \label{admissible}
 \mathscr{A}^n := \left\{(\a, \, \db)\in M^n\times\Z^n\colon
 a_j \neq a_k \ \textrm{ for any } j\neq k \ \textrm{ and } \
 \sum_{j=1}^n d_j = \chi(M)\right\}
\end{equation}
For any~$(\a, \, \db)\in\mathscr{A}^n$, 
we consider the (unique) $2$-form $\Psi = \Psi(\a, \, \db)$ that satisfies
\begin{equation} \label{Phiad}
 \begin{cases}
  -\Delta \Psi = 2\pi \displaystyle\sum_{j=1}^n d_j\,\delta_{\a_j}
   - \kappa\Vg \\
  \displaystyle\int_M \Psi = 0
 \end{cases}
\end{equation}
Here~$\kappa$ is the Gauss curvature of~$M$. 
The right-hand side of the equation for~$\Delta\Psi$
has zero average, due to the Gauss-Bonnet theorem
and the assumption that~$\sum_{j=1}^n d_j = \chi(M)$.
As a consequence, \eqref{Phiad} has a unique solution.
Moreover, elliptic regularity theory implies that 
$\Psi(\a, \, \db)\in W^{1,p}_{\tan}(M)$ 
for any~$p\in(1, \, 2)$. Let
\[
 \gf := 1 - \frac{\chi(M)}{2}
\]
be the genus of~$M$, and let~$\Harm^1(M)$ be the space
of harmonic~$1$-forms on~$M$. By Hodge theory, $\Harm^1(M)$
is a real vector space of dimension~$2\gf$. We equip~$\Harm^1(M)$
with the~$L^2$ norm (but any other norm would do, since~$\Harm^1(M)$
is finite-dimensional).

\begin{definition} \label{def:canonical_harmonic}
 Let~$(\a, \, \db)\in\mathscr{A}^n$
 and~$\xi\in\Harm^1(M)$. We say that a vector field~$u_*$ 
 is a \emph{canonical harmonic field} for~$(\a, \, \db, \, \xi)$
 if $u_*\in W^{1,1}_{\tan}(M)$, $\abs{u_*(x)} = 1$ for a.e.~$x\in M$ and
 \begin{equation} \label{canonicalvf}
  j(u_*) = \d^*\Psi(\a, \, \db) + \xi
 \end{equation}
 where~$\Psi(\a, \, \db)$ is defined by~\eqref{Phiad}. 
\end{definition}

\begin{remark} \label{rk:canonical_harmonic}
 A vector field~$u_*\in W^{1,1}_{\tan}(M)$
 is canonical harmonic (for some~$(\a, \, \db, \, \xi)$)
 if and only if $\abs{u_*}= 1$ a.e.~in~$M$,
 \[
  \d j(u_*) = 2\pi\sum_{j=1}^n d_j \delta_{a_j} -\kappa\Vg ,
  \qquad \d^* j(u_*) = 0.
 \]
\end{remark}

Not any triple~$(\a, \, \db, \, \xi)$ admits a canonical harmonic field.
Ignat and Jerrard~\cite[Theorem~2.1]{JerrardIgnat_full}
characterised the values of~$(\a, \, \db, \, \xi)$
that admit a canonical harmonic field; we recall their result. 
We say that a set~$L\subseteq\Harm^1(M)$
is an affine lattice if there exists a linear, invertible
map~$\alpha\colon\Harm^1(M)\to\R^{2\gf}$ and a vector~$b\in\R^{2\gf}$
such that~$\alpha(L) = \Z^{2\gf} + b$.
 
\begin{theorem}[{\cite{JerrardIgnat_full}}] \label{th:canonical}
 For any~$(\a, \, \db)\in\mathscr{A}^n$, there exists 
 an affine lattice~$\mathcal{L}(\a, \, \db)\subseteq\Harm^1(M)$
 such that a canonical harmonic field for~$(\a, \, \db)$ exists
 if and only if
 \begin{equation*} 
  \xi\in\mathcal{L}(\a, \, \db)
 \end{equation*}
 Moreover,
 \begin{itemize}
  \item[(i)] the canonical harmonic field for~$(\a, \, \db, \, \xi)$
  is essentially unique: any two canonical
  harmonic fields~$u_*$, $u^*$ for~$(\a, \, \db, \, \xi)$
  satisfy $u^* = (\cos\beta) u_* + i(\sin\beta)u_*$ for some 
  constant~$\beta\in\R$;
  
  \item[(ii)] any canonical harmonic field 
  for~$(\a, \, \db, \, \xi)$ belongs 
  to~$W^{1,p}_{\tan}(M)$ for any~$p\in (1, \, 2)$ and is smooth 
  in~$M \setminus \{a_1, \, \ldots, \, a_n\}$.
 \end{itemize}
\end{theorem}

Actually, Ignat and Jerrard's result is stronger: it provides
a characterisation of~$\mathcal{L}(\a, \, \db)$ 
in terms of~$\Psi(\a, \, \db)$
(we will recall it later on, in Section~\ref{sect:gradrenog})
and shows that~$\mathcal{L}(\a, \, \db)$
depends continuously on~$(\a, \, \db)$, in a suitable sense.

Given~$(\a, \, \db)\in\mathscr{A}^n$
and~$\xi\in\mathcal{L}(\a, \, \db)$, let~$u_*$
be the (essentially unique) canonical harmonic field for~$(\a, \, \db, \, \xi)$.
We define the (intrinsic) renormalized energy of~$(\a, \, \db, \, \xi)$ as
\begin{equation} \label{intrinsicren}
 W(\a, \, \db, \, \xi) := \lim_{\rho\to 0}
  \left( \frac{1}{2}\int_{M\setminus\cup_{j=1}^n B_\rho(a_j)} 
  \abs{D u_*}^2_g \Vg - \pi\sum_{j=1}^n d_j^2 \abs{\log\rho} \right) 
\end{equation}
It can be proved that the limit in~\eqref{intrinsicren} exists and is finite.
What is more, Ignat and Jerrard~\cite[Proposition~2.4]{JerrardIgnat_full}
gave a characterisation of~$W$ in terms of the Green function
for the Laplace-Beltrami operator on~$M$, thus extending
earlier results by Bethuel, Brezis and H\'elein~\cite{BBH}
in the Euclidean setting. 
Let~$\mathrm{Vol}(M) := \int_M\Vg$. Let
\[
 G\colon \{(x, \, y)\in M\times M\colon x\neq y\}\to\R
\]
be the Green function for the Laplace-Beltrami operator, 
i.e. the unique function such that
\begin{equation} \label{Green}
 \begin{cases} 
  -\Delta_x (G(x, \, y) \, \Vg) = \delta_y - \dfrac{\Vg}{\mathrm{Vol}(M)} 
    & \textrm{in } \mathscr{D}^\prime(M) \\[7pt]
  \displaystyle\int_M G(x, \, y) \Vg(x) = 0
 \end{cases}
\end{equation}
for any~$y\in M$. The function~$G$ may be decomposed as
\begin{equation} \label{Green-splitting}
 G(x, \, y) = G_0(x, \, y) + H(x, \, y)
\end{equation}
where~$H\in C^1(M\times M)$ and
\begin{equation} \label{Green-splitting-log}
  G_0(x, \, y) = -\frac{1}{2\pi} \log \dist_g(x, \, y)
 \qquad \textrm{if } \dist_g(x, \, y) \textrm{ is small enough}
\end{equation}
(see e.g.~\cite[Chapter~4.2]{aubin98}).
We also define the function~$\psi_0\colon M\to\R$ by
\begin{equation} \label{geompot}
 \begin{cases}
  -\Delta \psi_0 = - \kappa + \dfrac{2\pi \, \chi(M)}{\mathrm{Vol}(M)}  \\
  \displaystyle\int_M \psi_0 \Vg = 0
 \end{cases}
\end{equation}
By the Gauss-Bonnet theorem, the right-hand side of this equation
has zero average, so the solution exists and is unique.
By comparing~\eqref{Phiad}, \eqref{Green} and~\eqref{geompot},
we deduce
\begin{equation} \label{Phiad-Green}
 \Psi(\a, \, \db) 
 = 2\pi\sum_{j=1}^n d_j \, G(\cdot, \, a_j)\Vg + \psi_0\Vg
\end{equation}
Ignat and Jerrard~\cite[Proposition 2.4]{JerrardIgnat_full} proved 
the following
\begin{prop} \label{prop:renormalized}
 For any~$(\a, \, \db)\in\mathscr{A}^n$ and 
 any~$\xi\in\mathcal{L}(\a, \, \db)$, there holds
 \begin{align*}
 \label{eq:reno_ij}
  W(\a, \, \db, \, \xi)
  &= 4\pi^2\sum_{1\le j < k \le n} d_j \, d_k \, G(a_j, \, a_k)
   + 2\pi\sum_{j=1}^n \left( \pi d_j^2 H(a_j, \, a_j) 
   + d_j \psi_0(a_j) \right) \nonumber\\
  &\qquad\qquad\qquad + \frac{1}{2}\int_{M}\abs{\xi}^2_g \Vg
   + \frac{1}{2}\int_{M}\abs{\d\psi_0}^2_g \Vg 
 \end{align*}
\end{prop}

The goal of this section is to compute the gradient of the
renormalized energy~$W$ with respect to~$\a$. However,
the very definition of the gradient requires some care,
because the set~$\mathcal{L}(\a, \, \db)$ of admissible values
for~$\xi$ depends on~$(\a, \, \db)$, so the variables
$\a$, $\db$ and~$\xi$ are not independent.
Therefore, our first task is to make sure that the 
gradient~$\nabla_\a W$ is well-defined.

\begin{lemma} \label{lemma:bundle}
 For any~$(\a, \, \db)\in\mathscr{A}^n$ and
 any~$\xi\in\mathcal{L}(\a, \, \db)$, there exists
 an open neighbourhood~$U$ of~$\a$ in~$M^n$ and a unique
 smooth map~$\Xi\colon U\to\Harm^1(M)$ such that
 $\Xi(\a) = \xi$ and
 \[
  \Xi(\b) \in \mathcal{L}(\b, \, \db)
 \]
 for any~$\b\in U$.
\end{lemma}


Lemma~\ref{lemma:bundle} allows us to interpret the 
gradient~$\nabla_\a W$ with respect to~$\a$ 
in a non-ambiguous way: for any~$(\a, \, \db)\in\mathscr{A}^n$ and 
any~$\xi\in\mathcal{L}(\a, \, \db)$, we define
\begin{equation} \label{def:gradW}
 \nabla_{\a} W(\a, \, \db, \, \xi) 
 := \nabla_{\b} \, W(\b, \, \db, \, \Xi(\b))_{|\b=\a}
\end{equation}
where~$\Xi$ is the map given by Lemma~\ref{lemma:bundle}.
The gradient~$\nabla_\a W$ is well-defined now,
because Proposition~\ref{prop:renormalized} and Lemma~\ref{lemma:bundle}
imply that the map~$\b\mapsto W(\b, \, \db, \, \Xi(\b))$ 
is differentiable (and even smooth). 
As it turns out, $\nabla_\a W$ can characterised
in terms of canonical harmonic fields. 
Given $(\a, \, \db)\in\mathscr{A}^n$, $\a = (a_1, \, \ldots, \, a_n)$,
we consider geodesic balls~$B_\eta(a_j)$ of center~$a_j$
and radius~$\eta >0$ small enough, so that
the closed balls $\bar{B}_\eta(a_j)$ are pairwise disjoint. 
We denote by~$\nu$ the exterior unit normal to~$\partial B_\eta(a_j)$.

\begin{prop} \label{prop:gradient_reno}
 Let~$(\a, \, \db)\in\mathscr{A}^n$ and~$\xi\in\mathcal{L}(\a, \, \db)$.
 Let~$u_*\in W^{1,1}_{\tan}(M)$ be the (essentially unique)
 canonical harmonic field for~$(\a, \, \db, \, \xi)$.
 Let~$j\in\{1, \, \ldots, \, n\}$ be fixed, and 
 let~$e$ be a smooth vector field defined in a neighbourhood~$U$
 of~$a_j$, such that
 \begin{equation} 
  \label{eq:hypdive}
  \div e(a_j) = 0 \qquad \textrm{in } U.
 \end{equation}
 Then, we have
 \begin{equation*} 
  \begin{split}
   \lim_{\eta\to 0} \int_{\partial B_\eta(a_j)} 
    \left( (D_{e} u_*, \, D_{\nu} u_*)_g
    - \frac{1}{2} \abs{D u_*}^2_g (\nu, \, e)_g \right)\d\H^1
   = \left(\nabla_{a_j} W(\a, \, \db, \, \xi), e(a_j)\right)_g
  \end{split}
 \end{equation*}
\end{prop}

In the Euclidean setting, i.e.~when~$M$ is a domain in~$\R^2$,
an even stronger statement holds: if~$e\in\R^2$ is a constant, then
\begin{equation*} 
  \begin{split}
   \int_{\partial B_\eta(a_j)} 
    \left( (D_{e} u_*, \, D_{\nu} u_*)
    - \frac{1}{2} \abs{D u_*}^2 (\nu, \, e) \right)\d\H^1
   = \left(\nabla_{a_j} W(\a, \, \db), \, e\right)
  \end{split}
\end{equation*}
for any~$\eta > 0$ small enough (see e.g.~\cite[Theorem~VII.4, Theorem~VIII.3]{BBH} 
and~\cite[Theorem~5.1]{lin1}). In our setting, we need to take the
limit as~$\eta\to 0$ to compensate for curvature effects.

The rest of this section is devoted to the proof of
Proposition~\ref{prop:gradient_reno}.
We distinguish two cases, according to the genus of~$M$.

\subsection{The case~$\gf=0$}

First, we suppose that~$M$ has genus~$\gf=0$,
that is, $M$ is diffeomorphic to a sphere
(although the Riemannian metric on~$M$ is arbitrary).
In this case, the problem simplifies
because~$\Harm^1(M) = 0$, so~$W$ is a function 
of~$(\a, \, \db)$ only, $W = W(\a, \, \db)$.
Let us fix~$(\a, \, \db)\in\mathscr{A}^n$.
Let~$\Psi(\a, \, \db)$ be the $2$-form defined by~\eqref{Phiad}.
For any~$j\in\{1, \, \ldots, \, n\}$, we define the auxiliary 
function~$S_j\colon M\to\R$ as
\begin{equation} \label{S_j}
 S_j := \star\Psi(\a, \, \db) - 2\pi d_j G(\cdot, \, a_j)
 \stackrel{\eqref{Phiad-Green}}{=}
 2\pi\sum_{k\neq j} d_k G(\cdot, \, a_k) + \psi_0
\end{equation}
The function~$S_j$ is smooth in a neighbourhood of~$a_j$
and hence, it may be interpreted as the `regular part'
of~$\star\Psi(\a, \, \db)$ near~$a_j$.

By Equation~\eqref{Green-splitting}, we know that
the Green function~$G$ may be decomposed into
a `singular part', which blows up logarithmically
near the diagonal of~$M\times M$, and a `regular part'~$H\in C^1(M\times M)$.
Since~$G$ is symmetric in~$x$ and~$y$
(see e.g.~\cite[Theorem 4.13]{aubin98}), we deduce
\[
 H(x, \, y) = H(y, \, x)
\]
for any $x\in M$, $y\in M$ such that $\dist_g(x, \, y)$ is small enough.
As a consequence, in case~$x = y$, the gradients of~$H$
with respect to each argument are equal:
\begin{equation} \label{symmH}
 \nabla_x H(x, \, x) = \nabla_y H(x, \, x)
\end{equation}
We will write~$\nabla H(x, \, x) := \nabla_x H(x, \, x)
= \nabla_y H(x, \, x)$.

\begin{lemma} \label{lemma:gradrenoSH}
 For any~$j\in\{1, \, \ldots, \, n\}$, we have
 \[
  \begin{split}
   \nabla_{a_j} W(\a, \, \db)
   = 2\pi d_j \left(\nabla S_j(a_j) + 2\pi d_j \, \nabla H(a_j, \, a_j) \right)
  \end{split}
 \]
\end{lemma}
\begin{proof}
 Let~$j$ be fixed, and let~$\a^* := (a_k)_{k\neq j}$.
 By Proposition~\ref{prop:renormalized}, we may write 
 the renormalized energy as
 \begin{equation*} 
   W(\a, \, \db)
   = 4\pi^2 \, d_j \sum_{k\neq j} d_k \, G(a_j, \, a_k)
   + 2\pi d_j \, \psi_0(a_j) + 2\pi^2 d_j^2 \, H(a_j, \, a_j) + C(\a^*),
 \end{equation*}
 where~$C(\a^*)$ is a suitable
 function of~$\a^*$, which does not depend on~$a_j$.
 Using the definition of~$S_j$, i.e.~Equation~\eqref{S_j}, we may also write
 \begin{equation} \label{gradrenoSH1}
   W(\a, \, \db)
   = 2\pi d_j \left( S_j(a_j) + \pi d_j \, H(a_j, \, a_j)\right)
   + C(\a^*)
 \end{equation}
 It only remains to differentiate~\eqref{gradrenoSH1}
 with respect to~$a_j$. For any~$x\in M$, we have
 \[
  \nabla_x\left(H(x, \, x)\right)
  = \nabla_x H(x, \, x) + \nabla_y H(x, \, y)
  \stackrel{\eqref{symmH}}{=} 2\nabla H(x, \, x)
 \]
 and hence, the lemma follows.
\end{proof}

\begin{proof}[Proof of Proposition~\ref{prop:gradient_reno}]
Let~$e$ be a smooth, divergence-free vector field, defined in 
a neighbourhood of~$a_j$. Let
\[
 I_\eta(u_*)(a_j)
 := \int_{\partial B_\eta(a_j)}(D_{\nu} u_*, \, D_{e}u_*)_g \, \d\H^1
 - \frac{1}{2} \int_{\partial B_\eta(a_j)}\abs{D u_*}^2_g (\nu, \, e)_g\, \d\H^1
\]
The definition~\eqref{eq:current} of~$j(u_*)$
implies, via standard computations, that
$\abs{D u_*}_g^2= \abs{j(u_*)}_g^2$,  $j(u_*) = \d^*\Psi$ and 
$D_v u_* = j(u_*)(v) \, iu_*$ for any smooth vector field $v$. 
Moreover, $j(u_*) = \d\Psi(\a, \, \db)$ by 
Definition~\ref{def:canonical_harmonic}.
For simplicity, we write~$\Psi$ instead of~$\Psi(\a, \, \db)$ from now on.
Then, we have
\begin{equation} \label{eq:gradient_1}
 \begin{split}
  I_\eta(u_*)(a_j)
  = \int_{\partial B_\eta(a_j)}\d^*\Psi(\nu) \, \d^*\Psi(e) \, \d\H^1 
  - \frac{1}{2}\int_{\partial B_\eta(a_j)}\abs{\d^*\Psi}^2_g  (\nu, \, e)_g \, \d\H^1
 \end{split}
\end{equation} 
Now, 
\[
 \d^*\Psi = -\star \d(\star \Psi)
 \stackrel{\eqref{S_j}}{=} -\star \d(S_j + 2\pi d_j G(\cdot, a_j))
\]
and thus
\begin{equation} \label{eq:normdstarphi}
 \begin{split}
 \vert \d^*\Psi \vert^2_g
 & = \vert \d(S_j + 2\pi d_j G(\cdot,a_j))\vert^2_g 
 = \vert \nabla(S_j + 2\pi d_j G(\cdot,a_j))\vert^2_g \\
 & = \vert \nabla S_j\vert^2_g 
 + 4\pi^2 d_j^2 \, \vert \nabla G(\cdot,a_j)\vert^2_g
 + 4\pi d_j \, (\nabla S_j, \, \nabla G(\cdot,a_j))_g
 \end{split}
\end{equation} 
On the other hand, for any vector field~$X$ we have
\begin{equation*}
 \begin{split}
  \d^*\Psi(X) = -\star \d(S_j + 2\pi d_j G(\cdot, a_j))\,(X)
  &\stackrel{\eqref{eq:grad_ruotato}}{=} 
   -(i\nabla(S_j + 2\pi d_j G(\cdot, a_j)), \, X) \\
  &\hspace{.25cm}= (\nabla(S_j + 2\pi d_j G(\cdot, a_j)), \, iX)
 \end{split}
\end{equation*}
because in the tangent plane at each point, $i$ is a rotation
and its adjoint is equal to~$i^{-1} = -i$. As a consequence,
\begin{equation}\label{eq:prod_dstarphi}
 \begin{split}
 \d^*\Psi(\nu) \, \d^*\Psi(e)
 &= \nabla_{i\nu}(S_j + 2\pi d_j G(x,a_j)) 
  \, \nabla_{i e}(S_j + 2\pi d_j G(x,a_j)) \\
 &=  \nabla_{i\nu}S_j(\cdot) \, \nabla_{ie}S_j(\cdot)
  + 2\pi d_j \, \nabla_{i\nu}S_j(\cdot) \, \nabla_{ie} G(\cdot,a_j) \\
 &\qquad + 2\pi d_j \, \nabla_{i\nu} G(\cdot,a_j) \, \nabla_{i e}S_j(\cdot)
 + 4\pi^2 d_j^2 \, \nabla_{i\nu} G(\cdot,a_j) \, \nabla_{ie} G(\cdot,a_j)
 \end{split}
\end{equation}
We write the Green function~$G$ in the form~\eqref{Green-splitting}.
For ease of notation, we let~$L_j := \log \dist_g(\cdot, \, a_j)$
and~$H_j := H(\cdot, \, a_j)$, so that
\[
 2\pi d_j G(\cdot, \, a_j) = -d_j L_j + 2\pi d_j H_j
\]
The equations~\eqref{eq:normdstarphi} and~\eqref{eq:prod_dstarphi}
imply respectively
\begin{equation} \label{eq:normadstarphitris}
 \begin{split}
  \vert \d^*\Psi \vert^2_g
  &= \abs{\nabla S_j}^2_g + d_j^2 \abs{\nabla L_j}^2_g
   + 4\pi^2 d_j^2 \abs{\nabla H_j}^2_g \\
  &\qquad - 4\pi d_j^2 \, (\nabla L_j, \, \nabla H_j)_g
   - 2 d_j \, (\nabla S_j, \, \nabla L_j)_g
   + 4\pi d_j \, (\nabla S_j, \, \nabla H_j)_g  
 \end{split}
\end{equation}
and
\begin{equation} \label{eq:proddstarphitris}
 \begin{split}
  \d^*\Psi(\nu) \, \d^*\Psi(e) 
  & = \nabla_{i\nu}S_j \, \nabla_{ie}S_j
   - d_j \, \nabla_{i\nu}S_j \, \nabla_{ie} L_j
  + 2\pi d_j \, \nabla_{i\nu} S_j \, \nabla_{ie} H_j \\
  &\qquad - d_j \, \nabla_{i\nu} L_j \, \nabla_{ie} S_j
   +2\pi d_j \, \nabla_{i\nu} H_j \, \nabla_{ie} S_j
  + d_j^2 \, \nabla_{i\nu} L_j \, \nabla_{ie} L_j \\
  &\qquad -2\pi d_j^2 \, \nabla_{i\nu} L_j \, \nabla_{ie} H_j
   -2\pi d_j^2 \, \nabla_{i\nu} H_j \, \nabla_{ie} L_j 
   + 4\pi^2 d_j^2 \, \nabla_{i\nu} H_j \, \nabla_{ie} H_j
 \end{split}
\end{equation}
The function~$H$ is of class~$C^1$ in~$M\times M$
and, by Equation~\eqref{S_j}, $S_j$ is smooth 
in a neighbourhood of~$a_j$. In particular, both~$\nabla H_j$
and~$\nabla S_j$ are locally bounded around~$a_j$.
Therefore, we can write Equations~\eqref{eq:normadstarphitris}
and~\eqref{eq:normadstarphitris} as
\begin{equation} \label{normadstar4} 
 \abs{\d^*\Psi}^2_g = d_j^2 \abs{\nabla L_j}^2_g 
  - 2d_j(\nabla (S_j+2\pi d_j H_j), \, \nabla L_j)_g 
  + \mathrm{O}(1) 
\end{equation}
and
\begin{equation} \label{normadstarphi4}
 \begin{split}
  \d^*\Psi(\nu) \, \d^*\Psi(e) 
  = d_j^2 \, \nabla_{i\nu} L_j \, \nabla_{ie} L_j
   &- d_j \, \nabla_{i\nu}(S_j + 2\pi d_j H_j) \, \nabla_{ie} L_j \\
   &- d_j \, \nabla_{i\nu} L_j \, \nabla_{ie} (S_j + 2\pi d_j H_j)
   + \mathrm{O}(1) 
 \end{split}
\end{equation}
respectively. Gauss' lemma implies that
\begin{equation} \label{logdist}
 \nabla \left(\log \dist_g(\cdot, \, a_j)\right)
 = \frac{\nabla \dist_g(\cdot, \, a_j)}{\dist_g(\cdot, \, a_j)}
 = \frac{\nu}{\dist_g(\cdot, \, a_j)} \qquad 
 \textrm{in } B_\eta(a_j)\setminus\{a_j\}
\end{equation}
so on the circle~$\partial B_\eta(a_j)$ we have
$\nabla L_j = \nu/\eta$, $\nabla_{i\nu} L_j = 0$ and 
\begin{align}
 \abs{\d^*\Psi}^2 &= \frac{d_j^2}{\eta^2}
  - \frac{2d_j}{\eta} \nabla_\nu (S_j+2\pi d_j H_j) 
  + \mathrm{O}(1) \label{normadstar5} \\
 \d^*\Psi(\nu) \, \d^*\Psi(e) 
  &= - \frac{d_j}{\eta} \nabla_{i\nu}(S_j + 2\pi d_j H_j) (\nu, \, ie)
  + \mathrm{O}(1) \label{normadstarphi5}
\end{align}
By substituting~\eqref{normadstar5}, \eqref{normadstarphi5}
in~\eqref{eq:gradient_1}, we obtain
\begin{equation} \label{eq:gradient_2}
 \begin{split}
  I_\eta(u_*)(a_j)
  = &-\frac{d_j}{\eta} \int_{\partial B_\eta(a_j)} 
   \nabla_{i\nu} (S_j+2\pi d_j H_j) \, (\nu, \, ie)_g \, \d\H^1 \\
  &+ \frac{d_j}{\eta} \int_{\partial B_\eta(a_j)} 
   \nabla_\nu (S_j+2\pi d_j H_j) \, (\nu, \, e)_g \, \d\H^1 \\
  &\qquad\qquad\qquad - \frac{d_j^2}{2\eta^2} 
   \int_{\partial B_\eta(a_j)} (\nu, \, e)_g \, \d\H^1
   + \mathrm{O}(\eta) 
 \end{split}
\end{equation} 
The properties of~$i$ imply $(\nu, \, ie)_g = -(i\nu, \, e)_g$ and thus
\begin{equation} \label{eq:gradient_4}
 \begin{split}
  &- \nabla_{i\nu}(S_j +2\pi d_j H_j) \, (\nu, \, ie)_g
   + \nabla_\nu (S_j + 2\pi d_j H_j) \, (\nu, \, e)_g  \\
  &\qquad = \big(\nabla_{\nu}(S_j + 2\pi d_j H_j)\nu, \, e\big)_g 
   + \big(\nabla_{i\nu}(S_j + 2\pi d_j H_j) i\nu, \, e\big) _g
  = (\nabla(S_j + 2\pi d_j H_j), \, e)_g.
 \end{split}
\end{equation}
Moreover, by integrating by parts in the last term of \eqref{eq:gradient_2} (see~\eqref{eq:ip_div}),
we obtain
\begin{equation} \label{eq:gradient_3}
 \begin{split}
  \int_{\partial B_\eta(a_j)} (\nu, \, e)_g \, \d\H^1
  = \int_{B_\eta(a_j)} (\div e) \, \Vg = \O(\eta^3),
 \end{split}
\end{equation} 
because $\abs{\div e}= \O(\eta)$ in $B_\eta(a_j)$ as $\div e(a_j)= 0$ (recall \eqref{eq:hypdive}) and $\div e$ is smooth.
Combining~\eqref{eq:gradient_2}, \eqref{eq:gradient_3}
and~\eqref{eq:gradient_4}, we deduce
\begin{equation} \label{eq:gradient_5}
 \begin{split}
  I_\eta(u_*)(a_j)
  &= \frac{d_j}{\eta} \int_{\partial B_\eta(a_j)} 
   \nabla_{e} (S_j+2\pi d_j H_j) \, \d\H^1 + \mathrm{O}(\eta)
 \end{split}
\end{equation} 
and, since $\mathcal{H}^1(\partial B_\eta(a_j))= 2\pi\eta + \mathrm{O}(\eta^2)$
(see e.g.~\cite[Proposition 10]{spivak}), we conclude that
\begin{equation}
 \lim_{\eta\to 0} I_{\eta}(u_*)(a_j) 
 = 2\pi d_j\left(\nabla S_j(a_j) 
 + 2\pi d_j \nabla H(a_j, \, a_j), \, e(a_j)\right)_g
\end{equation}
The result follows by Lemma~\ref{lemma:gradrenoSH}.
\end{proof}

\subsection{The case~$\gf>0$}
\label{sect:gradrenog}

When~$\gf>0$, the very definition of~$\nabla_\a W$ requires some care, 
as the variables~$(\a, \, \db, \, \xi)$ are not independent
--- they are related to each other by the constraint
$\xi\in\mathcal{L}(\a, \, \db)$. Our first task is to prove
Lemma~\ref{lemma:bundle}, which allows us to define~$\nabla_\a W$
unambigously. First of all, we recall some results 
from~\cite{JerrardIgnat_full} which characterise~$\mathcal{L}(\a, \, \db)$.
We choose closed, simple, geodesic 
curves~$\gamma_1$, \ldots, $\gamma_{2\gf}$ in~$M$
whose homology classes generate the first homology group~$H_1(M; \, \Z)$.
(Such curves exist; see e.g.~\cite[Lemma~5.2]{JerrardIgnat_full}).
Then, the map
\begin{equation} \label{alpha}
 \alpha\colon\Harm^1(M)\to\R^{2\gf}, \qquad 
 \alpha_k(\xi) := \int_{\gamma_k} \xi \quad
 \textrm{for } k\in\{1, \, \ldots, \, 2\gf\} 
\end{equation}
is linear and invertible, essentially becuase of 
Hodge theory and de Rham's theorem
(for a detailed argument, see~\cite[Lemma~5.2]{JerrardIgnat_full}).
Let~$(\a, \, \db)\in\mathscr{A}^n$. Suppose first that for any~$j$ and~$k$, 
the curve~$\gamma_k$ does not contain the point~$a_j$. 
For any~$k$, we choose an orthogonal tangent frame~$\{\tau_{1,k}, \, \tau_{2,k}\}$
defined in a neighbourhood of~$\gamma_k$.
Let~$\mathcal{A}_k$ be the connection~$1$-form induced 
by~$\{\tau_{1,k}, \, \tau_{2,k}\}$, as in~\eqref{eq:connection1}. We define
\begin{equation} \label{zeta}
 \zeta_k(\a, \, \db) := \int_{\gamma_k} \left(\d^*\Psi(\a, \, \db) + \mathcal{A}_k\right)
 \qquad \textrm{for } k\in\{1, \, \ldots, \, 2\gf\},
\end{equation}
where~$\Psi(\a, \, \db)$ is the $2$-form defined by~\eqref{Phiad}. 
The number~$\zeta_k(\a, \, \db)$ depends on~$\mathcal{A}_k$, and hence 
on the choice of the frame. However, if~$\mathcal{A}_k^\prime$ is another 
connection arising from a different frame, then
\[
 \int_{\gamma_k} \mathcal{A}_k \equiv \int_{\gamma_k} \mathcal{A}_k^\prime
 \qquad \mod 2\pi
\]
(see~\cite[Lemma~6.2]{JerrardIgnat_full}).
Therefore, $\zeta$ is well-defined (and, as it turns out, continuous) 
as a map~$\zeta\colon\mathscr{A}^n\to(\R/2\pi\Z)^{2\gf}$.
If for some~$j$ and~$k$ the curve~$\gamma_k$ contains~$a_j$, then we consider
a small perturbation~$\lambda_k$ of~$\gamma_k$, which is still
a smooth, closed, simple curve, is homologous to~$\gamma_k$
but avoids all the points~$a_j$. We define~$\zeta_k(\a, \, \db)$
by integrating over~$\lambda_k$ instead of~$\gamma_k$. 
The set~$\mathcal{L}(\a, \, \db)$ in Theorem~\ref{th:canonical}
is defined in terms of~$\alpha$ and~$\mathcal{L}(\a, \, \db)$,
as follows: 
\begin{equation} \label{Lad}
 \mathcal{L}(\a, \, \db) := \left\{ 
 \xi\in\Harm^1(M)\colon \alpha(\xi) 
  + \zeta(\a, \, \db) \in (2\pi\Z)^{2\gf} \right\}
\end{equation}
Now, we proceed towards the proof of Lemma~\ref{lemma:bundle}.
For any~$a\in M$ and any~$v\in\T_a M$, we consider the function
$\sigma(\cdot, \, a, \, v)\colon M\setminus\{a\}\to\R$,
\begin{equation} \label{sigma_av}
  \sigma(x, \, a, \, v) := \left(\nabla_a G(x, \, a), \, v\right)
  \qquad \textrm{for } x\in M\setminus\{a\}
\end{equation}
(where~$\nabla_a G$ denotes the gradient of the Green function
with respect to its second argument). The function
$\sigma(\cdot, \, a, \, v)$ is harmonic, and hence smooth, 
in~$M\setminus\{a\}$. Indeed, let~$B_\eta(a)$ be a small 
geodesic disk centered at~$a$, and 
let~$\gamma\colon(-\delta, \, \delta)\to M$ be a smooth curve
with~$\gamma(0) = a$, $\gamma^\prime(0) = v$. Since the Green 
function~$G$ is smooth away from the diagonal of~$M\times M$, 
for any~$\eta > 0$ we have
\begin{equation} \label{Green-conv}
 \frac{G(\cdot, \, \gamma(t)) - G(\cdot, \, a)}{t} \to \sigma(\cdot, \, a, \, v)
 \qquad \textrm{in } C^1(M\setminus B_\eta(a)) \ \textrm{ as } t\to 0
\end{equation}
However, the function~$G(\cdot, \, \gamma(t)) - G(\cdot, \, a)$
is harmonic in~$M\setminus\bar{B}_\eta(a)$ for~$t$ small enough.
By taking the (distributional) Laplacian in both sides
of~\eqref{Green-conv}, we deduce that~$\sigma(\cdot, \, a, \, v)$
is harmonic in~$M\setminus\bar{B}_\eta(a)$.


\begin{lemma} \label{lemma:zeta}
 For any~$(\a, \, \db)\in\mathscr{A}^n$, there exists
 an open neighbourhood~$V$ of~$\a$ in~$M^n$ such that
 the map~$V\to (\R/2\pi\Z)^{2\gf}$, $\b\mapsto \zeta(\b, \, \db)$
 has a smooth lifting~$V\to\R^{2\gf}$, which we still 
 denote~$\zeta(\cdot, \, \db)$
 by abuse of notation. Moreover, for any~$j\in\{1, \, \ldots, \, n\}$,
 any~$k\in\{1, \, \ldots, \, 2\gf\}$
 and any~$v\in T_{a_j} M$, there holds
 \[
  \left(\nabla_{a_j}\zeta_k(\a, \, \db), \, v\right)_g
  = - 2\pi d_j \, \int_{\gamma_k} \star\d\sigma(\cdot, \, a_j, \, v)
 \]
 where~$\sigma(\cdot, \, a_j, \, v)$ is defined by~\eqref{sigma_av}.
\end{lemma}
\begin{proof}
 Let~$(\a, \, \db)\in\mathscr{A}^n$ be fixed. As we observed before,
 up to a perturbation we may assume that for each~$k$ and~$j$,
 the curve~$\gamma_k$ does not contain~$a_j$. 
 We take a small open neighbourhood~$V$
 of~$\a$ in~$M^n$ such that~$V\times\{\db\}\subseteq\mathscr{A}^n$
 --- that is, for any~$\b = (b_1, \, \ldots, \, b_n)\in V$,
 the points~$b_j$ are distinct; this is possible because 
 the points~$a_j$ themselves are distinct. By taking~$V$ small enough,
 we can also make sure that for any~$k$, $j$ and any~$\b\in V$,
 $\gamma_k$ does not contain~$b_j$. 
 For each~$k$, we \emph{fix} an orthonormal tangent frame
 defined in a neighbourhood of~$\gamma_k$ and let~$A_k$
 be the corresponding connection~$1$-form.
 Then, Equation~\eqref{zeta} defines unambigously a map 
 $\b\in V\mapsto\zeta(\b, \, \db)\in\R^{2\gf}$. 
 We claim that this map is smooth. By writing~$\Psi(\a, \, \db)$
 in terms of the Green function, as in~\eqref{Phiad-Green},
 and observing that~$\d^*(f\Vg) = -\star\d\star(f\Vg) = -\star\d f$
 for any function~$f$, we obtain
 \begin{equation} \label{zeta-Green}
  \zeta_k(\b, \, \db) = - 2\pi \sum_{j=1}^n d_j
  \int_{\gamma_k} \star\d G(\cdot, \, b_j)
  - \int_{\gamma_k} \star\d \psi_0 + \int_{\gamma_k} A_k 
 \end{equation}
 where~$\psi_0$ is the smooth function defined by~\eqref{geompot}.
 Since we have assumed that the curves~$\gamma_k$
 do not contain any point~$b_j$ for~$\b\in V$, the Green 
 function~$G(\cdot, \, b_j)$ is smooth in both its arguments, 
 and the smoothness of~$\zeta_k$ follows.
 
 Let~$j\in\{1, \, \ldots, \, n\}$ be fixed.
 To compute the derivative of~$\nabla_{a_j}\zeta_k$,
 we fix a vector~$v\in T_{a_j} M$ and take a smooth curve
 $\gamma\colon (-\delta, \, \delta)\to M$ such that
 $\gamma(0) = a_j$, $\gamma^\prime(0) = v$.
 Let~$\b(t) = (a_1, \, \ldots, \, a_{j-1}, \, 
 \gamma(t), \, a_{j+1}, \, \ldots, \, a_n)$.
 By smoothness of~$G$, we have
 \[
  \frac{\d}{\d t}_{|t = 0} \zeta_k(\b(t), \, \db) 
  \stackrel{\eqref{zeta-Green}}{=}
  - 2\pi d_j \, \frac{\d}{\d t}_{|t=0}
  \int_{\gamma_k} \star\d G(\cdot, \, \gamma(t))
  \stackrel{\eqref{Green-conv}}{=} - 2\pi d_j 
  \int_{\gamma_k} \star\d\sigma(\cdot, \, a_j, \, v)
 \]
 and the lemma follows.
\end{proof}

\begin{proof}[Proof of Lemma~\ref{lemma:bundle}]
 Let~$(\a, \, \db)$. By Lemma~\ref{lemma:zeta},
 there exists an open neighbourhood~$V$ of~$\a$
 in~$M^n$ such that the function~$\zeta(\cdot, \, \db)\colon V\to\R^{2\gf}$
 is smooth. Let~$\xi\in\mathcal{L}(\a, \, \db)$. 
 We need to find a neighbourhood~$U\subseteq V$ of~$\a$
 and a smooth map~$\Xi\colon U\to\Harm^1(M)$ 
 such that $\Xi(\a) = \xi$ and
 \begin{equation} \label{bundle1}
  \alpha \Xi(\b) + \zeta(\b, \, \db) \in (2\pi\Z)^{2\gf}
  \qquad \textrm{for any } \b\in U
 \end{equation}
 Here~$\alpha$ is the linear map defined by~\eqref{alpha}.
 (Equation~\eqref{bundle1} is equivalent to~$\Xi(\b)\in\mathcal{L}(\b, \, \db)$,
 due to~\eqref{Lad}.) As~$(2\pi\Z)^{2\gf}$ is discrete, 
 \eqref{bundle1} is equivalent to require 
 that~$\b\mapsto \alpha\Xi(\b) + \zeta(\b, \, \db)$
 be constant, that is
 \begin{equation} \label{bundle2}
  \alpha \Xi(\b) + \zeta(\b, \, \db) 
  = \alpha \xi + \zeta(\a, \, \db) 
  \qquad \textrm{for any } \b\in U
 \end{equation}
 Since~$\alpha$ is invertible \cite[Lemma~5.2]{JerrardIgnat_full},
 the existence, uniqueness and regularity of a map~$\Xi$
 satisfying~\eqref{bundle2} follow by the implicit function theorem.
\end{proof}

\begin{remark} \label{rk:bundle}
 For any index~$j\in\{1, \, \ldots, \, n\}$
 and any~$v\in T_{a_j} M$, the directional
 derivative~$\Xi^\prime_{j, \a, v} := (\nabla_{a_j} \Xi(\a), \, v)_g$ 
 may be computed by differentiating~\eqref{bundle2}:
 \[
  \Xi^\prime_{j,\a,v} 
  = - \alpha^{-1}\left(\nabla_{a_j}\zeta(\a, \, \db), \, v \right)_g
 \]
 Recalling the definition of~$\alpha$,
 i.e.~Equation~\eqref{alpha}, and Lemma~\ref{lemma:zeta},
 we deduce that~$\Xi^\prime_{j,\a,v}\in\Harm^1(M)$
 is the unique harmonic $1$-form that satisfies
 \begin{equation} \label{derivativePsi}
  \int_{\gamma_k} \Xi^\prime_{j,\a,v}
  = - \left(\nabla_{a_j}\zeta_k(\a, \, \db), \, v \right)_g
  = 2\pi d_j \int_{\gamma_k} \star\d\sigma(\cdot, \, a_j, \, v)
 \end{equation}
 for any~$k\in\{1, \, \ldots, \, 2\gf\}$. This characterisation
 will be useful in the proof of Proposition~\ref{prop:gradient_reno}.
\end{remark}

Before we proceed to the proof of Proposition~\ref{prop:gradient_reno},
we point out a few consequences of Lemma~\ref{lemma:bundle}.

\begin{corollary} \label{cor:uniqueness}
 Let~$(\a^0, \, \db)\in\mathscr{A}^n$ 
 and~$\xi^0\in\mathcal{L}(\a^0, \, \d)$. Then, the following properties hold.
 \begin{enumerate}[label=(\roman*)]
  \item For any continuous map $\a\colon [0, \, T]\to M^n$
  such that~$\a(0)= \a^0$ and~$(\a(t), \, \db)\in\mathscr{A}^n$ for any~$t$,
  there exists a unique continuous map~$\xi\colon [0, \, T]\to\Harm^1(M)$
  such that
  \begin{equation} \label{corun1}
   \xi(t)\in\mathcal{L}(\a(t), \, \db) \quad 
    \textrm{for any } t\in [0, \, T], \qquad
    \xi(0) = \xi^0
  \end{equation}
  
  \item The (forward-in-time) Cauchy problem
  \begin{equation} \label{corun2}
   \begin{cases}
    \dfrac{\d}{\d t} \a(t) = 
    -\dfrac{1}{\pi}\nabla_{\a} W(\a(t), \, \db, \, \xi(t)) \\[.4cm]
    \xi(t)\in\mathcal{L}(\a(t), \, \db) \\
    \a(0) = \a^0, \qquad \xi(0) = \xi^0
   \end{cases}
  \end{equation}
  has a \emph{unique} maximal solution~$(\a, \, \xi)$,
  defined in an interval~$[0, \, t^*)$.
 \end{enumerate} 
\end{corollary}
\begin{proof}
 Let~$U\subseteq M^n$ be a neighbourhood of~$\a^0$ 
 and~$\Xi\colon U\to\Harm^1(M)$
 a smooth map such that
 \[
  \Xi(\b) \in \mathcal{L}(\b, \, \db) \quad 
   \textrm{for any } \b\in U, \qquad
  \Xi(\a^0) = \xi^0, 
 \]
 as given by Lemma~\ref{lemma:bundle}.
 We prove~(i) first. Let~$\a\colon [0, \, T]\to M^n$
 be a continuous map such that~$\a(0)= \a^0$
 and~$(\a(t), \, \db)\in\mathscr{A}^n$ for any~$t$.
 Then, there exists~$\tau>0$ such that
 $\xi(t) := \Xi(\a(t))$ is well-defined and continuous 
 for $t\in [0, \, \tau)$. By construction, $\xi$ satisfies~\eqref{corun1}.
 If~$\tilde{\xi}\colon [0, \, \tau)\to\Harm^1(M)$ is another 
 continuous map that satisfies~\eqref{corun1}, 
 then \eqref{Lad} implies
 \[
  \tilde{\xi}(t) - \xi(t) 
   \in\alpha^{-1}\left(2\pi\Z\right)^{2\gf}
 \]
 for any~$t\in [0, \, \tau)$, where~$\alpha\colon\Harm^1(M)\to\R^{2\gf}$
 is given by~\eqref{alpha}. As~$\tilde{\xi} - \xi$ is continuous
 and~$(2\pi\Z)^{2\gf}$ is discrete, we obtain~$\tilde{\xi} = \xi$.
 Now, (i) follows by a covering argument.
 
 Next, we prove~(ii). We define
 $Z(\b) := W(\b, \, \db, \, \Xi(\b))$ for~$\b\in U$.
 We consider the (forward-in-time) Cauchy problem
 \begin{equation} \label{corun3}
  \begin{cases}
   \dfrac{\d}{\d t} \a(t) = 
    -\dfrac{1}{\pi}\nabla_{\a} Z(\a(t)) \\[.4cm]
    \a(0) = \a^0,
  \end{cases}
 \end{equation}
 If~$\a$ is a solution of~\eqref{corun3},
 then~$(\a, \, \xi := \Xi(\a))$ is a solution of~\eqref{corun2}.
 Conversely, if~$(\a, \, \xi)$ is a solution of~\eqref{corun2}
 then~$\a$ is a solution of~\eqref{corun3}.
 The function~$Z$ is smooth, so local existence and uniqueness
 of a solution for~\eqref{corun3} (and hence, \eqref{corun2})
 is given by the classical Cauchy-Lipschitz theory
 for ordinary differential equations.
\end{proof}

We move on towards the proof of Proposition~\ref{prop:gradient_reno}.
As a first step, we need to understand the local behaviour of the 
function~$\sigma(\cdot, \, a, \, v)$ defined by~\eqref{sigma_av} 
in a neighbourhood of~$a$. In turns, this requires some
information on the Green function~$G$ and its regular part~$H$,
defined in~\eqref{Green-splitting}.

\begin{lemma} \label{lemma:H}
 The function~$H\in C^1(M\times M)$ in~\eqref{Green-splitting}
 satisfies
 \begin{equation} \label{Hbound}
  \abs{\nabla_x \nabla_y H(x, \, y)}_g \lesssim 1 + \abs{\log \dist_g(x, \, y)}
 \end{equation}
 for any~$(x, \, y)\in M\times M$ such that~$x\neq y$.
\end{lemma}
\begin{proof}
 The proof relies on a representation formula for the Green function
 on compact manifolds, given in~\cite[Theorem~4.13]{aubin98}.
 For~$(x, \, y)\in M\times M$ with~$x\neq y$, we define
 \begin{equation} \label{H0}
  G_0(x, \, y) := -\frac{1}{2\pi} \log \dist_g(x, \, y) \, f(\dist_g(x, \, y)),
 \end{equation}
 where~$f\in C^\infty_{\mathrm{c}}[0, \, +\infty)$ is a cut-off
 function, such that~$f(t)=1$ for $t$ sufficiently small
 (say, $t$ is smaller than half the injectivity radius of~$M$)
 and~$f(t)=0$ for~$t$ larger than the injectivity radius of~$M$.
 We define
 \begin{equation} \label{H1}
  \Gamma_1(x, \, y) := -\Delta_x G_0(x, \, y) + \delta_y
 \end{equation}
 (i.e., $\Gamma_1$ is the absolutely continuous part of 
 the distributional Laplacian of~$G_0$ in its first argument).
 By working in geodesic coordinates centered at~$y$, 
 we obtain (see~\cite[\oldS~4.9]{aubin98})
 \begin{equation} \label{H2}
  \Gamma_1(x, \, y) = \frac{1}{2\pi \dist_g(x, \, y)}
  \partial_\nu\left(\log\sqrt{\det g(x)}\right)
 \end{equation}
 where~$\partial_\nu$ denotes the derivative in the radial direction.
 The metric tensor in geodesic coordinates satisfies
 $g(x) = \Id + \mathrm{O}(\dist_g(x, \, y)^2)$,
 so~$\Gamma_1$ is bounded. We define 
 inductively the sequence of functions~$\Gamma_k$ as
 \[
  \Gamma_{k+1}(x, \, y) := \int_M \Gamma_k(x, \, z) \, \Gamma_1(z, \, y) \, \Vg(z)
  \qquad \textrm{for } k\geq 1, \ (x, \, y)\in M\times M, \ x\neq y.
 \]
 An argument by induction, see~\cite[Proposition~4.12]{aubin98},
 shows that~$\Gamma_k$ extends to a function of class~$C^{k-1}(M\times M)$
 for~$k\geq 2$. Now, $H$ can be written in the form
 \begin{equation} \label{H3}
  H(x, \, y) = G(x, \, y) - G_0(x, \, y)
  = \sum_{k=1}^2 \int_M \Gamma_k(x, \, z) \, G_0(z, \, y) \, \Vg(z)
  + F(x, \, y),
 \end{equation}
 for some function~$F$ that satisfies
 \begin{equation} \label{H4}
  \Delta_x F(x, \, y) = \Gamma_3(x, \, y) + \frac{1}{\mathrm{Vol}(M)}
 \end{equation}
 (see Equations~(17) and~(18) in~\cite[Theorem~4.13]{aubin98}).
 By differentiating~\eqref{H4} with respect to~$y$, we obtain
 \begin{equation*}
  \Delta_x \nabla_y F(x, \, y) = \nabla_y \Gamma_3(x, \, y)
 \end{equation*}
 Since~$\Gamma_3$ is of class~$C^2$, Schauder's estimates
 implies that~$\nabla_y F(\cdot, \, y)\in C^{2, \beta}(M)$ 
 for any~$\beta\in (0, \, 1)$ and any~$y\in M$. 
 Moreover, by differentiating~\eqref{H0} and~\eqref{H1}, we see that
 \begin{equation} \label{H5}
  \abs{\nabla_x \Gamma_1(x, \, y)}_g + \abs{\nabla_y G_0(x, \, y)}_g
   \lesssim \dist_g(x, \, y)^{-1}
 \end{equation} 
 while on the other hand~$\abs{\nabla_x \Gamma_2(x, \, y)}_g \leq C$
 because~$\Gamma_2$ is of class~$C^1$. Equation~\eqref{H2} gives
 \begin{equation} \label{H6}
  \begin{split}
   \nabla_x\nabla_y H(x, \, y) 
   = \sum_{k=1}^2 \int_M \nabla_x\Gamma_k(x, \, z) \, \nabla_y G_0(z, \, y) 
    \, \Vg(z) + \mathrm{O}(1)
  \end{split}
 \end{equation}
 The estimate~\eqref{Hbound} follows from~\eqref{H5} and~\eqref{H6} ---
 see e.g.~\cite[Proposition~4.12]{aubin98} for the details.
\end{proof}

Given a point~$a\in M$, we will denote by~$\nu$ 
the unit vector field, locally defined 
in a neighbourhood of~$a$ except at the point~$a$ itself,
that is orthogonal to the geodesic circles centered at~$a$ 
and points outward.

\begin{lemma} \label{lemma:Dnu}
  For any~$a\in M$, we have
 \[
   D\nu = \frac{1}{\dist_g(\cdot, \, a)}
   (i\nu)\otimes(i\nu)^\flat + \O(1) 
 \]
 in a neighbourhood of~$a$, except at the point~$a$ itself.
\end{lemma}
The operator~$(i\nu)\otimes(i\nu)^\flat\colon TM\to TM$
is, at each point~$p\in M$, the othogonal projection 
onto the line spanned by~$i\nu(p)$. Actually, the covariant
derivative of~$D\nu$ is related to the Hessian of the function 
$\dist_g(\cdot, \, a)^2$. There are results in the literature
which provide a fine description of the Hessian
of~$\dist_g(\cdot, \, a)^2$ and hence, of~$D\nu$
(see e.g.~\cite[Chapter~VI]{Petersen}).
However, the statement given here is enough for our purposes.
We include a proof, for the sake of convenience.
\begin{proof}[Proof of Lemma~\ref{lemma:Dnu}]
 We work in geodesic normal coordinates centered at~$a$.
 By abuse of notation, we identify~$x$ with a point in~$\R^2$
 and~$a$ with the origin in~$\R^2$. Then, by Gauss' lemma, 
 we can write $\dist_g(x, \, a) = \abs{x}$ and
 $\nu(x) = x/\abs{x}$. Let~$\Gamma^k_{ij}$ denote 
 the Christoffel symbols in normal coordinates.
 For any smooth vector field~$w$, we have
 \[
  \begin{split}
   D_w\nu &= \left( w^j \frac{\partial\nu^k}{\partial x^j}
    + \Gamma^k_{ij} \nu^i w^j\right) \frac{\partial}{\partial x^k} \\
   &= w^j \frac{\partial}{\partial x^j}\left(\frac{x^k}{\abs{x}}\right)
    \frac{\partial}{\partial x^k} + \O(1)
   = \frac{w^j}{\abs{x}}\left(\delta_j^k - \nu^k\nu^j\right)
    \frac{\partial}{\partial x^k} + \O(1)
  \end{split}
 \]
 As the metric tensor, in normal coordinates,
 satisfies~$g_{ij}(x) = \delta_{ij} + \O(\abs{x}^2)$, we deduce
 \[
  \begin{split}
   D_w\nu &= \frac{1}{\abs{x}}\left(w - (\nu, \, w)_g \nu\right) + \O(1)
   = \frac{1}{\abs{x}} (i\nu, \, w)_g \, i\nu + \O(1)
  \end{split}
 \]
 and the lemma follows. 
\end{proof}

\begin{lemma} \label{lemma:sigma}
 Let~$a\in M$. Let~$v$ be a smooth vector field, locally defined
 in a neighbourhood of~$a$. 
 Then, the function~$\sigma(\cdot, \, a, \, v(a))$
 defined by~\eqref{sigma_av} satisfies
 \begin{align}
  \sigma(\cdot, \, a, \, v(a)) &= \frac{1}{2\pi \dist_g(\cdot, \, a)} 
   \left(\nu, \, v\right)_g + \mathrm{O}(1) \label{lemmasigma} \\
  \star\d\sigma(\cdot, \, a, \, v(a)) &=
   - \frac{1}{2\pi \dist_g(\cdot, \, a)^2}\left((i\nu, \, v)_g\,\nu^\flat
    + (\nu, \, v)_g\star\nu^\flat\right)
    + \mathrm{O}(\dist_g(\cdot, \, a)^{-1})
   \label{lemmadsigma}
 \end{align}
\end{lemma}
\begin{proof}
 For simplicity, we write~$\sigma$ instead 
 of~$\sigma(\cdot, \, a, \, v(a))$.
 Recalling the definition of~$\sigma$, and writing the Green function
 in the form~\eqref{Green-splitting}, we obtain
 \begin{equation} \label{sigma-1}
  \sigma(x) = -\frac{1}{2\pi} \left(\nabla_a \log\dist_g(x, \, a), 
   \, v(a)\right)_g
  + \left(\nabla_a H(x, \, a), \, v(a)\right)_g
 \end{equation}
 The term~$(\nabla_a H(x, \, a), \, v(a))_g$ is~$\mathrm{O}(1)$
 because~$H$ is of class~$C^1$ in both variables.
 We compute $\nabla_a\log\dist_g(x, \, a)$.
 By Gauss' lemma, $\nabla_a\dist_g(x, \, a) = \tilde{\nu}(a)\in T_a M$,
 where~$\tilde{\nu}$ is the unit vector field
 that is orthogonal to the geodesic circles centered at~$x$ 
 and points outward. By the chain rule, we obtain
 \begin{equation} \label{sigma0}
  \begin{split}
   \sigma(x) = -\frac{1}{2\pi\dist_g(x, \, a)} 
   \left(\tilde{\nu}(a), \, v(a)\right)_g
   + \left(\nabla_a H(x, \, a), \, v(a)\right)_g
  \end{split}
 \end{equation}
 If~$x$ and~$a$ are close enough, there exists
 a unique minimizing geodesic~$\gamma\colon[0, \, 1]\to M$
 such that~$\gamma(0)=a$, $\gamma(1) =x$.
 Let~$\tilde{\gamma}\colon[0, \, 1]\to M$
 be the geodesic defined by~$\tilde{\gamma}(t) := \gamma(1-t)$.
 Let~$P_{x,a}\colon T_x M \to T_a M$ denote parallel
 transport along~$\tilde{\gamma}$ (see e.g.~\cite[Chapter~6, p.~234]{spivak}).
 Then,
 \[
  \tilde{\nu}(a) = \tilde{\gamma}^\prime(1) = P_{x,a}\tilde{\gamma}^\prime(0)
  = - P_{x,a}\gamma^\prime(1) = -P_{x,a}\nu(x)
 \]
 and hence, \eqref{sigma0} can be written as
 \begin{equation} \label{sigma0tris}
  \begin{split}
   \sigma(x) 
   &= \frac{1}{2\pi\dist_g(x, \, a)} 
    \left(P_{x,a}\nu(x), \, v(a)\right)_g
    + \left(\nabla_a H(x, \, a), \, v(a)\right)_g
  \end{split}
 \end{equation}
 The parallel transport  map~$P_{x,a}$ is an isometry,
 and its inverse~$P_{a,x} := P_{x,a}^{-1}\colon T_a M\to T_x M$
 is the parallel tranport along~$\gamma$. Then, \eqref{sigma0tris}
 is equivalent to
 \begin{equation} \label{sigma0bis}
  \begin{split}
   \sigma(x) 
   &= \frac{1}{2\pi\dist_g(x, \, a)} 
    \left(\nu(x), \, P_{a,x}v(a)\right)_g
    + \left(\nabla_a H(x, \, a), \, v(a)\right)_g \
  \end{split}
 \end{equation}
 Since both~$v$ and~$x\mapsto P_{a,x}v(a)$ are smooth,
 with~$P_{a,a}v(a) = v(a)$, we have
 \begin{equation} \label{Paxv}
  \abs{P_{a,x}v(a) - v(x)}_g = \O(\dist_g(x, \, a))
 \end{equation}
 and hence, \eqref{lemmasigma} follows.
 
 We pass to the proof of~\eqref{lemmadsigma}.
 By differentiating~\eqref{sigma0bis}, we obtain
 \begin{equation*} 
  \begin{split}
   \d\sigma(x) &= \frac{1}{2\pi} 
   \d\left(\frac{\nu(x)}{\dist_g(x, \, a)}, \, P_{a,x}v(a)\right)_g
   + \d\left(\nabla_a H(x, \, a), \, v(a)\right)_g
  \end{split}
 \end{equation*}
 Here and throughout the rest of the proof,
 $\d$ denotes the exterior differential with respect to~$x$.
 By Lemma~\ref{lemma:H}, $\d(\nabla_a H(x, \, a), \, v(a))_g =
 \mathrm{O}(\log\dist_g(x, \, a))$. Moreover, the vector 
 field~$x\mapsto P_{a,x} v(a)$ is smooth (in particular,
 its derivatives are bounded). Therefore,
 \begin{equation*} 
  \begin{split}
   \d\sigma(x) &= -\frac{(\nu(x), \, P_{a,x}v(a))_g}{2\pi\dist_g(x, \, a)^2}
    \, \d\left(\dist_g(x, \, a)\right) 
   + \frac{(D \nu(x), \, P_{a,x} v(a))}{2\pi\dist_g(x, \, a)}
   + \O(\dist_g(x, \, a)^{-1}) \\
   &= -\frac{(\nu(x), \, P_{a,x}v(a))_g}
    {2\pi\dist_g(x, \, a)^2} \, \nu^\flat(x)
   + \frac{(i\nu(x), \, P_{a,x} v(a))_g}{2\pi\dist_g(x, \, a)^2} (i\nu(x))^\flat
   + \O(\dist_g(x, \, a)^{-1})
  \end{split}
 \end{equation*}
 The second equality follows by Gauss' lemma 
 and Lemma~\ref{lemma:Dnu}.
 Finally, by applying~\eqref{Paxv} and Lemma~\ref{lem:vettore_ruotato},
 we deduce
 \begin{equation*} 
  \begin{split}
   \d\sigma(x) &= -\frac{(\nu(x), \, v(x))_g}
    {2\pi\dist_g(x, \, a)^2} \, \nu^\flat(x)
   + \frac{(i\nu(x), \, v(x))_g}{2\pi\dist_g(x, \, a)^2} \star\nu^\flat(x)
   + \O(\dist_g(x, \, a)^{-1})
  \end{split}
 \end{equation*}
 and the lemma follows.
\end{proof}

Our next result gives a characterisation of~$\nabla_{a_j} W(\a, \, \db, \, \xi)$
in terms of~$\xi$, $H$ and the function~$S_j$ defined in~\eqref{S_j}.

\begin{lemma} \label{lemma:gradrenoSHg}
 Let~$(\a, \db)\in\mathscr{A}^n$, $\xi\in\mathcal{L}(\a, \, \db)$.
 For any~$j\in\{1, \, \ldots, \, n\}$, we have
 \[
  \begin{split}
   \nabla_{a_j} W(\a, \, \db, \, \xi)
   = 2\pi d_j \left(\nabla S_j(a_j) + 2\pi d_j \, \nabla H(a_j, \, a_j) 
   + i\xi^\#(a_j) \right)
  \end{split}
 \]
\end{lemma}
\begin{proof}
 We consider the expression for~$W$ given by Proposition~\ref{prop:renormalized}.
 We computed already, in Lemma~\ref{lemma:gradrenoSH}, the gradient 
 of the terms that do not depend on~$\xi$; now, we 
 focus our attention on the term that depends on~$\xi$.
 We take a tangent vector~$v\in T_{a_j} M$.
 We consider the smooth map~$\b\mapsto\Xi(\b)$,
 defined locally in a neighbourhood of~$\a$ in~$M^n$,
 such that~$\Xi(\a) = \xi$ and~$\Xi(\b)\in\mathcal{L}(\b, \, \db)$
 for any~$\b$, as given by Lemma~\ref{lemma:bundle}. 
 It suffices to prove that
 \begin{equation} \label{gradrenoSHgmain}
  \frac{1}{2}\left(\nabla_{a_j} \int_M \abs{\Xi(\a)}^2 \, \Vg, \, v \right)_g
  = 2\pi d_j \left(i\xi^{\#}(a_j), \, v\right)_g
 \end{equation}
 Combining~\eqref{gradrenoSHgmain} with Lemma~\ref{lemma:gradrenoSH},
 the result will follow.

 \setcounter{step}{0}
 \begin{step}
  Let~$\gamma_1$, \ldots, $\gamma_{2\gf}$ is a family of 
  simple, closed, smooth curves whose homology classes generate
  the homology group~$H_1(M; \, \Z)$. 
  As before, we assume (up to a perturbation of the~$\gamma_k$'s)
  that the curves~$\gamma_k$ do not contain any of the points~$a_j$.
  Let~$\Xi^\prime := (\nabla_{a_j} \Xi(\a), \, v)_g$ be the derivative
  of~$\Xi$ with respect to the~$a_j$-variable, in the direction of~$v$.
  We have seen in Remark~\ref{rk:bundle}, Equation~\eqref{derivativePsi},
  that~$\Xi^\prime$ is the unique harmonic $1$-form such that
  \begin{equation}  \label{gradrenSHg2}
   \int_{\gamma_k} \Xi^\prime
   = 2\pi d_j \int_{\gamma_k} \star\d\sigma(\cdot, \, a_j, \, v)
  \end{equation}
  for any~$k\in\{1, \, \ldots, \, 2\gf\}$, 
  where~$\sigma(\cdot, \, a_j, \, v)$ is given by~\eqref{sigma_av}.
  From now on, we write~$\sigma$ instead of~$\sigma(\cdot, \, a_j, \, v)$.
  Since~$\sigma$ is smooth in~$M\setminus\{a_j\}$, the $1$-form
  \begin{equation} \label{gradrenSHg3}
   \beta := \frac{\Xi^\prime}{2\pi d_j} - \star\d\sigma
  \end{equation}
  is smooth in~$M\setminus\{a_j\}$ and satisfies
  \begin{equation} \label{gradrenSHg4}
   \int_{\gamma_k} \beta = 0 
   \qquad \textrm{for any } k\in\{1, \, \ldots, \, 2\gf\}
  \end{equation}
  due to~\eqref{gradrenSHg2}. By differentiating~\eqref{gradrenSHg3},
  and keeping in mind that a harmonic form is closed, we obtain
  \[
   \d\beta = -\d\star\d\sigma = \star\d^*\d\sigma = \star\Delta\sigma = 0
   \qquad \textrm{in } M\setminus\{a_j\},
  \]
  i.e.~$\beta$ is closed in~$M\setminus\{a_j\}$.
  Moreover, the homology classes of the curves~$\gamma_1$, \, \ldots, \, $\gamma_{2\gf}$
  are also generators of~$H_1(M\setminus\{a_j\}; \, \Z)$,
  because the inclusion~$M\setminus\{a_j\}\hookrightarrow M$
  induces a group isomorphism~$H_1(M\setminus\{a_j\}; \, \Z)\to H_1(M; \, \Z)$.
  This may be seen by writing the surface~$M$ as a quotient of a polygon
  (more precisely, a~$(4\gf)$-gon) with suitable identifications 
  on the boundary; see e.g.~\cite[Chapter~1, Section~5]{Massey}
  and Figure~\ref{fig:CW}.
  $M\setminus\{a_j\}$ retracts by deformation
  onto the $1$-skeleton of~$M$ and~$M$ is obtained by attaching 
  a~$2$-cell along a trivial cycle of the~$1$-skeleton,
  so $H_1(M\setminus\{a_j\}; \, \Z)\simeq H_1(M; \, \Z)$.\footnote{An 
  alternative argument: there is an exact sequence
  \[
   H_2(M; \, \Z) \to H_2(M, \, M\setminus\{a_j\}; \, \Z) 
   \to H_1(M\setminus\{a_j\}; \, \Z) \xrightarrow{\chi_*} H_1(M; \, \Z)
   \to 
   0
  \]
  of relative homology groups with integer coefficients
  (see e.g.~\cite[Theorem~2.16]{Hatcher}). Here $\chi_*$
  is the group homomorphism induced by the 
  inclusion $\chi\colon M\setminus\{a_j\}\hookrightarrow M$.
  As~$M$ is closed and orientable, the first arrow is an isomorphism
  (see e.g.~\cite[Theorem~3.26]{Hatcher}); it follows that~$\chi_*$ is an isomorphism, too. }
  Then, thanks to~\eqref{gradrenSHg4}, $\beta$ is not only
  closed but also exact in~$M\setminus\{a_j\}$. As a consequence,
  there exists a function~$f\in C^\infty(M\setminus\{a_j\})$ such that
  \begin{equation} \label{gradrenSHg5}
   \frac{\Xi^\prime}{2\pi d_j} = \star\d\sigma + \d f
   \qquad \textrm{in } M\setminus\{a_j\}
  \end{equation}
 \end{step}

 \begin{step}
  Let
  \begin{equation*}
   I := \frac{1}{4\pi d_j}
   \left(\nabla_{a_j} \int_M \abs{\Xi(\a)}^2 \, \Vg, \, v \right)_g
  \end{equation*}
  Let~$\eta>0$ be a small parameter and let~$B_\eta(a_j)$
  be the geodesic disk with radius~$\eta$ and center~$a_j$. 
  We may evaluate~$I$ as
  \begin{equation*}
   \begin{split}
    I = \frac{1}{2\pi d_j} \int_M (\xi, \, \Xi^\prime)_g \, \Vg
    &= \frac{1}{2\pi d_j} \lim_{\eta\to 0}\int_{M\setminus B_\eta(a_j)}
     (\xi, \, \Xi^\prime)_g \, \Vg \\
    &\hspace{-.2cm} \stackrel{\eqref{gradrenSHg5}}{=}
    \lim_{\eta\to 0}\int_{M\setminus B_\eta(a_j)}
     (\xi, \, \star\d\sigma + \d f)_g \, \Vg \\
    &= \lim_{\eta\to 0}\int_{M\setminus B_\eta(a_j)}
     (\xi, \, -\d^*(\star\sigma) + \d f)_g \, \Vg
   \end{split}
  \end{equation*}
  We integrate by parts in the right-hand side, 
  by applying~\eqref{eq:in_parts_form}:
  \begin{equation} \label{gradrenSHg6}
   I = \lim_{\eta\to 0}\int_{\partial B_\eta(a_j)}
     \left(-\sigma\xi - f(\star\xi)\right)
  \end{equation}
  The volume integrals are equal to zero because~$\xi$
  is harmonic, hence~$\d\xi=0$, $\d^*\xi=0$.
  The signs account for orientations effects,
  as~$\partial B_\eta(a_j)$ is given the orientation induced by~$B_\eta(a_j)$,
  not by~$M\setminus B_\eta(a_j)$. Let~$\nu$ be the outward-pointing
  unit normal to~$\partial B_\eta(a_j)$. Then, $i\nu$ is a unit-norm,
  tangent vector field on~$\partial B_\eta(a_j)$, consistent with the orientation
  induced by~$B_\eta(a_j)$. By definition, we have
  \begin{equation} \label{gradrenSHg6.5}
   \int_{\partial B_\eta(a_j)} \left(-\sigma\xi - f(\star\xi)\right)
   = \int_{\partial B_\eta(a_j)} \left(-\sigma(\xi^\#, \, i\nu)_g
    - f((\star\xi)^\#, \, i\nu)_g\right) \d\H^1
  \end{equation}
  and, by Lemma~\ref{lem:vettore_ruotato}, $(\star\xi)^\# = i\xi^\#$.
  Therefore, from~\eqref{gradrenSHg6} and~\eqref{gradrenSHg6.5} we obtain
  \begin{equation*}
   \begin{split}
    I &= \lim_{\eta\to 0} \left(-\int_{\partial B_\eta(a_j)}
     \sigma \, (\xi^\#, \, i\nu)_g \, \d\H^1
     - \int_{\partial B_\eta(a_j)}
     f (i\xi^\#, \, i\nu)_g \, \d\H^1 \right) 
   \end{split}
  \end{equation*}
  or, since~$i$ is a rotation by~$\pi/2$ in each tangent plane, 
  \begin{equation} \label{gradrenSHg7}
   \begin{split}
    I &= \lim_{\eta\to 0} \left(\int_{\partial B_\eta(a_j)}
     \sigma \, (i\xi^\#, \, \nu)_g \, \d\H^1
     - \int_{\partial B_\eta(a_j)}
     f (i\xi^\#, \, i\nu)_g \, \d\H^1 \right) 
   \end{split}
  \end{equation}
 \end{step}

 \begin{figure}[t]
  \includegraphics[height=.23\textheight]{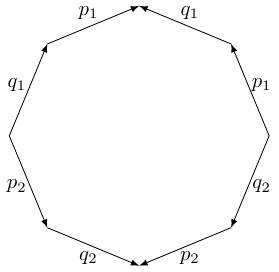}
   \caption{A closed surface of genus~$\gf > 0$ can be written
   as the quotient of a polygon, modulo a suitable equivalence 
   on the boundary. In this example, $\gf = 2$. Boundary
   edges with same label are identified to each other,
   in the direction given by the arrows. The boundary of
   the polygon carries the cycle 
   $p_1 + q_1 - p_1 - q_1 + p_2 + q_2 - p_2 - q_2 = 0$.}
  \label{fig:CW}
 \end{figure}
 
 \begin{step}
  In order to proceed further, we need to estimate~$\sigma$
  and~$f$ on~$\partial B_\eta(a_j)$.
  We work in normal geodesic coordinates centered at~$a_j$.
  We identify points in~$B_\eta(a_j)$ with points in~$\R^2$
  and vector fields with maps~$B_\eta(a_j)\subseteq\R^2\to\R^2$
  (in particular, we identify $v$ with a constant~$v\in\R^2$).
  An estimate for~$\sigma$ is given by Lemma~\ref{lemma:sigma}:
  \begin{equation} \label{gradrenSHg8}
   \sigma = \frac{1}{2\pi\eta} (\nu, \, v)_g + \mathrm{O}(1)
    \qquad \textrm{on } \partial B_\eta(a_j)
  \end{equation}
  As for~$f$, we combine~\eqref{gradrenSHg5} with~\eqref{lemmadsigma}
  to obtain
  \begin{equation}  \label{gradrenSHg9}
   \d f = -\star\d\sigma + \mathrm{O}(1)
   = \frac{1}{2\pi\eta^2}\left((i\nu, \, v)_g\nu^\flat
    + (\nu, \, v)_g\star\nu^\flat\right)
    + \mathrm{O}(\eta^{-1})
   \qquad \textrm{on } \partial B_\eta(a_j)
  \end{equation}
  Let~$x_\eta\in\partial B_\eta(a_j)$ be a reference point, say
  $x_\eta = a_j + (\eta, \, 0)$. For any~$x\in \partial B_\eta(a_j)$,
  let~$\gamma_x$ be the circular arc in~$\partial B_\eta(a_j)$
  with endpoints~$x_\eta$ and~$x$, oriented in the anticlockwise 
  direction. Then
  \[
   f(x) = f(x_\eta) + \int_{\gamma_x} \d f 
   \stackrel{\eqref{gradrenSHg9}}{=} f(x_\eta) 
   + \frac{1}{2\pi\eta^2}\int_{\gamma_x}
    (\nu, \, v)_g\star\nu^\flat + \mathrm{O}(1)
  \]
  By Gauss' lemma, when we work in geodesic coordinates
  we have~$\nu(x) = (x-a_j)/\abs{x-a_j}$ and moreover, the metric
  tensor satisfies~$g_{ij}(x) = \delta_{ij} + \mathrm{O}(\eta^2)$.
  Therefore, writing~$x = a_j + \eta(\cos\theta_x, \, \sin\theta_x)$
  for some~$\theta_x\in(0, \, 2\pi)$, we may evaluate the integral as
  \begin{equation} \label{gradrenSHg10}
   \begin{split}
    f(x) &= f(x_\eta) + \frac{1}{2\pi\eta} \int_0^{\theta_x}
    \left(v^1\cos\theta + v^2\sin\theta\right) \d\theta
    + \mathrm{O}(1) \\
    &= \frac{1}{2\pi\eta}\left(v^1\sin\theta_x - v^2\cos\theta_x\right) 
     + C_\eta + \mathrm{O}(1)
    = -\frac{1}{2\pi\eta}(i\nu(x), \, v)_g + C_\eta + \mathrm{O}(1)
   \end{split}
  \end{equation}
  for any~$x\in\partial B_\eta(a_j)$, where~$C_\eta$ depends on~$\eta$
  but not on~$x$.
 \end{step}
 
 \begin{step}
  We combine~\eqref{gradrenSHg7}, \eqref{gradrenSHg8}
  and~\eqref{gradrenSHg10}:
  \begin{equation} \label{gradrenSHg11}
   \begin{split}
    I = \lim_{\eta\to 0} \bigg(\frac{1}{2\pi\eta}
     \int_{\partial B_\eta(a_j)} (i\xi^\#, \, \nu)_g \, (\nu, \, v)_g \, \d\H^1
     &+ \frac{1}{2\pi\eta} \int_{\partial B_\eta(a_j)}
     (i\xi^\#, \, i\nu)_g \, (i\nu, \, v)_g \, \d\H^1 \\
    + \frac{C_\eta}{2\pi\eta} 
     \int_{\partial B_\eta(a_j)} (i\xi^\#, \, i\nu)_g \, \d\H^1 
     &+ \mathrm{O}(\eta) \bigg) 
   \end{split}
  \end{equation}
  where~$C_\eta$ is an (unknown) function of~$\eta$. However,
  Stokes' theorem implies
  \[
   \begin{split}
    \int_{\partial B_\eta(a_j)} (i\xi^\#, \, i\nu)_g \, \d\H^1 
    = \int_{\partial B_\eta(a_j)} \star\xi 
    = \int_{B_\eta(a_j)} \d(\star\xi) = 0
   \end{split}
  \]
  because~$\xi$ is harmonic.
  Therefore, \eqref{gradrenSHg11} simplifies to
  \begin{equation*}
   \begin{split}
    I &= \lim_{\eta\to 0} \frac{1}{2\pi\eta}
     \int_{\partial B_\eta(a_j)}
     \left(i\xi^\#, \, (\nu, \, v)_g \, \nu +(i\nu, \, v)_g \, i\nu\right)_g\d\H^1
     = \lim_{\eta\to 0} \frac{1}{2\pi\eta}
     \int_{\partial B_\eta(a_j)} (i\xi^\#, \, v)_g \,\d\H^1
   \end{split}
  \end{equation*}
  As~$\H^1(\partial B_\eta(a_j)) = 2\pi\eta + \mathrm{O}(\eta)$
  (see e.g.~\cite[Proposition 10]{spivak}),
  the equality~\eqref{gradrenoSHgmain} follows. \qedhere
 \end{step}
\end{proof}

\begin{proof}[Proof of Proposition~\ref{prop:gradient_reno}] 
 Let~$(\a, \, \db)\in\mathscr{A}^n$, $\xi\in\mathcal{L}(\a, \, \db)$ 
 and let~$u_*$ be a canonical harmonic field for~$(\a, \, \db, \, \xi)$.
 Let~$e$ be a smooth, divergence-free vector field, defined in a neighbourhood
 of~$a_j$, and
 \[
  I_\eta(u_*)(a_j)
  := \int_{\partial B_\eta(a_j)}(D_{\nu} u_*, \, D_{e}u_*)_g\,\d\H^1
  - \frac{1}{2} \int_{\partial B_\eta(a_j)}\vert D u_*\vert^2_g 
   (\nu, \, e)_g\,\d\H^1
 \]
 We recall that, as a consequence of~\eqref{eq:current},
 we have $\vert D u_*\vert^2 = \vert j(u_*)\vert^2$, 
 $D_v u_* = j(u_*)(v) \, iu_*$ for any smooth vector field~$v$.
 Moreover, the definition of canonical harmonic field
 (Definition~\ref{def:canonical_harmonic})
 gives~$j(u_*) = \d^*\Psi + \xi$ (where~$\Psi :=\Psi(\a, \, \db)$
 is given by~\eqref{Phiad}). Then, we may evaluate~$I_\eta(u_*)(a_j)$ as
 \begin{equation} \label{grdrenog1}
  \begin{split}
   I_\eta(u_*)(a_j)
   &= \int_{\partial B_\eta(a_j)} \left(\d^*\Psi(\nu) \, \d^*\Psi(e) 
   - \frac{1}{2} \abs{\d^*\Psi}^2_g (\nu, \, e)_g\right)\d\H^1  \\
   &\qquad + \int_{\partial B_\eta(a_j)} \left(\xi(\nu) \, \xi(e) 
   - \frac{1}{2} \abs{\xi}^2_g (\nu, \, e)_g\right)\d\H^1 \\
   &\qquad + \int_{\partial B_\eta(a_j)} 
   \left(\d^*\Psi(\nu) \, \xi(e) + \xi(\nu) \, \d^*\Psi(e)
   - (\d^*\Psi, \, \xi)_g (\nu, \, e)_g\right)\d\H^1
  \end{split}
 \end{equation} 
 We call~$J^1_\eta$, $J^2_\eta$ and~$J^3_\eta$
 the three integrals at the right-hand side of~\eqref{grdrenog1},
 respectively. The limit of~$J^1_\eta$ was calculated already in
 the the previous section, when we studied the case~$\gf=0$; we have
 \begin{equation} \label{grdrenog-J1}
  \lim_{\eta\to 0} J^1_\eta = 2\pi d_j \left(\nabla S_j(a_j) 
   + 2\pi d_j \, \nabla H(a_j, \, a_j), \, e(a_j)\right)_g
 \end{equation}
 The form~$\xi$ is smooth in~$M$, so
 \begin{equation} \label{grdrenog-J2}
  \lim_{\eta\to 0} J^2_\eta = 0
 \end{equation}
 It only remains to estimate~$J^3_\eta$.
 Similarly to the case~$\gf = 0$, we write
 \[
  \begin{split}
   \d^*\Psi = -\star \d(\star \Psi)
   &\stackrel{\eqref{S_j}}{=} -\star \d(S_j + 2\pi d_j G(\cdot, a_j)) \\
   &\stackrel{\eqref{Green-splitting}}{=}
    -\star\d S_j + d_j \star \d\left(\log \dist_g(\cdot, \, a_j)\right)
    - 2\pi d_j \star \d H(\cdot, \, a_j) \\
   &= d_j \star \d\left(\log \dist_g(\cdot, \, a_j)\right) + \mathrm{O}(1)
  \end{split}
 \]
 As we have seen before (see Equation~\eqref{logdist}),
 $\nabla(\log \dist_g(\cdot, \, a_j)) = \dist_g(\cdot, \, a_j)^{-1} \nu$.
 With the help of~\eqref{eq:ruotato} we obtain
 \begin{equation} \label{grdrenog2}
  \begin{split}
   \d^*\Psi &= \frac{d_j}{\eta} (i\nu)^\flat + \mathrm{O}(1)
  \end{split}
 \end{equation}
 We substitute~\eqref{grdrenog2} in the definition of~$J^3_\eta$:
 \begin{equation} \label{grdrenog3}
  \begin{split}
   J^3_\eta &= \frac{d_j}{\eta} \int_{\partial B_\eta(a_j)} 
    \left( (\xi^\#, \, \nu)_g \, (i\nu, \, e)_g
   - (\xi^\#, \, i\nu)_g (\nu, \, e)_g\right)\d\H^1 + \mathrm{O}(\eta)
  \end{split}
 \end{equation} 
 Since~$\{\nu, \, i\nu\}$ is an orthonormal frame,
 we can write~$i\xi^\# = (i\xi^\#, \, \nu)_g\nu + (i\xi^\#, \, i\nu)_g\,i\nu$.
 By the properties of~$i$, we deduce
 \[
  \begin{split}
   (\xi^\#, \, \nu)_g \, (i\nu, \, e)_g - (\xi^\#, \, i\nu)_g (\nu, \, e)_g
   &= (i\xi^\#, \, i\nu)_g \, (i\nu, \, e)_g + (i\xi^\#, \, \nu)_g (\nu, \, e)_g
   = (i\xi^\#, \, e)_g
  \end{split}
 \]
 and hence
 \begin{equation} \label{grdrenog-J3}
  \begin{split}
   J^3_\eta
   &= \frac{d_j}{\eta} \int_{\partial B_\eta(a_j)} 
    (i\xi^\#, \, e)_g \,\d\H^1 + \mathrm{O}(\eta)
   = 2\pi d_j \, (i\xi^\#(a_j), \, e(a_j))_g + \mathrm{O}(\eta)
  \end{split}
 \end{equation} 
 Combining~\eqref{grdrenog-J1}, \eqref{grdrenog-J2} and~\eqref{grdrenog-J3},
 we obtain 
 \begin{equation*}
  \begin{split}
   \lim_{\eta\to 0} I_\eta(u_*)(a_j)
   &= 2\pi d_j \left(\nabla S_j(a_j) 
   + 2\pi d_j \, \nabla H(a_j, \, a_j) + i\xi^\#(a_j), 
   \, e(a_j)\right)_g
  \end{split}
 \end{equation*} 
 and the proposition follows by Lemma~\ref{lemma:gradrenoSHg}.
\end{proof}

\section{Vortices: the stationary case}
\label{sect:stationary}

\subsection{Compactness and Product Estimates}
\label{ssect:product_space}
In this section, we gather some preliminary results
on the Ginzburg-Landau energy for time-independent vector fields 
--- most importantly, the compactness results of Ignat \& Jerrard \cite[Proposition 8.1 and Corollary 8.3]{JerrardIgnat_full}
and the product estimates of Sandier \& Serfaty \cite[Theorem 1]{SS-product}.

\begin{prop}
\label{prop:comp_prod_est_space}
Let $u_\eps\in H^1_{\tang}(M)$ be a sequence such that 
\[
F_\eps(u_\eps)\le C\abs{\log\eps}.
\]
Then, there exists a (non-relabelled) subsequence such that
\begin{equation}
\label{eq:compact_space}
\omega(u_\eps) \to \mu = 2\pi \sum_{j=1}^n d_j \delta_{a_j}
\qquad \textrm{in } W^{-1,p}(M) \quad \textrm{for any } p\in (1, \, 2).
\end{equation}
for $n$ distinct points $a_1,\ldots, a_n$ and nonzero integers $d_1,\ldots,d_n$ 
such that $\sum_{j=1}^n d_j = \chi(M)$, $\pi \sum_{j=1}^n\abs{d_j}\le C$ and 
\begin{equation}
\label{eq:gammaliminf_stat}
\liminf_{\eps\to 0}\frac{1}{\pi \abs{\log\eps}}F_\eps(u_\eps)\ge \sum_{j=1}^n \abs{d_j}.
\end{equation}
Moreover, 
for any vector fields 
$X = X^k\frac{\partial }{\partial x_k}$ and $Y = Y^k\frac{\partial }{\partial x_k}$
with $X^k$ and $Y^k$ in $C^{0}_{c}(M)$, there holds
\begin{equation}
\label{eq:product_curvo_section}
\liminf_{\eps\to 0}\frac{1}{\abs{\log\eps}}\left(\int_{M}\vert D_{X}u_\eps\vert_g^2 \Vg\right)^{1/2}\left(\int_{M}\vert D_{Y} u_\eps\vert^2_{g}\Vg\right)^{1/2}\ge\abs{\frac{1}{2}\int_{M}\mu[X,Y]}
\end{equation}
\end{prop}
\begin{proof}
The compactness result is borrowed from Ignat\& Jerrard \cite[Theorem 2.1, Proposition 8.1 and Corollary 8.3]{JerrardIgnat_full}. 

The proof of the product estimate \eqref{eq:product_curvo_section} is based on the validity of the analogous product estimate in the Euclidean two-dimensional setting and is similar to the proof of the space-time product estimate in Proposition \ref{th:comp_prod_est}, therefore we skip it. 
\end{proof}

We state a few consequences of the product
estimate~\eqref{eq:product_curvo_section}, analogous to the results 
of \cite[Corollary 4]{SS-product}. The proof is (almost) identical.

\begin{lemma}
\label{lem:corollario4SS}
 Let~$u_\eps\in H^1_{\tang}(M)$ be a sequence of vector fields
 such that $\omega(u_\eps) \to 2\pi\sum_{j=1}^n d_j \delta_{a_j}$ in~$W^{-1,1}(M)$, where
$a_1$, \ldots, $a_n$ are distinct points in $M$ and $d_1$, \ldots, $d_n$ are non-zero integers.
Moreover, assume that $u_\eps$ is uniformly bounded in~$L^\infty(M)$ and that
 \begin{equation} \label{eq:upper_bound_coro}
  \frac{1}{2}\int_{M}\vert D u_\eps\vert^2_g \, \Vg 
  \le \pi n\abs{\log\eps} + \mathrm{o}_{\eps\to 0}(\abs{\log\eps}).
 \end{equation}
 Let~$r>0$ be small enough, in such a way that the balls~$B_r(a_j)$ 
 are pairwise disjoint. For any~$B_r(a_j)$, let $\{\tau_1, \, \tau_2\}$
 be a local orthonormal tangent frame on~$B_r(a_j)$,
 with~$\tau_2 = i\tau_1$. Let $X$ and $Y$
 be smooth vector fields on $M$. Then, $\abs{d_j} = 1$ for any $j$ and the following hold:
 \begin{align}
  &\lim_{\eps\to 0}\frac{1}{\abs{\log\eps}}\int_{\bigcup_{j=1}^n B_r(a_j)}\vert D_{\tau_j} u_\eps\vert^2_g \Vg = \pi n,\,\,\,\,\,\,\,\,\,\,\,j=1,2\label{eq:energy_projection}\\
  & \lim_{\eps\to 0}\frac{1}{2\abs{\log\eps}} \int_{M}\vert D u_\eps\vert^2_g \Vg
  = \lim_{\eps\to 0}\frac{1}{2\abs{\log\eps}} \int_{\bigcup_{j=1}^n B_r(a_j)}\vert D u_\eps\vert^2_g \Vg = n\pi, \label{eq:conv_energy_palle_lemma}\\
  &\lim_{\eps\to 0}\frac{1}{\abs{\log\eps}}\int_{M}\abs{D_{X}u_\eps}_g^2 \Vg = \pi \sum_{j=1}^n\abs{X(a_j)}_g^2,\label{eq:polarization}\\
  &\lim_{\eps\to 0}\frac{1}{\abs{\log\eps}}\int_{M}(D_{X}v_\eps,D_{Y}v_\eps)_g \Vg = \pi \sum_{j=1}^n(X(a_j),Y(a_j))_g. \label{eq:polarization2}
 \end{align}
\end{lemma}
\begin{proof}
 Let~$f,g\in C_c(B_r(a_j))$ be such that~$\abs{f} \leq 1$,
 $\abs{g}\leq 1$ and~$f(a_j) = g(a_j) = 1$. 
 Let~$\mu := 2\pi\sum_{j=1}^{n}d_j \delta_{a_j}$.
 By taking~$X = f\tau_1$, $Y=g\tau_2$ in the product
 estimate~\eqref{eq:product_curvo_section}, we obtain
 \begin{equation} \label{eq:coro41}
  \begin{split}
   &\liminf_{\eps\to 0}\frac{1}{\abs{\log\eps}}\left(\int_{B_r(a_j)}\vert D_{\tau_1}u_\eps\vert_g^2 \Vg\cdot\int_{B_r(a_j)}\vert D_{\tau_2} u_\eps\vert^2_{g}\Vg\right)^{1/2} \\
   &\qquad\qquad\qquad 
   \ge\abs{\frac{1}{2}\int_{M}\mu[f\tau_1, \, g\tau_2]} = \pi \abs{d_j} f(a_j) \, g(a_j)  = \pi \abs{d_j}
  \end{split}
 \end{equation}
 In particular, since 
 \[
  \int_{B_r(a_j)}\vert D u_\eps\vert^2_g \Vg = \int_{B_r(a_j)}\vert D_{\tau_1}u_\eps\vert_g^2 +  \vert D_{\tau_2} u_\eps\vert^2_{g}\Vg
 \]
 and 
 \[
  \left(\int_{B_r(a_j)}\vert D_{\tau_1}u_\eps\vert_g^2 \Vg\cdot\int_{B_r(a_j)}\vert D_{\tau_2} u_\eps\vert^2_{g}\Vg\right)^{1/2}\le  \frac{1}{2}\int_{B_r(a_j)}\vert D_{\tau_1}u_\eps\vert_g^2  +  \vert D_{\tau_2} u_\eps\vert^2_{g}\Vg
 \]
 we obtain the lower bound 
 \begin{align*}
  \frac{1}{2\abs{\log\eps}}\int_{B_r(a_j)}\vert D u_\eps\vert^2_g \Vg &\ge
  \frac{1}{\abs{\log\eps}}\left(\int_{B_r(a_j)}\vert D_{\tau_1}u_\eps\vert_g^2   \Vg\cdot\int_{B_r(a_j)}\vert D_{\tau_2} u_\eps\vert^2_{g}\Vg\right)^{1/2}\\
 &\ge\pi\abs{d_j} + \mathrm{o}_{\eps\to 0}(1)
 \end{align*}
 On the other hand, recalling the assumption~\eqref{eq:upper_bound_coro},
 we conclude that $\abs{d_j} = 1$ for any $j$ and
 \begin{equation} \label{eq:conv_palle}
  \begin{split}
   \frac{1}{2\abs{\log\eps}}\int_{B_r(a_j)}\vert D u_\eps\vert^2_g \Vg & =
   \frac{1}{\abs{\log\eps}}\left(\int_{B_r(a_j)}\vert D_{\tau_1}u_\eps\vert_g^2 \Vg\cdot\int_{B_r(a_j)}\vert D_{\tau_2} u_\eps\vert^2_{g}\Vg\right)^{1/2}\\
   &=\pi + \mathrm{o}_{\eps\to 0}(1)
  \end{split}
 \end{equation}
 and thus
 \[
  \frac{1}{\vert\log\eps\vert}\left(\left(\int_{B_r(a_j)}\vert D_{\tau_1}u_\eps\vert_g^2\right)^{1/2} -\left(\int_{B_r(a_j)}\vert D_{\tau_2}u_\eps\vert_g^2\right)^{1/2} \right)^{2} = \mathrm{o}_{\eps\to 0}(1).
 \]
 As a result, we obtain 
 \[
  \frac{1}{\abs{\log\eps}}\int_{B_r(a_j)}\vert D_{\tau_1}u_\eps\vert_g^2 = 
  \frac{1}{\abs{\log\eps}}\int_{B_r(a_j)}\vert D_{\tau_2}u_\eps\vert_g^2 + \mathrm{o}_{\eps\to 0}(1)
 \]
 and thus 
 \[
 \frac{1}{\abs{\log\eps}}\int_{B_r(a_j)}\vert D_{\tau_k}u_\eps\vert_g^2 = \pi + \mathrm{o}_{\eps\to 0}(1) \qquad \textrm{for } k=1,2.
 \]
 Therefore, \eqref{eq:energy_projection}
 and~\eqref{eq:conv_energy_palle_lemma} follow. 
 The estimate~\eqref{eq:conv_energy_palle_lemma},
 combined with the assumption~\eqref{eq:conv_palle}, imply
 \begin{equation}
 \label{eq:fuori_palle}
  \lim_{\eps\to 0}\frac{1}{2\abs{\log\eps}}
  \int_{M\setminus\bigcup_{j=1}^n B_r(a_j)}\vert D u_\eps\vert^2_g \Vg =0.
 \end{equation}

 Now we prove~\eqref{eq:polarization}.
 Let~$X$ be a smooth vector field on~$M$. 
 For~$\delta \in (0, \, r)$, we have
 \begin{equation}
  \label{eq:1}
  \int_{B_r(a_j)}\abs{D_X u_\eps}^2_{g}\Vg = \int_{B_r(a_j)\setminus B_\delta(a_j)}\abs{D_X u_\eps}^2_{g}\Vg + \int_{B_\delta(a_j)}\abs{D_X u_\eps}^2_{g}\Vg.
 \end{equation}
 As in Subsection~\ref{sssec:normal}, we work in geodesic normal
 coordinates with center at~$a_j$ and represent~$u_\eps$, $X$
 as vector fields~$v_\eps\colon B_\delta(0)\subseteq\R^2\to\R^2$, 
 $\bar{X}\colon B_\delta(0)\to\R^2$, as in~\eqref{uv}.
 Assume that the orthonormal frame~$\{\tau_1, \, \tau_2\}$
 satisfies~\eqref{eq:smallA}. Then, Lemma~\ref{lemma:normD} implies
 \[
  \int_{B_\delta(a_j)}\abs{D_X u_\eps}^2_{g}\Vg 
  = \left(1 + \O(\delta^2)\right)
   \int_{B_{\delta}(0)} \left( \abs{\bar{X}\cdot\nabla v_\eps}^2 
   + \O(\vert u_\eps\vert^2) \right) \d x,
 \] 
 On the other hand, \cite[Corollary 4]{SS-product} gives that 
 \[
  \lim_{\eps\to 0}\frac{1}{\abs{\log\eps}}
  \int_{B_{\delta}(0)}\abs{\bar{X}\cdot\nabla v_\eps}^2\d x 
  = \pi \abs{\bar{X}(0)}^2 = \pi\abs{X(a_j)}_g^2.
 \]
 Therefore, recalling that~$u_\eps$ is bounded in~$L^\infty(M)$ by assumption,
 \[
 \lim_{\eps\to 0}\frac{1}{\abs{\log\eps}}\int_{B_\delta(a_j)}\abs{D_X u_\eps}^2_{g} \d x
 = \pi \vert X(a_j)\vert^2 + O(\delta^2). 
 \]
 Moreover, \eqref{eq:fuori_palle} implies
 \[
  \limsup_{\eps\to 0}\frac{1}{\abs{\log\eps}} \int_{B_r(a_j)\setminus B_\delta(a_j)}\abs{D_X u_\eps}^2_{g} \Vg\le
  \limsup_{\eps\to 0}\frac{1}{\abs{\log\eps}}  \int_{B_r(a_j)\setminus B_\delta(a_j)}\vert D u_\eps\vert^2\Vg = 0
 \]
 for any~$\delta$. Therefore, we have
 \[
  \lim_{\eps\to 0}\frac{1}{\abs{\log\eps}} \int_{B_r(a_j)}\abs{D_X u_\eps}^2_{g} \Vg= \pi \abs{X(a_j)}^2_g + \O(\delta^2)
 \]
 and, letting~$\delta\to 0$,
 \begin{equation} \label{prodest2}
  \lim_{\eps\to 0}\frac{1}{\abs{\log\eps}} \int_{B_r(a_j)}\abs{D_X u_\eps}^2_{g} \Vg= \pi \abs{X(a_j)}^2_g .
 \end{equation}
 Equation~\eqref{eq:polarization} now follows by~\eqref{prodest2}
 and~\eqref{eq:fuori_palle}. Equation~\eqref{eq:polarization2}
 follows by applying~\eqref{eq:polarization} to $X-Y$ and to $X+Y$
 and taking the difference. 
\end{proof}

If the vector fields~$u_\eps$ satisfy an energy upper bound
that is stronger than~\eqref{eq:upper_bound_coro},
then we can prove even stronger results --- in particular,
we can show that the energy of~$u_\eps$ is uniformly bounded
away from a finite set of singularities. These types of results are classical
in the context of Ginzburg-Landau functionals, but we need to
adapt them to our context. Our proof relies crucially on 
the analysis in~\cite{JerrardIgnat_full}. We consider
a sequence of vector fields~$(u_\eps)_{\eps>0}$ in~$H^1_{\tang}(M)$
that satisfies
\begin{equation} \label{hp:H1bounds-flat} 
 \begin{split}
  \omega(u_\eps) \to 2\pi\sum_{j=1}^n d_j \delta_{a_j} \qquad 
  \textrm{in } W^{-1,1}(M) \quad  \textrm{as } \eps\to 0,
 \end{split} 
\end{equation}
where $a_1$, \ldots, $a_n$ are disctinct points in $M$ and $d_1$, \ldots, $d_n$
are non-zero integers. We also assume that
\begin{gather}
 \norm{u_\eps}_{L^\infty(M)} \leq C_0 \label{hp:H1bounds-Linfty} \\
 \int_{M} \left(\frac{1}{2}\abs{D u_\eps}^2_g 
  + \frac{1}{4\eps^2}(1 -  \abs{u_\eps}^2_g)^2\right)\Vg 
  \le \pi n \abs{\log\eps}  + C_0 \label{hp:H1bounds-energy}
\end{gather}
for some $\eps$-independent constant~$C_0$.
By Lemma~\ref{lem:corollario4SS}, we have $\abs{d_j}=1$ for any~$j$.

\begin{lemma} \label{lemma:H1bounds}
 Let~$u_\eps\in H^1_{\tang}(M)$ be a sequence of vector fields
 that satisfies~\eqref{hp:H1bounds-flat}, 
 \eqref{hp:H1bounds-Linfty}, \eqref{hp:H1bounds-energy}.
 Then, for any~$r>0$ small enough there exists~$\eps_0(r)>0$
 such that, for~$\eps\in (0, \, \eps_0(r))$, we have
 \begin{equation} \label{H1bound}
  \int_{M\setminus\bigcup_{j=1}^n B_r(a_j)} \left( \frac{1}{2}\abs{D u_\eps}^2_g
   + \frac{1}{4\eps^2}(\abs{u_\eps}_g^2 - 1)^2 \right) \Vg 
   \le \pi n \abs{\log r} + C,
 \end{equation}
 for some constant~$C$ that depends only on~$M$, $C_0$ and~$n$ (not on~$\eps$, $r$).
\end{lemma}

Estimates similar to~\eqref{H1bound} are already contained
in Ignat and Jerrard's work~\cite{JerrardIgnat_full} 
(see for instance Lem\-ma~8.2 and Proposition~9.1).
However, since the statement of Lemma~\ref{lemma:H1bounds}
does not appear explicitely in~\cite{JerrardIgnat_full}, 
we include a proof, for the sake of completeness.
We will deduce Lemma~\ref{lemma:H1bounds} from the so-called
`ball construction', which was first introduced in the 
Euclidean setting by Sandier~\cite{Sandier} and Jerrard~\cite{Jerrard}
(see~\cite[Lemma~8.2]{JerrardIgnat_full} for the extension
to the Riemannian setting). We will also need an auxiliary result:

\begin{lemma} \label{lemma:indexflat}
 Let~$u_\eps\in H^1_{\tang}(M)$ be a sequence of vector fields
 that satisfies~\eqref{hp:H1bounds-flat}, 
 \eqref{hp:H1bounds-Linfty}, \eqref{hp:H1bounds-energy}.
 Let~$\rho>0$ be small enough, so that the closed balls~$\bar{B}_{\rho}(a_j)$
 are pairwise disjoint. Given an index~$j\in\{1, \, \ldots, \, n\}$, 
 we assume that there exist a constant~$C$ such that
 \[
  \int_{\partial B_\rho(a_j)} \left(\frac{1}{2}\abs{D u_\eps}^2_g
  + \frac{1}{4\eps^2}(1 - \abs{u_\eps}^2_g)^2\right) \d\mathscr{H}^1
  \lesssim \frac{C\abs{\log\eps}}{\rho}
 \]
 for any~$\eps$ small enough. Then, for any~$\eps$
 small enough we have~$\abs{u_\eps}_g \geq 1/2$ 
 on~$\partial B_\rho(a_j)$ and
 \[
  \ind(u_\eps, \, \partial B_\rho(a_j)) = d_j
 \]
\end{lemma}

We will prove Lemma~\ref{lemma:indexflat} in Appendix~\ref{app:flatindex}.

\begin{proof}[Proof of Lemma~\ref{lemma:H1bounds}]
 By a density argument
 (as in~\cite[Lemma~5.1 and Proposition~8.1, Step~1]{JerrardIgnat_full}), 
 we may assume without loss of generality that the fields~$u_\eps$
 are smooth. Then, we apply~\cite[Lemma~8.2]{JerrardIgnat_full}. 
 For a given~$r > 0$ (sufficiently small) and any~$\eps$,
 we find a finite collection of pairwise disjoint
 balls~$B_{1,\eps}, \, \ldots, \, B_{N_\eps, \eps}$,
 of centers~$b_{k,\eps}$ and radii~$r_{k,\eps}$, 
 that satisfy the following properties:
 \begin{gather} 
  \sum_{k=1}^{N_\eps} r_{k,\eps} \leq \frac{r}{16} \label{H1b1} \\
  \abs{u_\eps(x)}_g \geq \frac{1}{2} \qquad 
   \textrm{for } x\in M\setminus \bigcup_{k=1}^{N_\eps} B_{k,\eps} \label{H1b2} \\
  \int_{B_{k,\eps}}\left(\frac{1}{2}\abs{D u_\eps}^2_g 
   + \frac{1}{4\eps^2}(1 - \abs{u_\eps}_g^2)^2\right)\Vg 
   \geq \abs{\ind(u_\eps, \, \partial B_{k,\eps})} 
   \left(\pi\log\frac{r}{\eps} - C\right) \label{H1b3}
 \end{gather}
 for any~$k$, $\eps$ and some constant~$C$ that depends only on~$M$, $C_0$.
 We claim that, for any~$\eps$ small enough and any $k = 1, \, \ldots, \, N_\eps$,
 there holds
 \begin{equation} \label{H1b4}
  \ind(u_\eps, \, \partial B_{k,\eps})\neq 0 
  \qquad \Longrightarrow \qquad 
  b_{k,\eps}\in \bigcup_{j=1}^n B_{r/2}(a_j)
 \end{equation}
 (We recall that the~$a_j$'s are the points in the support
 of the limit of~$\omega(u_\eps)$, i.e. $\omega(u_\eps) \to
 2\pi\sum_{j=1}^n d_j \delta_{a_j}$ in~$W^{-1,1}(M)$.)
 Suppose, towards a contradiction, that~\eqref{H1b4}
 is false. Then, up to a permutation of the indices~$k$
 and extraction of a subsequence~$\eps\to 0$, we may assume that
 \[
   \ind(u_\eps, \, \partial B_{1,\eps})\neq 0, \qquad 
   b_{1,\eps}\in M\setminus \bigcup_{j=1}^n B_{r/2}(a_j)
 \]
 By Equation~\eqref{H1b1}, we obtain
 $B_{1,\eps}\subseteq M\setminus\bigcup_{j=1}^n B_{r/4}(a_j)$
 and hence, thanks to~\eqref{H1b3},
 \begin{equation} \label{H1b5}
  \begin{split}
   \int_{M\setminus\bigcup_{j=1}^n B_{r/4}(a_j)}
    \left(\frac{1}{2}\abs{D u_\eps}^2_g 
    + \frac{1}{4\eps^2}(1 - \abs{u_\eps}_g^2)^2\right)\Vg 
   \geq \pi\log\frac{r}{\eps} - C
  \end{split}
 \end{equation}
 However, the assumption~\eqref{hp:H1bounds-energy} and 
 the lower energy bound~\eqref{eq:conv_energy_palle_lemma} 
 from Proposition~\ref{lem:corollario4SS}, combined, give
 \begin{equation} \label{H1b6}
  \begin{split}
   \int_{M\setminus\bigcup_{j=1}^n B_{r/4}(a_j)}
    \left(\frac{1}{2}\abs{D u_\eps}^2_g 
    + \frac{1}{4\eps^2}(1 - \abs{u_\eps}_g^2)^2\right)\Vg 
   = \mathrm{o}_{\eps\to 0}(\abs{\log\eps})
  \end{split}
 \end{equation}
 Equation~\eqref{H1b6} contradicts~\eqref{H1b5}, and~\eqref{H1b4} follows.
 
 Next, we claim that, for any~$\eps$ small enough and any~$j=1, \, \ldots, \, n$,
 there exists~$k_* = k_*(j, \, \eps) \in\{1, \, \ldots, \, N_\eps\}$ such that
 \begin{equation} \label{H1b7}
  \ind(u_\eps, \, \partial B_{k_*,\eps})\neq 0,
  \qquad b_{k_*,\eps}\in B_{r/2}(a_j)
 \end{equation}
 Suppose, again towards a contradiction, that~\eqref{H1b7}
 does not hold. Then, there exists an index~$j\in\{1, \, \ldots, \, n\}$
 and a countable subsequence~$\eps_m\to 0$
 such that, for any~$m$ and any~$k\in\{1, \, \ldots, \, N_{\eps_m}\}$,
 \begin{equation} \label{H1b8}
  b_{k,\eps_m}\in B_{r/2}(a_j) \qquad \Longrightarrow \qquad 
  \ind(u_{\eps_m}, \, \partial B_{k,\eps_m}) = 0
 \end{equation}
 By Fatou's lemma, we have
 \[
  \begin{split}
   &\mathscr{H}^1\left(\bigcup_{M\in\N}\bigcap_{m\geq M}
   \left\{\rho\in (r/4, \, r/2)\colon \partial B_\rho(a_j) \cap 
   \bigcup_{k=1}^{N_{\eps_m}} B_{k,\eps_m} \neq\emptyset \right\}\right) \\
   &\qquad\qquad \leq \liminf_{m\to +\infty} \mathscr{H}^1
   \left\{\rho\in (r/4, \, r/2)\colon \partial B_\rho(a_j)\cap 
   \bigcup_{k=1}^{N_{\eps_m}} B_{k,\eps_m} \neq\emptyset \right\}
   \stackrel{\eqref{H1b1}}{\leq} \frac{r}{8}
  \end{split}
 \]
 Therefore, there exists a radius~$\rho\in (r/4, \, r/2)$ and a (non-relabelled)
 subsequence~$\eps_m\to 0$ such that 
 \begin{equation} \label{H1b9}
  \partial B_\rho(a_j)\cap \bigcup_{k=1}^{N_{\eps_m}} B_{k,\eps_m} = \emptyset
 \end{equation}
 By Fubini theorem, we may also assume that
 \[
  \int_{\partial B_\rho(a_j)} \left(\frac{1}{2}\abs{D u_{\eps_m}}^2_g
  + \frac{1}{4\eps^2}(1 - \abs{u_{\eps_m}}^2_g)^2\right) \d\mathscr{H}^1
  \lesssim \frac{C\abs{\log\eps_m}}{\rho}
 \]
 From~\eqref{H1b1}, \eqref{H1b8} and~\eqref{H1b9}, we deduce that
 $\abs{u_{\eps_m}(x)}_g\geq 1/2$ for any~$x\in\partial B_\rho(a_j)$ and that 
 \begin{equation*}
  \ind(u_{\eps_m}, \, \partial B_\rho(a_j)) = 0
 \end{equation*}
 However, Lemma~\ref{lemma:indexflat} implies
 \[
  \ind(u_{\eps_m}, \, \partial B_\rho(a_j)) = d_j \neq 0
 \]
 for any~$m$ large enough. This is a contradiction, so~\eqref{H1b7} follows.
 
 We can now conclude the proof of the lemma.
 Without loss of generality, by taking~$r$ small enough, we may
 suppose that the balls~$B_r(a_j)$ are pairwise disjoint. 
 By Equation~\eqref{H1b7}, for any~$\eps$
 small enough there exist at least~$n$ values of~$k$ such that
 $\ind(u_\eps, \, \partial B_{k,\eps})\neq 0$. By~\eqref{H1b1} 
 and~\eqref{H1b4}, for these values of~$k$ we have
 $B_{k,\eps}\subseteq \bigcup_{j=1}^n B_r(a_j)$. 
 Then, \eqref{H1b3} implies
 \begin{equation} \label{H1blast}
  \int_{\bigcup_{j=1}^n B_r(a_j)}\left(\frac{1}{2}\abs{D u_\eps}^2_g 
   + \frac{1}{4\eps^2}(1 - \abs{u_\eps}_g^2)^2\right)\Vg 
   \geq n\left(\pi\log\frac{r}{\eps} - C\right)
 \end{equation}
 Combining~\eqref{H1blast} with~\eqref{hp:H1bounds-energy}, the lemma follows.
\end{proof}

Lemma~\ref{lemma:H1bounds} implies a global upper bound 
for the sequence~$u_\eps$ in~$W^{1,p}(M)$, for~$p\in (1, \, 2)$

\begin{corollary} \label{cor:W1,pbounds}
 Let~$u_\eps\in H^1_{\tang}(M)$ be a sequence of vector fields
 that satisfies~\eqref{hp:H1bounds-flat}, 
 \eqref{hp:H1bounds-Linfty}, \eqref{hp:H1bounds-energy}.
 Then, for any~$p\in (1, \, 2)$ we have
 \[
  \norm{D u_\eps}_{L^p(M)} \leq C_p,
 \]
 where~$C_p$ depends only on~$M$, $p$, $n$ and the constant~$C_0$
 that appears in~\eqref{hp:H1bounds-energy}.
\end{corollary}

Corollary~\ref{cor:W1,pbounds} follows from Lemma~\ref{lemma:H1bounds}
by an argument due to Struwe~\cite{struwe94}; details may be found,
e.g., in~\cite[Lemma~12.3]{JerrardIgnat_full}.

\subsection{Compactness for solutions of an inhomogeneous Ginzburg-Landau equation}
\label{ssec:inhomGL}

The results in Section \ref{ssect:product_space} are valid for any sequence
of vector fields that satisfies appropriate bounds. Instead, in this section
we consider vector fields~$u_\eps$ on~$M$ that solve an 
inhomogeneous Ginzburg-Landau equation, that is
\begin{equation} \label{inhomGL}
 -\Delta_g u_\eps + \frac{1}{\eps^2} \left(\abs{u_\eps}^2_g - 1\right) u_\eps 
  = f_\eps, 
\end{equation}
where~$(f_\eps)_{\eps>0}$ is a given sequence of vector fields on~$M$.
For the time being, there is still no dependence on time
in Equation~\eqref{inhomGL}.
However, later on (Section~\ref{sec:dynavortex}) we will apply these results
to the parabolic Ginzburg-Landau equation~\eqref{eq:GL-intro},
by working at fixed time~$t$ and taking~$f_\eps = -\partial_t u_\eps(t)/\abs{\log\eps}$.

As before, we assume that~$u_\eps$ satisfies
\begin{equation} \label{hp:H1bounds-flat-bis} 
 \begin{split}
  \omega(u_\eps) \to 2\pi\sum_{j=1}^n d_j \delta_{a_j} \qquad 
  \textrm{in } W^{-1,1}(M) \quad  \textrm{as } \eps\to 0,
 \end{split} 
\end{equation}
where~$a_1$, \ldots, $a_n$ are distinct points in $M$ and $d_1$, \ldots, $d_n$
are non-zero integers, and that
\begin{gather}
 \norm{u_\eps}_{L^\infty(M)} \leq C_0 \label{hp:H1bounds-Linfty-bis} \\
 \int_{M} \left(\frac{1}{2}\abs{D u_\eps}^2_g 
  + \frac{1}{4\eps^2}(1 -  \abs{u_\eps}^2_g)^2\right)\Vg 
  \le \pi n \abs{\log\eps}  + C_0 \label{hp:H1bounds-energy-bis}
\end{gather}
for some $\eps$-independent constant~$C_0$.
Lemma~\ref{lemma:H1bounds} and Corollary~\ref{cor:W1,pbounds}
imply that, up to extraction of a subsequence, 
\[
 u_\eps\rightharpoonup u_0 \qquad \textrm{weakly in } W^{1,p}(M)
 \textrm{ and in } H^1_{\mathrm{loc}}(M\setminus\{a_1, \, \ldots, \, a_n\})
\]
for some limit vector field~$u_0$, such that $\abs{u_0}_g = 1$ a.e.~in~$M$.
We are going to prove stronger convergence results.

\begin{lemma} \label{lemma:clearingout}
 Let~$u_\eps\in H^1_{\tang}(M)$ be a sequence of solutions of~\eqref{inhomGL}
 that satisfies~\eqref{hp:H1bounds-flat-bis}, 
 \eqref{hp:H1bounds-Linfty-bis}, \eqref{hp:H1bounds-energy-bis}.
 Assume that the sequence~$(f_\eps)_{\eps>0}$ is bounded in~$L^2_{\tang}(M)$. Then, 
 \[
  \abs{u_\eps}_g\to 1
 \]
 locally uniformly in $M\setminus\{a_1, \, \ldots, \, a_n\}$.
\end{lemma}


The proof of Lemma~\ref{lemma:clearingout} is based on classical arguments from
the Ginzburg-Landau theory, see~\cite{BBH}, and depends on some auxiliary results.
We recall that~$e_\eps(u)$ denotes the Ginzburg-Landau energy density
of a vector field~$u$, as defined in~\eqref{eq:energy_density}.

\begin{lemma} \label{lemma:stressenergy}
 Let~$u_\eps$ be a solution of~\eqref{inhomGL}, with $f_\eps\in L^2_{\tang}(M)$.
 Let~$U\subseteq M$ be a smooth, open set. Let~$X$ be a smooth vector field, defined in
 a neighbourhood of~$\overline{U}$, and let $\{\tau_1, \, \tau_2\}$
 be a smooth orthonormal tangent frame, defined in a neighbourhood
 of~$\overline{U}$. Then,
 \[
  \begin{split}
   &\int_U\left(e_\eps(u_\eps) \, (\div X) - (D_{\tau_j}u_\eps, \, D_{\tau_k}u_\eps)_g 
    \, (\tau_j, \, D_{\tau_k} X)_g \right) \Vg \\
   &\qquad = \int_U \left((D_{\tau_j} u_\eps, \, D_{\tau_k}u_\eps)_g \, (D_{\tau_k}\tau_j, \, X)_g  
   - (D_X u_\eps, \, f_\eps)_g \right) \Vg \\
   &\qquad -\int_{U}\left((R(\tau_j,\tau_k)u_\eps,D_k u_\eps) +(D_{[\tau_j,\tau_k]}u_\eps,D_k u_\eps)\right) \, (X,\tau_j)_g\,\Vg\\
   &\hspace{3cm} + \int_{\partial U}\left(e_\eps(u_\eps) (X, \, \nu)_g
   - (D_X u_\eps, \, D_\nu u_\eps)_g\right) \d\H^1
  \end{split}
 \]
where~$\nu$ is the outer normal to~$\partial U$.
\end{lemma}
\begin{proof}
 This lemma is a variant of the so-called stress-energy identity,
 see e.g.~\cite[Theorem 1.3.6]{helein} and 
 \cite[Proposition 3.7]{SS-book}.
 To simplify the notation, we write $D_j := D_{\tau_j}$.
 For a scalar function~$\varphi$, we write 
 $\nabla_j := (\nabla\varphi, \, \tau_j)_g = \d\varphi(\tau_j)$.
 Under the assumption~$f_\eps\in L^2_{\tang}(M)$, the second 
 derivatives of~$u_\eps$ exist in~$L^2_{\tang}(M)$, by elliptic regularity.
 Therefore, we can compute the gradient of the energy 
 density~$e_\eps(u_\eps)$.
 First of all, we recall that 
 \[
 D_j D_k u_\eps = D_k D_j u_\eps +R(\tau_j,\tau_k)u_\eps + D_{[\tau_j,\tau_k]}u_\eps,
 \]
 where $R(\cdot,\cdot)$ is the Riemannian curvature tensor and $[\cdot,\cdot]$ are the Lie Brackets, namely $[\tau_j,\tau_k]:=D_{\tau_j}\tau_k-D_{\tau_k}\tau_j$ since the Levi-Civita connection $D$ is symmetric by construction (i.e. torsion free).
 We have (recall that $\left\{\tau_1,\tau_2\right\}$ is an orthonormal frame)
 \[
  \begin{split}
   \nabla_j \left(\frac{1}{2}\abs{D u_\eps}^2_g\right)
   &= (D_jD_k u_\eps, \, D_k u_\eps)_g = (D_kD_j u_\eps,D_k u_\eps) +(R(\tau_j,\tau_k)u_\eps,D_k u_\eps) +(D_{[\tau_j,\tau_k]}u_\eps,D_k u_\eps)\\
&= \nabla_k (D_j u_\eps, \, D_k u_\eps)_g - (D_j u_\eps, \, D_kD_k u_\eps)_g + (R(\tau_j,\tau_k)u_\eps,D_k u_\eps) +(D_{[\tau_j,\tau_k]}u_\eps,D_k u_\eps)\\
   &\hspace{-0.3cm}\stackrel{\eqref{eq:rough_frame}}{=}
   \nabla_k (D_j u_\eps, \, D_k u_\eps)_g - (D_j u_\eps, \, \Delta_g u_\eps)_g 
    \\
    &+ (\div\tau_k) \, (D_j u_\eps, \, D_k u_\eps)_g + (R(\tau_j,\tau_k)u_\eps,D_k u_\eps) +(D_{[\tau_j,\tau_k]}u_\eps,D_k u_\eps)
  \end{split}
 \]
The equation~\eqref{inhomGL} implies
 \[
  \begin{split}
   \nabla_j \left(\frac{1}{2}\abs{D u_\eps}^2_g\right)
   &= \nabla_k (D_j u_\eps, \, D_k u_\eps)_g + (D_j u_\eps, \, f_\eps)_g
    - \frac{1}{\eps^2}(\abs{u_\eps}^2 - 1) \, (D_j u_\eps, \, u_\eps) \\
    &+ (\div\tau_k) \, (D_j u_\eps, \, D_k u_\eps)_g + (R(\tau_j,\tau_k)u_\eps,D_k u_\eps) +(D_{[\tau_j,\tau_k]}u_\eps,D_k u_\eps)
  \end{split}
 \]
 and hence,
 \begin{equation} \label{stressen1}
  \begin{split}
   \nabla_j e_\eps(u_\eps)
   &= \nabla_k (D_j u_\eps, \, D_k u_\eps)_g + (D_j u_\eps, \, f_\eps)_g
    + (\div\tau_k) \, (D_j u_\eps, \, D_k u_\eps)_g \\
    &+ (R(\tau_j,\tau_k)u_\eps,D_k u_\eps) +(D_{[\tau_j,\tau_k]}u_\eps,D_k u_\eps)
  \end{split}
 \end{equation}
 We integrate by parts, as in~\eqref{eq:ip_div} and~\eqref{eq:ip_covariant}:
 \[
  \begin{split}
   &\int_U\left(e_\eps(u_\eps) \, (\div X) - (D_{j}u_\eps, \, D_{k}u_\eps)_g 
    \, (\tau_j, \, D_{k} X)_g \right) \Vg \\
   &\quad = \int_U \left( - \nabla_j e_\eps(u_\eps)\, (X, \, \tau_j)_g
    + \nabla_k (D_{j}u_\eps, \, D_{k}u_\eps)_g \, (\tau_j, \, X)_g 
    +  (D_{j}u_\eps, \, D_{k}u_\eps)_g \, (D_k\tau_j, \, X)_g  \right) \Vg \\
   &\quad\quad + \int_U (\div\tau_k) \, (D_{j}u_\eps, \, D_{k}u_\eps)_g \, (\tau_j, \, X)_g \, \Vg \\
   &\hspace{3cm} + \int_{\partial U} \left(e_\eps(u_\eps) \, (X, \, \nu)_g - (D_{j}u_\eps, \, D_{k}u_\eps)_g 
    \, (\tau_j, \, X)_g \, (\tau_k, \, \nu)_g \right) \d\H^1
  \end{split}
 \]
 By substituting~\eqref{stressen1} in the right-hand side, 
 the lemma follows.
\end{proof}

\begin{lemma} \label{lemma:pohozaev}
 Let~$u_\eps$ be a solution of~\eqref{inhomGL}, 
 with $f_\eps\in L^2_{\tang}(M)$. 
 Let~$B_\rho(p)\subseteq M$ be a geodesic ball in~$M$,
 with radius~$\rho$ smaller than the injectivity radius of~$M$.
 Then, 
 \[
  \begin{split}
   \frac{1}{\eps^2}\int_{B_\rho(p)} (\abs{u_\eps}_g^2 - 1)^2 \,  \Vg 
   &\lesssim \rho \norm{f_\eps}_{L^2(B_\rho(p))} \norm{D u_\eps}_{L^2(B_\rho(p))} \\
   &\qquad\qquad + \rho^2 \int_{B_\rho(p)} e_\eps(u_\eps) \, \Vg 
   + \rho \int_{\partial B_\rho(p)} e_\eps(u_\eps) \, \d\H^1
  \end{split}
 \]
\end{lemma}
\begin{proof}
 We work in geodesic normal coordinates centered at~$p$
 and identify~$B_\rho(p)$ with a Euclidean ball, $B_\rho(0)\subseteq\R^2$.
 We recall that the metric tensor and the Christoffel symbols,
 in normal coordinates, satisfy $g_{ij}(x) = \delta_{ij} + \O(|x|^2)$
 and $\Gamma^k_{ij}(x) = \O(\abs{x})$, respectively.
 We consider the orthonormal tangent frame
 \begin{equation*}
  \tau_1 := \frac{\partial_1}{\abs{\partial_1}_g}, \qquad
  \tau_2 := \frac{\partial_2 - (\partial_2, \, \tau_1)_g \, \tau_1}
   {\abs{\partial_2 - (\partial_2, \, \tau_1)_g \, \tau_1}_g}
 \end{equation*}
 (where~$\partial_1$, $\partial_2$ are the coordinate fields,
 $\partial_j := \partial/\partial x^j$). This frame satisfies
 $D_{\tau_k}\tau_j = \O(\abs{x})$ for any~$j$, $k$ and thus $[\tau_j,\tau_k]= D_{\tau_j}\tau_k-D_{\tau_k}\tau_j = \O(\abs{x})$. 
 We apply Lemma~\ref{lemma:stressenergy} to the vector field 
 defined in coordinates as $X(x) := x$. We have
 \[
  D_{\tau_k} X = \left(\frac{\partial X^i}{\partial x^j}\tau_k^j
   + \Gamma^i_{jh} X^h \tau_k^j\right) \frac{\partial}{\partial x^i} 
  = \tau_k + \O(\abs{x}^2), \qquad 
   \div X = 2 + \O(\abs{x}^2)
 \]
 and the lemma follows.
\end{proof}

\begin{proof}[Proof of Lemma~\ref{lemma:clearingout}]
 For any~$\sigma > 0$, we define
 \[
  M_\sigma := M\setminus \bigcup_{j=1}^n B_{\sigma}(a_j)
 \]
 We fix~$\sigma > 0$. Our task is to prove that $\abs{u_\eps}_g\to 1$
 uniformly in~$M_\sigma$. We denote by~$C$
 a generic positive constant, that may depend on~$M$, $n$, $C_0$
 but not on~$\eps$, $\sigma$.
 As a preliminary remark, we observe that the assumptions
 \eqref{hp:H1bounds-flat}, \eqref{hp:H1bounds-Linfty}, \eqref{hp:H1bounds-energy}  
 and Lemma~\ref{lemma:H1bounds} imply
 \begin{equation} \label{clear0}
  \int_{M_{\sigma/4}}
   \left(\frac{1}{2}\abs{\nabla u_\eps}^2_g 
   + \frac{1}{4\eps^2}(1 - \abs{u_\eps}^2_g)^2\right) \Vg
  \leq \pi n \abs{\log\sigma} + C
 \end{equation}
 We split the proof into several steps. 
 
 \setcounter{step}{0}
 \begin{step}[An extension result]
  Let~$q\in (1, \, +\infty)$. We claim that there exists a 
  constant~$C_q$ such that, for any~$\sigma > 0$ smaller than 
  the injectivity radius of~$M$, any~$x_0\in M$ and any scalar
  function~$\rho\in W^{1,q}(B_\sigma(x_0)\setminus B_{\sigma/2}(x_0))$,
  there exists an extension~$\tilde{\rho}\in W^{1,q}(B_\sigma(x_0))$ 
  of~$\rho$ such that
  \begin{equation} \label{clear-ext}
   \norm{\nabla\tilde{\rho}}_{L^q(B_\sigma(x_0))} \leq C_q 
   \norm{\nabla\rho}_{L^q(B_\sigma(x_0)\setminus B_{\sigma/2}(x_0))}
  \end{equation}
  By working in normal geodesic coordinates centered at~$x_0$,
  without loss of generality we may identify~$B_\sigma(x_0)$
  with an Euclidean ball, $B_\sigma(0)\subseteq\R^2$.
  As the inequality~\eqref{clear-ext} is scale-invariant,
  by a scaling argument it suffices to consider the case~$\sigma = 1$.
  If the average~$\fint_{B_1(0)\setminus B_{1/2}(0)}\rho = 0$,
  then~\eqref{clear-ext} follows by standard extension results
  and the Poincar\'e inequality. In general, we first construct an 
  extension~$\tilde{\rho}_0$ of~$\rho_0 := \rho - \fint_{B_1(0)\setminus B_{1/2}(0)}\rho$, then define~$\tilde{\rho} := \tilde{\rho}_0 + \fint_{B_1(0)\setminus B_{1/2}(0)}\rho$.
 \end{step}

 \begin{step}[H\"older bounds on~$\abs{u_\eps}_g$]
  The equation~\eqref{inhomGL} and the assumption~\eqref{hp:H1bounds-Linfty} imply
  \begin{equation*} 
   \begin{split}
    \norm{\Delta_g u_\eps}_{L^2(M_{\sigma/4})}
    &\leq \frac{1}{\eps^2} \norm{\abs{u_\eps}^2_g - 1}_{L^2(M_{\sigma/4})}
      \norm{u_\eps}_{L^\infty(M)} + \norm{f_\eps}_{L^2(M)} \\
    &\leq \frac{C}{\eps^2} \norm{\abs{u_\eps}^2_g - 1}_{L^2(M_{\sigma/4})} + C
   \end{split}
  \end{equation*}
  (The norm~$\norm{f_\eps}_{L^2(M)}$ is bounded by assumption.)
  On the other hand, the energy bound~\eqref{clear0} gives
  \[
   \norm{\abs{u_\eps}^2_g - 1}_{L^2(M_{\sigma/4})} \leq C \eps \abs{\log\sigma}^{1/2}
  \]
  therefore
  \begin{equation} \label{clear1-1}
   \begin{split}
    \norm{\Delta_g u_\eps}_{L^2(M_{\sigma/4})}
    &\leq C \eps^{-1} \abs{\log\sigma}^{1/2}
   \end{split}
  \end{equation}
  Let~$\alpha\in [1/2, \, 1)$ and
  \[
   q := \frac{2}{1 - \alpha} \geq 4
  \]
  By applying elliptic regularity results
  and the Gagliardo-Nirenberg interpolation inequality, we deduce
  \begin{equation*} 
   \begin{split}
    \norm{D u_\eps}_{L^q(M_{\sigma/2})}
    &\leq C\norm{\Delta_g u_\eps}_{L^2(M_{\sigma/4})}^\alpha
     \norm{u_\eps}_{L^\infty(M)}^{1-\alpha} + C\norm{u_\eps}_{L^\infty(M)}\\
    &\hspace{-.65cm}
    \stackrel{\eqref{hp:H1bounds-Linfty}, \, \eqref{clear1-1}}{\leq} 
     C \eps^{-\alpha} \abs{\log\sigma}^{\alpha/2}
   \end{split}
  \end{equation*}
  Due to~\eqref{clear-ext}, there exists a sequence of scalar 
  functions $\tilde{\rho}_\eps\in W^{1,q}(M)$ such that
  $\tilde{\rho}_\eps = \abs{u_\eps}_g$ 
  out of~$\bigcup_{j=1}^n B_{\sigma/2}(a_j)$ and
  \begin{equation*} 
   \begin{split}
    \norm{\nabla\tilde{\rho}_\eps}_{L^q(M)}
    &\leq C \eps^{-\alpha} \abs{\log\sigma}^{\alpha/2}
   \end{split}
  \end{equation*}
  Then, by applying the Sobolev embedding $W^{1,q}(M)\hookrightarrow C^\alpha(M)$
  to~$\tilde{\rho}_\eps$, we conclude
  \begin{equation} \label{clear1-2}
   \begin{split}
    [\abs{u_\eps}_g]_{C^\alpha(M_{\sigma/2})}
    &\leq C \eps^{-\alpha} \abs{\log\sigma}^{\alpha/2},
   \end{split}
  \end{equation}
  where $[\abs{u_\eps}_g]_{C^\alpha(M_{\sigma/2})}$ denotes
  the $\alpha$-H\"older seminorm of~$u_\eps$ on~$M_{\sigma/2}$,
  that is
  \[
   [\abs{u_\eps}_g]_{C^\alpha(M_{\sigma/2})}
   := \sup_{x, \, y\in M_{\sigma/2}, \ x\neq y}
   \frac{\abs{\abs{u_\eps(x)}_g - \abs{u_\eps(y)}_g}}{\dist_g(x, \, y)^\alpha}
  \]
 \end{step}
 
 \begin{step}
  We show that $\abs{u_\eps}_g\to 1$ uniformly in~$M_\sigma$.
  Towards a contradiction, suppose this claim is false.
  Then, there exists a number~$\delta > 0$, a (non-relabelled) subsequence
  $\eps\to 0$ and, for any~$\eps$ in the subsequence, a pont~$x_\eps\in M_\sigma$
  such that 
  \begin{equation} \label{clear3-1} 
    \abs{\abs{u_\eps(x_\eps)}_g - 1} \geq \delta
  \end{equation}
  Due to~\eqref{clear1-2}, we deduce that $\abs{u_\eps}_g$
  remains close  to~$1 - \delta$ on a ball, centered at~$x_\eps$,
  of radius comparable to~$\eps\abs{\log\sigma}^{-1/2}$. 
  More precisely, we have
  \[
    \abs{\abs{u_\eps(x)}_g - 1} \geq \frac{\delta}{2}
    \qquad \textrm{for any } x\in B_{\lambda\eps\abs{\log\sigma}^{-1/2}}(x_\eps),
  \]
  where $\lambda > 0$ is a positive parameter that depends only
  on~$\delta$ and the constant~$C$ in~\eqref{clear1-2}
  --- in particular, $\lambda$ is independent of~$\eps$ and~$\sigma$.
  As a consequence, we have
  \begin{equation} \label{clear3-4}
   \frac{1}{\eps^2} \int_{B_{\lambda\eps\abs{\log\sigma}^{-1/2}}(x_\eps)} 
    \left(\abs{u_\eps}^2_g - 1\right)^2 \, \Vg 
    \geq \frac{C_\delta}{\abs{\log\sigma}}
  \end{equation}
  where~$C_\delta$ is a strictly positive constant, depending on~$\delta$
  but not on~$\eps$, $\sigma$.
 
  Now, let~$\eta > 0$ be a small parameter, to be chosen later on.
  For any~$\eps$ small enough, there exists a 
  radius~$\rho_\eps\in (\eps^{1/2}, \, \eps^{1/4})$ such that
  \begin{equation} \label{clear3-2}
   \int_{B_{\rho_\eps}(x_\eps)}  \left(\frac{1}{2}\abs{D u_\eps}_g^2
   + \frac{1}{4\eps^2}(\abs{u_\eps}^2_g - 1)^2\right) \d\H^1 \leq \frac{\eta}{\rho_\eps} 
  \end{equation} 
  --- for otherwise, we would have
  \[ 
   \int_{B_{\eps^{1/4}}(x_\eps)}  \left(\frac{1}{2}\abs{D u_\eps}_g^2
   + \frac{1}{4\eps^2}(\abs{u_\eps}^2_g - 1)^2\right) \Vg 
   \geq \int_{\eps^{1/2}}^{\eps^{1/4}} \frac{\eta}{\rho} \, \d\rho
   =  \frac{\eta}{4} \abs{\log\eps} 
  \]
  which contradicts the energy bound~\eqref{clear0} for~$\eps$ small enough.
  We can estimate the integral of~$(\abs{u_\eps}^2_g - 1)^2$ 
  on~$B_{\rho_\eps}(x_\eps)$ by applying Lemma~\ref{lemma:pohozaev}.
  We obtain, recalling~\eqref{clear3-2} and the energy bound~\eqref{clear0},
  \begin{equation} \label{clear3-3}
   \frac{1}{\eps^2}\int_{B_{\rho_\eps}(x_\eps)} (\abs{u_\eps}_g^2 - 1)^2 \,  \Vg 
   \lesssim \eps^{1/4} 
    \abs{\log\sigma}^{1/2} + \eps^{1/2} \abs{\log\sigma} + \eta
  \end{equation}
  If we choose~$\eta$ small enough, depending on~$\delta$ and~$\sigma$ only,
  then the inequality~\eqref{clear3-3} contradicts~\eqref{clear3-4}
  when~$\eps$ is small. The lemma follows. \qedhere
 \end{step}
\end{proof}

The next lemma further discusses the convergence of $u_\eps$
as~$\eps\to 0$, away from the singularities. Our proof follows closely 
F. H. Lin \cite{lin1} and \cite{lin2}. 

\begin{lemma} \label{lem:lontano_sing}
 Let~$u_\eps\in H^1_{\tang}(M)$ be a sequence of solutions of~\eqref{inhomGL}
 that satisfies~\eqref{hp:H1bounds-flat-bis}, 
 \eqref{hp:H1bounds-Linfty-bis}, \eqref{hp:H1bounds-energy-bis}.
 Assume that the sequence~$(f_\eps)_{\eps>0}$ is bounded in~$L^2_{\tang}(M)$.
 Then, up to extraction of a subsequence, we have
 \begin{align} 
   &u_\eps\to u_0 \hspace{2.25cm} \textrm{strongly in }  
   H^1_{\mathrm{loc}}(M\setminus\{a_1, \, \ldots, \, a_n\})
   \label{eq:forte_lontano_sing} \\
   &u_\eps\to u_0 \hspace{2.25cm} \textrm{locally uniformly in } 
   M\setminus \{a_1, \, \ldots, \, a_n\}
   \label{eq:unif_lontano_sing} \\
   &\frac{1}{\eps}(\abs{u_\eps}_g^2 - 1)\to 0 \qquad \textrm{strongly in }  
   L^2_{\mathrm{loc}}(M\setminus\{a_1, \, \ldots, \, a_n\}) 
   \label{eq:forte_lontano_sing_abs}
  \end{align}
\end{lemma}

\begin{remark} \label{rk:strongW1,p}
 Corollary~\ref{cor:W1,pbounds} and the strong
 convergence~\eqref{eq:forte_lontano_sing} imply,
 via the H\"older inequality, that $u_\eps\to u_0$ strongly 
 in~$W^{1,p}(M)$ for any~$p\in [1, \, 2)$.
\end{remark}

\begin{proof}[Proof of Lemma~\ref{lem:lontano_sing}]
Let~$B\csubset M\{a_1, \, \ldots, \, a_n\}$
be a geodesic ball, of radius smaller than the
injectivity radius of~$M$. By Lemma~\ref{lemma:H1bounds}
and Fubini's theorem, we can make sure that
\begin{equation} \label{eq:energy_bound_palla}
  \begin{split}
   \int_{B} \left(\abs{D u_\eps}^2_g + \frac{1}{2\eps^2}
   (1 - \abs{u_\eps}_g^2)^2\right)\Vg
   + \int_{\partial B} \left(\abs{D u_\eps}^2_g + \frac{1}{2\eps}
   (1 - \abs{u_\eps}_g^2)^2\right) \d\H^1 \le C 
  \end{split}
\end{equation}
The constant~$C$ does depend on the ball~$B$ (in particular,
on the distance of~$B$ from the singular points~$a_1$, \ldots, $a_n$),
but not on~$\eps$. Moreover, Lemma~\ref{lemma:clearingout} implies
\begin{equation}\label{hp:lontano_sing_B-abs}
  \abs{u_\eps}_g \to 1  \qquad \textrm{uniformly in } \overline{B}.
\end{equation}

\setcounter{step}{0}
\begin{step}
We let $\rho_\eps:=\vert u_\eps\vert_g$ and $v_\eps:=\frac{u_\eps}{\vert u_\eps\vert_g}$. 
The field~$v_\eps$ is well-defined and belongs to~$H^1_{\tang}(B)$,
thanks to~\eqref{hp:lontano_sing_B-abs}. 
By lifting results (see e.g.~\cite[Lemma 6.1]{ssvM3AS}), there exists a scalar 
function $\alpha_\eps\in H^1(B)$ such that
\begin{equation} 
v_\eps = \cos(\alpha_\eps) \hat{e}_1 + i\sin(\alpha_\eps) \hat{e}_1,
\end{equation}
where $\left\{\hat{e}_1,\hat{e}_2\right\}=\left\{\hat{e}_1,i\hat{e}_1\right\}$
is an orthonormal moving frame in $B$ such that $\div \mathbb{A}=0$
(see~\eqref{eq:div_free_A}).
We rewrite~\eqref{eq:GL_intro} in terms of~$\rho_\eps$ and~$\alpha_\eps$. 
First of all, note that 
\[
\Delta_g u_\eps = \Delta_g (\rho_\eps v_\eps) = \rho_\eps \Delta_g v_\eps + v_\eps\Delta_g \rho_\eps + 2 g^{ij}\partial_i \rho_\eps D_j v_\eps.
\]
Moreover (see \cite[Lemma 6.2 and Lemma 6.3]{ssvM3AS})
\[
D v_\eps = (\d \alpha_\eps -\mathcal{A})iv_\eps,
\]
and ($\mathbb{A}:=\mathcal{A}^{\sharp}$)
\[
\Delta_g v_\eps = (\Delta \alpha_\eps) iv_\eps -\vert \nabla \alpha -\mathbb{A}\vert^2_g v_\eps.
\]
Therefore, 
\begin{equation}
\label{eq:Laplacian_deviation}
\Delta_g u_\eps = \Big(\rho_\eps \Delta \alpha_\eps + 2(\nabla \alpha_\eps -\mathbb{A},\nabla \rho_\eps)_g\Big)iv_\eps + \Big(\Delta_g \rho_\eps -\vert\nabla\alpha_\eps -\mathbb{A}\vert^2_g \Big)v_\eps.
\end{equation}
We decompose $f_\eps = (f_\eps, \, v_\eps)_g\,v_\eps + (f_\eps, \, iv_\eps)_gi v_\eps
=: f_\eps^1 \, v_\eps + f_\eps^2  \, i v_\eps$. Since we have assumed that
$f_\eps$ is bounded in $L^2_{\tang}(M)$, we have
\begin{equation}
\label{eq:rhsB}
\norm{f^1_\eps}_{L^2(B)} + \norm{f^2_\eps}_{L^2(B)} \leq C.
\end{equation}   
We can thus rewrite \eqref{eq:GL_intro} as the following system
\begin{equation}
\label{eq:alpha+rho}
\begin{cases}
\displaystyle\div(\rho_\eps^2(\nabla \alpha_\eps -\mathbb{A})) = f^1_\eps,\\
\\
\displaystyle\Delta_g \rho_\eps + \frac{\rho_\eps}{\eps^2}(1-\rho_\eps^2) -\rho_\eps\vert\nabla \alpha_\eps\vert^2_g = f^2_\eps.
\end{cases}
\end{equation}
\end{step}
\begin{step}
Since $\rho_\eps\ge 1/2$ and $f^1_\eps\in L^2(B)$, it turns out that $\alpha_\eps\in H^2(B)$ and thus $h_\eps:=f^2_\eps +\rho_\eps\vert \nabla \alpha_\eps\vert^2_g\in L^2$. 
Then, we multiply the equation for $\rho_\eps$ with $1-\rho_\eps$ and we integrate on $B$. We get
\begin{align}
\label{eq:stima_palla}
\int_{B}\vert \nabla \rho_\eps\vert^2_g \Vg + \frac{1}{\eps^2}\int_{B}\rho_\eps (1-\rho_\eps)^2(1+\rho_\eps)\Vg \nonumber\\
= \int_{B}h_\eps (1-\rho_\eps) \Vg -\int_{\partial B}\partial_\nu \rho_\eps(1-\rho_\eps)\, \d\H^1.
\end{align} 
Since $\rho_\eps\ge 1/2$
\[
\frac{1}{\eps^2}\int_{B}\rho_\eps (1-\rho_\eps)^2(1+\rho_\eps)\Vg \ge \frac{3}{4\eps^2}\int_{B}(1-\rho_\eps)^2\Vg.
\]
On the other hand, since $h_\eps$ is in $L^2$ we get
\[
\Big\vert \int_{B}h_\eps (1-\rho_\eps)\Vg\Big\vert \le \frac{\delta}{2\eps^2}\int_{B}(1-\rho_\eps)^2\Vg + \frac{\eps^2}{2\delta}\int_{B} h_\eps^2 \Vg
\]
and (thanks to \eqref{eq:energy_bound_palla} and~\eqref{hp:lontano_sing_B-abs})
\[
\int_{\partial B}\partial_\nu \rho_\eps (1-\rho_\eps)\, \d\H^1 \le C\|\rho_\eps -1\|_{L^\infty(\partial B)}.
\]
Choosing $\delta =1/2$, we conclude that 
\[
\int_{B}\vert \nabla \rho_\eps\vert^2_g \Vg + \frac{1}{4\eps^2}\int_{B}(1-\rho_\eps)^2\Vg \le C\big(\eps^2 \|h_\eps\|_{L^2}^2 + \|\rho_\eps -1\|_{L^\infty(\partial B)}\big).
\]
Therefore,
\begin{equation}
\label{eq:gradrho0}
 \rho_\eps \to 1 \qquad  \textrm{strongly in } H^1(B)
\end{equation}
Moreover, since
\[
 (1-x^2)^2 = (1-x)^2(1+x)^2\le 4(1-x)^2,
\]
for $x\in (1/2,1)$, we conclude that
\begin{equation}
\label{eq:zero_pote_palla}
\lim_{\eps\to 0}\frac{1}{\eps^2}\int_{B}(1-\vert u_\eps\vert_g^2)^2 \, \Vg =0
\end{equation}
\end{step}
\begin{step}
The first equation of~\eqref{eq:Laplacian_deviation} can be
equivalently written as
\[
 \Delta_g\alpha_\eps = (1 - \rho_\eps^2)\Delta_g\alpha_\eps
 + \underbrace{(\nabla(1 - \rho_\eps^2), \, \nabla\alpha_\eps)_g
 - \div(\rho_\eps^2\mathbb{A}) + f_\eps^2}_{=: g_\eps}
\]
Both~$\rho_\eps$ and~$\alpha_\eps$ are bounded in~$H^1(B)$,
due to the energy estimate~\eqref{eq:energy_bound_palla},
while~$f_\eps^1$ is bounded in~$L^2(B)$ by~\eqref{eq:rhsB}.
Therefore, $g_\eps$ is bounded in~$L^2(B)$.
Let $B^\prime\csubset B$ be a slightly smaller ball.
Thanks to~\eqref{hp:lontano_sing_B-abs}, we can apply 
elliptic regularity estimates and conclude that $\alpha_\eps$
is bounded in~$H^2(B^\prime)$. By compact Sobolev embedding,
it follows that $\alpha_\eps$ converges strongly in~$H^1(B^\prime)$
(possibly up to extraction of a subsequence).
Therefore, since
\[
D_{\hat{e}_k}v_\eps = (\d\alpha_\eps(\hat{e}_k) -\mathcal{A}(\hat{e}_k))i v_\eps,
\]
for~$k=1, \, 2$, we get that $D_{\hat{e}_k}v_\eps$ is strongly convergent in $L^2(B^\prime)$. 
As
\[
D_{\hat{e}_k} u_\eps =\rho_\eps D_{\hat{e}_k}v_\eps
  + v_\eps (\nabla\rho_\eps, \, \hat{e}_k)_g,
\]
thanks to \eqref{eq:gradrho0} we conclude that 
$D u_\eps$ converges strongly in $L^2(B^\prime)$. 
Finally, as a byproduct of the argument above,
we obtain that $u_\eps/\abs{u_\eps}$ is bounded
in~$H^2_{\mathrm{loc}}(M\setminus\{a_1, \, \ldots, \, a_n\})$.
Therefore, $u_\eps/\abs{u_\eps}$
converges locally uniformly in~$M\setminus\{a_1, \, \ldots, \, a_n\}$,
by Sobolev embeddings. By Lemma~\ref{lemma:clearingout}, we deduce that
$u_\eps\to u_0$ locally uniformly in $M\setminus\{a_1, \, \ldots, \, a_n\}$.
\qedhere
\end{step}
\end{proof}

\section{Vortices: the dynamical case}
\label{sec:dynavortex}

\subsection{Space-Time Compactness and Product Estimates}
\label{ssec:product}

In this section we discuss the compactness of sequences of vector fields $u_\eps\colon [0, \, T]\times M\to TM$ satisfying suitable energy bounds. 
These are the typical bounds satisfied by the solutions of the (rescaled) heat flow of the Ginzburg-Landau equation. 

First of all, for a smooth vector field $u\colon[0, \, T]\times M\to TM$ we define the space-time current $j_Q(u)$ as 
\begin{equation}
\label{eq:space_time_current}
 j_Q(u) 
 := (D^Q u, \, i u)_{g_\mathrm{cyl}}
 = (\partial_t u, \, iu)_{g} \, \d t + j(u),
\end{equation}
where~$D^Q$ denotes the Riemannian connection
on the cylinder~$Q := [0, \, T]\times M$
and~$g_{\mathrm{cyl}}$ denotes the product metric in~$Q$,
$g_\mathrm{cyl} := {\d t}^2 + g$.
We denote with $\d_{Q}$ the exterior derivative on~$Q$ 
while $\d_t$ denotes the exterior derivative with respect to time and~$\d$ denotes the exterior derivative with respect to 
the variables of~$M$. Recalling from~\eqref{eq:vorticity} that 
$\d j(u) = \omega(u) -\kappa \Vg$, we obtain
\begin{equation}
\label{eq:cylindrical_vorticity}
 \d_{Q} \, j_Q(u) 
 = \d\left((\partial_t u, \, iu)_g \, \d t\right) + \d_{Q}j(u) 
 = V + \omega(u) -\kappa\Vg, 
\end{equation}
where $V$ denotes the two-form on $Q$ 
\begin{equation}
\label{eq:def_V}
 V := \d\left((\partial_t u, \, iu)_g \, \d t\right) + \d_t j(u). 
\end{equation}
Then, we define the space-time vorticity
\begin{equation}
\label{eq:defJ}
 J(u):= V + \omega(u).
\end{equation}

Now, we fix some $p\in M$ and $\delta>0$ smaller than the
injectivity radius and consider the geodesic ball~$B_\delta(p)$.
The geodesic normal coordinates define a
map~$\Phi\colon B_\delta(p)\to B_\delta(0)\subseteq\R^2$.
Let~$\{\tau_1, \, \tau_2\}$ be a positively oriented,
orthonormal tangent frame
on~$B_\delta(p)$ that satisfies~\eqref{eq:smallA}.
With the notation of Subsection~\ref{sssec:normalGL}, we let 
$v=(v_1, \, v_2)\colon [0, \, t]\times B_\delta(0)\to \mathbb{R}^2$
be the unique vector field such that
\begin{align}
 \label{eq:normal_representation_space_time}
 u(t, \, \Phi^{-1}(x)) = v^1(t, \, x) \, \tau_1(\Phi^{-1}(x))
  + v^2(t, \, x) \, \tau_2(\Phi^{-1}(x))
\end{align}
for any~$(t, \, x)\in [0, \, t]\times B_\delta(0)$.
As we have seen already in Subsection~\ref{sssec:normalGL}, 
we may express the Ginzburg-Landau energy of~$u$ in~$B_\delta(p)$
in terms of the (Euclidean) Ginzburg-Landau energy of~$v$,
see Equation~\eqref{eq:energy_normal_coordinates}. 
By applying~\eqref{eq:basic_normal}, \eqref{eq:basic_normal_nabla}
and~\eqref{eq:change_volume}, we also obtain
\begin{equation}
\label{eq:partialt_normal_coordinates}
 \int_{0}^T\int_{B_\delta(p)} \abs{\partial_t u}^2_{g} \, \Vg\,\d t 
 = \left(1 + \O(\delta^2)\right)
 \int_{0}^T\int_{B_\delta(0)}\abs{\partial_t v}^2 \, \d x \, \d t
\end{equation}
Now, we write the space-time vorticity in terms of its Euclidean counterpart. 
We let $\bar{\jmath}(v)$ be the (Euclidean) current of
$v$ with respect to the space variable $x$, namely
\begin{equation}
\label{eq:flat_space_current}
 \bar{\jmath}(v) := \sum_{k=1}^2 (\partial_k v, \, iv) \, \d x^k.
\end{equation} 
We define the form
\begin{equation} \label{eq:flat_V}
 \bar{V} := \d \left((\partial_t v, \, iv)\, \d t\right)
  + \d_t \bar{\jmath}(v)
\end{equation}
and the space-time Euclidean vorticity 
\begin{equation} \label{eq:flat_vorticity}
 \bar{J}(v) := \bar{V}+ \d \bar{\jmath}(v)
\end{equation}
(which is consistent with~\cite{SS-product} up to a factor $2$).

\begin{lemma} \label{lemma:pullback}
 Let
 \begin{equation} \label{frakR}
  \mathfrak{R} := \left(1 - \abs{v}^2\right) (\Phi_Q^{-1})^*\mathcal{A},
 \end{equation}
 where~$\mathcal{A}$ is the connection~$1$-form induced by
 the frame~$\{\tau_1, \, \tau_2\}$, as in~\eqref{eq:connection1}.
 Then, we have
 \begin{align}
  (\Phi_Q^{-1})^*V &= \bar{V} + \d_t\mathfrak{R} \label{pullbackV} \\
  (\Phi_Q^{-1})^*\omega(u) &= \d\bar{\jmath}(v) + \d\mathfrak{R} 
   \label{pullbackomega} \\
  (\Phi_Q^{-1})^*J(u) &= \bar{J}(v) 
   + \d_t\mathfrak{R} +\d\mathfrak{R} \label{pullbackJ}
 \end{align}
\end{lemma}
\begin{proof}
 As a preliminary step, we will show that
 \begin{equation} \label{eq:pull_back}
   (\Phi_Q^{-1})^* j_Q(u) = (\partial_t v, iv) \, \d t 
    + \bar{\jmath}(v) -\vert v\vert^2 (\Phi_Q^{-1})^*\mathcal{A}.
 \end{equation}
 We take the pull-back via~$\Phi_Q^{-1}$ of both sides of~\eqref{eq:space_time_current}:
 \begin{equation} \label{plbk1}
  (\Phi_Q^{-1})^* j_Q(u) = 
   (\Phi_Q^{-1})^* \left((\partial_tu, \, iu)_g \, \d t \right)
    + (\Phi_Q^{-1})^* j(u)
 \end{equation}
 We study separately the terms in the right-hand side of~\eqref{plbk1}.
 By differentiating both sides of~\eqref{eq:normal_representation_space_time}
 with respect to~$t$, we obtain
 \begin{equation} \label{plbk12}
   (\partial_t u)\circ\Phi_Q^{-1} 
   = \partial_t v^1 \, (\tau_1\circ\Phi^{-1})
    + \partial_t v^2 \, (\tau_2\circ\Phi^{-1})
 \end{equation}
 As a consequence,
 \begin{equation} \label{plbk13}
  \begin{split}
   (\Phi_Q^{-1})^* \left((\partial_tu, \, iu)_g \, \d t \right)
    &= ((\partial_tu, \, iu)_g\circ\Phi_Q^{-1}) \, \d t \\
    &= \left((\partial_t u)\circ\Phi_Q^{-1}, 
    \, i (u\circ\Phi_Q^{-1}) \right)_g \, \d t
   \stackrel{\eqref{eq:normal_representation_space_time},\eqref{plbk12}}{=}
    (\partial_t v, \, iv) \, \d t 
  \end{split}
 \end{equation}
 Now, we consider the second term in the right-hand side of~\eqref{plbk1}.
 Let~$u^k := v^k\circ\Phi_Q$ for~$k=1,2$. 
 Equation~\eqref{eq:normal_representation_space_time}
 may be written as
 \[
  u(t, \, y) = u^1(t, \, y) \, \tau_1(y) + u^2(t, \, y) \, \tau_2(y)
 \]
 for any~$(t, \, y)\in [0, \, T]\times B_\delta(p)$.
 Then, the current of~$u$ (with respect to the space variables only)
 may be written as
 \begin{equation} \label{plbk14}
   \begin{split}
    j(u) = (Du, \, iu)_g 
     &= \left(\d u^1 \otimes \tau_1 + u^1 \, D\tau_1 
     + \d u^2 \otimes \tau_2 + u^2 \, D \tau_2, 
     \, u^1 \, \tau_2 - u^2 \, \tau_1\right)_g \\
     &= u^1 \d u^2 - u^2 \d u^1 - \abs{u}^2_g\mathcal{A}
   \end{split}
 \end{equation}
 We take the pull-back via~$\Phi_Q^{-1}$ of both sides of~\eqref{plbk14}.
 Since the pull-back commutes with the exterior differential,
 and since $(\Phi_Q^{-1})^*u^k = u^k\circ\Phi_Q^{-1} = v^k$ for~$k=1,2$,
 we obtain
 \begin{equation} \label{plbk15}
   (\Phi_Q^{-1})^* j(u) = \bar{\jmath}(v) - \abs{v}^2 (\Phi_Q^{-1})^*\mathcal{A},
 \end{equation}
 Equation~\eqref{eq:pull_back} now follows from~\eqref{plbk12}, 
 \eqref{plbk13} and~\eqref{plbk15}.
 
 \medskip
 \noindent
 \emph{Proof of~\eqref{pullbackV}.}
 We take the pull-back via~$\Phi_Q^{-1}$ of both sides of~\eqref{eq:def_V}.
 As the pull-back operator is linear and commutes with the exterior derivative,
 we obtain
 \[
  \begin{split}
    (\Phi_Q^{-1})^* V 
    &= (\Phi_Q^{-1})^* \d\left((\partial_t u, \, iu)_g \, \d t\right)
     + (\Phi_Q^{-1})^*\d_t j(u) \\
    &= \d\left((\Phi_Q^{-1})^*((\partial_t u, \, iu)_g \, \d t)\right)
     + \d_t \left((\Phi_Q^{-1})^*j(u)\right) \\
  \end{split}
 \]
 We apply~\eqref{plbk13} and~\eqref{plbk15} to the right-hand side:
 \begin{equation} \label{plbk22}
  \begin{split}
    (\Phi_Q^{-1})^* V 
    &= \d \left((\partial_t v, \, iv) \, \d t\right)
     + \d_t \bar{\jmath}(v) 
     - \d_t\left(\abs{v}^2 (\Phi_Q^{-1})^*\mathcal{A}\right) 
    \stackrel{\eqref{eq:flat_V}}{=} \bar{V} 
     - \d_t\left(\abs{v}^2 (\Phi_Q^{-1})^*\mathcal{A}\right) 
  \end{split}
 \end{equation}
 As the form~$(\Phi_Q^{-1})^*\mathcal{A}$ only depends 
 on the variables in~$M$, not on~$t$, we have
 \begin{equation} \label{plbk23}
  - \d_t\left(\abs{v}^2 (\Phi_Q^{-1})^*\mathcal{A}\right)
  = \d_t\left((1 - \abs{v}^2) (\Phi_Q^{-1})^*\mathcal{A}\right)
  = \d_t\mathfrak{R}
 \end{equation}
 Equation~\eqref{pullbackV} follows from~\eqref{plbk22} and~\eqref{plbk23}.
 
 \medskip
 \noindent
 \emph{Proof of~\eqref{pullbackomega}.}
 We take the exterior differential~$\d_Q$ of both sides of~\eqref{eq:pull_back}:
 \begin{equation*}
  \begin{split}
   (\Phi_Q^{-1})^* \d_Q j_Q(u) = \d\left((\partial_t v, iv) \, \d t\right) 
    + \d_t\bar{\jmath}(v) + \d\bar{\jmath}(v)
    - \d_t\left(\abs{v}^2(\Phi_Q^{-1})^*\mathcal{A}\right)
    - \d\left(\abs{v}^2(\Phi_Q^{-1})^*\mathcal{A}\right)
  \end{split}
 \end{equation*}
 The right-hand side may be further simplified
 by applying~\eqref{eq:def_V} and~\eqref{plbk22}:
 \begin{equation} \label{plbk31}
  \begin{split}
   (\Phi_Q^{-1})^* \d_Q j_Q(u) 
   &= \bar{V} + \d\bar{\jmath}(v)
    - \d_t\left(\abs{v}^2(\Phi_Q^{-1})^*\mathcal{A}\right)
    - \d\left(\abs{v}^2(\Phi_Q^{-1})^*\mathcal{A}\right) \\
   &= (\Phi_Q^{-1})^* V + \d\bar{\jmath}(v)
    - \d\left(\abs{v}^2(\Phi_Q^{-1})^*\mathcal{A}\right) 
  \end{split}
 \end{equation}
 On the other hand, by taking the pull-back via~$\Phi_Q^{-1}$
 in both sides of~\eqref{eq:cylindrical_vorticity}, we obtain
 \begin{equation} \label{plbk32}
  (\Phi_Q^{-1})^*\d_{Q} \, j_Q(u) 
  = (\Phi_Q^{-1})^*V + (\Phi_Q^{-1})^*\omega(u) 
  - (\Phi_Q^{-1})^*(\kappa\Vg)
 \end{equation}
 Comparing~\eqref{plbk31} with~\eqref{plbk32} gives
 \begin{equation*}
  \begin{split}
   (\Phi_Q^{-1})^*\omega(u)  
   = \d\bar{\jmath}(v)
    - \d\left(\abs{v}^2(\Phi_Q^{-1})^*\mathcal{A}\right)
    + (\Phi_Q^{-1})^*(\kappa\Vg)
  \end{split}
 \end{equation*}
 and hence, recalling that~$\d\mathcal{A} = \kappa\Vg$,
 \begin{equation*}
  \begin{split}
   (\Phi_Q^{-1})^*\omega(u)  
   &= \d\bar{\jmath}(v)
    - \d\left(\abs{v}^2\right) \wedge (\Phi_Q^{-1})^*\mathcal{A}
    + \left(1 - \abs{v}^2\right) \d\left((\Phi_Q^{-1})^*\mathcal{A}\right) \\
   &= \d\bar{\jmath}(v)
    + \d\left(1 - \abs{v}^2\right) \wedge (\Phi_Q^{-1})^*\mathcal{A}
    + \left(1 - \abs{v}^2\right) \d\left((\Phi_Q^{-1})^*\mathcal{A}\right) \\
   &=\d\bar{\jmath}(v) + \d\mathfrak{R}
  \end{split}
 \end{equation*}
 
 \medskip
 \noindent
 \emph{Proof of~\eqref{pullbackJ}.}
 We take the pull-back in Equation~\eqref{eq:defJ}
 and apply~\eqref{pullbackV}, \eqref{pullbackomega}:
 \[
  \begin{split}
   (\Phi_Q^{-1})^* J(u) = (\Phi_Q^{-1})^* V 
    + (\Phi_Q^{-1})^* \omega(u)
    = \bar{V} + \d_t\mathfrak{R}
     + \bar{\jmath}(v) + \d\mathfrak{R}
  \end{split}
 \]
 Equation~\eqref{pullbackJ} follows, thanks 
 to~\eqref{eq:flat_vorticity}.
\end{proof}

In the next Proposition $\mathscr{M}_2(M)$ will denote the space of $2$-forms whose coefficients are bounded Radon measures on $M$. $\mathscr{M}_2(M)$ is endowed with the flat norm (i.e., the $W^{-1,1}$-norm).

\begin{prop}[Space-Time Compactness]
\label{prop:compactness}
Given $u_\eps:[0, \, T]\times M\to TM$ such that 
\begin{align}
& \max_{[0, \, T]\times M}\vert u_\eps\vert_g \le C,\label{eq:max_princ}\\
& F_\eps(u_\eps) = \frac{1}{2}\int_{M} \vert Du_\eps\vert^2_g + \frac{1}{2\eps^2}(1 -\vert u_\eps\vert^2_g)^2 \, \Vg \le C\vert \log\eps\vert, \,\,\,\,\forall t\in [0,T],\label{eq:ipo_comp}\\
&\int_{Q}\vert \partial_t u_\eps\vert^2_g \, \Vg \, \d t\le C\vert \log\eps\vert,\label{eq:kinetic_th}
\end{align}  
for some positive constant $C$ independent of $\eps$, we let $V_\eps$ be defined, starting from $u_\eps$, as in \eqref{eq:def_V} and let $\omega(u_\eps):=\d j(u_\eps) + \kappa \Vg$. 
Then, there exist $J\in L^2(0,T;\mathscr{M}_2(M))$,
$V\in L^2(0,T;\mathscr{M}_2(M))$ and $\mu\in C^{0,1/2}([0,T];\mathscr{M}_2(M))$ such that (up to subsequences)
\begin{align}
& J(u_\eps) \xrightarrow{\eps\to 0}J 
  \qquad \hbox{ in } W^{-1,p}(Q) \quad
  \textrm{for any } p\in (1, \, 2) \label{eq:comp_Jac}\\
& \omega(u_\eps(t)) \xrightarrow{\eps\to 0} \mu(t)
  \qquad \hbox{ in } W^{-1,p}(M) \quad
 \textrm{for any } p\in (1, \, 2) 
 \textrm{ and } t\in [0, \, T] \label{eq:convforanyt}\\
& V_\eps \xrightarrow{\eps\to 0} V 
 \qquad \hbox{ in } W^{-1,p}(Q) \quad
 \textrm{for any } p\in (1, \, 2) \label{eq:comp_time}
\end{align}
and
\begin{equation}
\label{eq:jacobian_limit}
J = \mu + V.
\end{equation}
Moreover, the measure $\mu$ has the form
\begin{equation}
\label{eq:sum_dirac}
\mu(t) = 2\pi \sum_{j=i}^{n(t)}d_j(t) \delta_{a_j(t)},
\end{equation}
where $n:[0,T]\to \mathbb{N}$ and~$d_j(t)\in \mathbb{Z}$,
$a_j(t)\in M$ for all~$t\in [0,T]$, $j=1,\ldots,n(t)$.
\end{prop}
\begin{proof}
 We split the proof into steps.

\setcounter{step}{0}
\begin{step}[Localization]
 Let~$\delta > 0$ be a small parameter.
 Let $\mathfrak{B}=\left\{B^1_{\delta},\ldots,B^N_{\delta}\right\}$ 
 be a finite open cover of~$M$  by geodesic balls of radius~$\delta$.
 Each ball~$B^k_\delta$ is equipped with normal geodesic coordinates
 $\Phi^k\colon B^k_\delta\to B_\delta(0)\subseteq\R^2$.
 The cover~$\mathfrak{B}$ induces a finite cover of~$Q = [0, \, T]\times M$,
 \[
  [0, \, T]\times\mathfrak{B} :=
  \left\{[0, \, T]\times B^1_{\delta}, \, \ldots, \,
  [0, \, T]\times B^N_{\delta}\right\}
 \]
 and~$[0, \, T]\times B^k_\delta$ is equipped with
 coordinates~$\Phi^k_Q = (\Id, \, \Phi^k)$. For any~$k$,
 we choose an orthonormal, positively oriented,
 tangent frame~$\{\tau_1^k, \, \tau_2^k\}$ defined on~$B^j_\delta$.
 We represent~${u_\eps}_{|[0, T]\times B^k_\delta}$ by a
 map~$v_\eps^k\colon [0, \, T]\times B_\delta(0)\to\R^2$,
 as in~\eqref{eq:normal_representation_space_time}.
 The assumptions~\eqref{eq:ipo_comp} and~\eqref{eq:kinetic_th}, 
 combined with~\eqref{eq:energy_normal_coordinates}
 and~\eqref{eq:partialt_normal_coordinates}, imply
 \begin{align}
  & \frac{1}{2}\int_{B_\delta(0)} \left(|\nabla v_\eps^k|^2
   + \frac{1}{2\eps^2}(1 - |v_\eps^k|^2)^2 \right) \d x
   \le C\abs{\log\eps} \qquad \textrm{for all } t\in [0, \, T], 
   \label{eq:energy_est_normal_coord} \\
  &\int_{[0, T]\times B_\delta(0)}
   |\partial_t v_\eps^k |^2 \, \d x \, \d t
   \le C\abs{\log\eps}
 \end{align}  
 Let~$\bar{\jmath}(v_\eps^k)$, $\bar{V}_\eps^k$ and~$\bar{J}(v_\eps^k)$
 be defined as in~\eqref{eq:flat_space_current}, \eqref{eq:flat_V}
 and~\eqref{eq:flat_vorticity}, respectively.
 We can now apply the Jerrard-Soner compactness result for Jacobians 
 \cite[Theorem 3.1]{JerrardSoner-jacobians} 
 and the Sandier-Serfaty product estimates 
 \cite[Theorem 3]{SS-product}). We find a (non-relabelled) subsequence and
 \[
  \bar{\mu}^k\in L^\infty(0, \, T; \, \mathscr{M}_2(B_\delta(0))),
  \quad \bar{V}^k\in L^2(0, \, T; \, \mathscr{M}_2(B_\delta(0)))
  \qquad \textrm{for } k=1, \, \ldots, \, N
 \]
 such that for any~$k=1, \, \ldots, \, N$ and~$p\in (1, \, 2)$,
 \begin{align} 
  \d\bar{\jmath}(v_\eps^k) \xrightarrow{\eps\to 0} \bar{\mu}^k, \quad
   \bar{V}_\eps^k \xrightarrow{\eps\to 0} \bar{V}^k
   \qquad &\textrm{in } W^{-1,p}([0, \, T]\times B_\delta(0)) \label{compk} \\
  \d\bar{\jmath}(v_\eps^k(t)) \xrightarrow{\eps\to 0} \bar{\mu}^k(t)
   \qquad &\textrm{in } W^{-1,p}(B_\delta(0))  
   \quad \textrm{for any } t\in [0, \, T] \label{compk-t}
 \end{align}  
\end{step}

\begin{step}
 For any~$k$, let~$\mathcal{A}^k$ be the connection~$1$-form
 induced by the choice of the frame~$\{\tau_1^k, \, \tau_2^k\}$.
 Let
 \[
  \mathfrak{R}^k_\eps := \left(1 - |v_\eps^k|^2\right) 
   (\Phi_Q^{-1})^*\mathcal{A}^k
 \]
 The energy estimate~\eqref{eq:energy_est_normal_coord} implies that
 \begin{equation} \label{compstR}
  \sup_{t\in [0, \, T]} \|\mathfrak{R}^k_\eps\|_{L^2(B_\delta(0))}
  \lesssim \sup_{t\in [0, \, T]} \|1 - |v_\eps^k|\|_{L^2(B_\delta(0))}
  \lesssim \eps \abs{\log\eps}
 \end{equation}
 Therefore, for any~$k$ we have
 \begin{align} 
  \d_t\mathfrak{R}^k_\eps \xrightarrow{\eps\to 0} 0
   \qquad &\textrm{in } W^{-1,2}([0, \, T]\times B_\delta(0)),
   \label{compact_space_time-R} \\
  \d\mathfrak{R}^k_\eps \xrightarrow{\eps\to 0} 0
  \qquad &\textrm{in } L^\infty(0, \, T; \, W^{-1,2}(B_\delta(0))),
  \label{compact_space_time-Rt}
 \end{align}
 We combine~\eqref{compk} with~\eqref{compact_space_time-R}, 
 \eqref{compact_space_time-Rt} and Lemma~\ref{lemma:pullback} 
 to obtain
 \[
  J(u_\eps)_{|[0, T]\times B^k_\delta} = \Phi_Q^*\left(\bar{J}(v_\eps^k)
   + \d_t\mathfrak{R}_\eps^k + \d\mathfrak{R}_\eps^k\right)
  \xrightarrow{\eps\to 0} \Phi_Q^*\left(\bar{V}^k + \bar{\mu}^k\right) 
  =: J^k
 \]
 in~$W^{-1,p}([0, \, T]\times B^k_\delta)$, for any~$p\in (1, \, 2)$
 and~$k=1, \, \ldots, \, N$. (A duality argument shows that the
 pull-back via a smooth map preserves the convergence in~$W^{-1,p}$.)
 Now, let $\left\{\eta_1^2,\ldots,\eta_N^2\right\}$
 be a partition of unity subordinate to~$\mathfrak{B}$. 
 We have $\hbox{supp}\,(\eta_k^2)\subset B^k_{\delta}$ 
 for every $k=1,\ldots, N$ and hence,
 \[
  J(u_\eps) = \sum_{k=1}^N \eta_{k}^2 \, J(u_\eps)
  \xrightarrow{\eps\to 0} \sum_{k=1}^N \eta_{k}^2 \, J^k =: J
 \]
 in~$W^{-1,p}([0, \, T]\times B^k_\delta)$ for any~$p\in (1, \, 2)$.
 This proves~\eqref{eq:comp_Jac}. The proof of~\eqref{eq:convforanyt}
 and~\eqref{eq:comp_time} is analogous. 
 Since $J(u_\eps) = \omega(u_\eps) + V_\eps$ (see~\eqref{eq:defJ}),
 Equation~\eqref{eq:jacobian_limit} follows in the limit.
\end{step}

\begin{step}
 By combining Equation~\eqref{eq:cylindrical_vorticity}
 with~\eqref{eq:def_V}, we can write
 \begin{equation} \label{comp-transp1}
  \d_Q j_Q(u_\eps) = V_\eps + \d j(u_\eps)
 \end{equation}
 The smoothness of~$u_\eps$ is smooth, and hence of~$\d_Q j_Q(u_\eps)$,
 implies that $\d_{Q}(\d_Qj_Q(u_\eps)) = 0$, 
 $\d(\d j(u_\eps)) = 0$ and, by differentiating~\eqref{eq:def_V},
 $\d_t V_\eps = 0$. By taking the exterior differential~$\d_Q$
 of both sides of~\eqref{comp-transp1}, we obtain
 \[
  0 = \d V_\eps + \d_t \d j(u_\eps) 
  = \d V_\eps + \d_t \left(\omega(u_\eps) - \kappa\Vg\right)
  = \d V_\eps + \d_t\,\omega(u_\eps) 
 \]
 Therefore, by taking the limit as~$\eps\to 0$,
 we obtain
 \begin{equation}
  \label{eq:transport}
  \partial_t \mu + \d V = 0
  \qquad \hbox{in the sense of distributions. }
 \end{equation}
\end{step}

\begin{step}
 The $L^2$-in-time regularity of $V$ gives, exactly as in the Euclidean case (with similar proof) the time regularity of $\mu$, namely that 
$\mu\in C^{0,1/2}([0,T];\mathscr{M}_2(M))$. 
More precisely, 
using the transport equation \eqref{eq:transport} we obtain that (up to considering smooth approximations of $V$ and $\mu$, see \cite[Theorem 3]{SS-product}), for any smooth $\psi:M\to \mathbb{R}$ and any $0\le t_1<t_2<T$
\begin{equation}
\label{eq:trasporto_2}
\int_{[t_1,t_2]\times M} V\wedge \d_Q\psi = -\int_{[t_1,t_2]\times M} \d V\wedge \psi = \int_{[t_1,t_2]\times M}\d_t (\mu\wedge \psi) = \int_{M}(\mu(t_2)-\mu(t_1))\psi. 
\end{equation}
Therefore, since $V\in L^2(0,T;\mathscr{M}_2(M))$, for any smooth $\psi\in C^{0,1}(M)$ with $\|\nabla\psi\|_{L^\infty} \le 1$ in $M$ we have 
\begin{equation}
\label{eq:holdermu}
\abs{\int_{M}(\mu(t_2)-\mu(t_1))\psi} \le \|V\|_{L^2(0,T;\mathscr{M}_2(M))}\sqrt{t_2 -t_1}.
\end{equation}
Thus, 
\[
\norm{\mu(t_2)-\mu(t_1)}_{\text{flat}}\le \|V\|_{L^2(0,T;\mathscr{M}_2(M))}\sqrt{t_2 -t_1},
\]
namely $\mu\in C^{0,1/2}([0,T],\mathscr{M}_2(M))$. In particular, \eqref{eq:holdermu} gives that $t\mapsto \int_{M}\mu(t) \psi$ is in $C^{0,1/2}(0,T)$ for any $\psi\in C^{0,1}(M)$ (with $\|\nabla \psi\|_{L^\infty} \le 1$). 
As a result, we have \eqref{eq:convforanyt}. Moreover 
thanks to \eqref{eq:ipo_comp}, we apply for any fixed $t\in (0,T)$  
Proposition \ref{prop:comp_prod_est_space} and we conclude that 
the measure $\mu$ has the form 
\eqref{eq:sum_dirac}. \qedhere
\end{step}
\end{proof} 

\begin{prop}[Space-Time Product Estimates]
\label{th:comp_prod_est}
 Let~$u_\eps\colon[0, \, T]\times M\to TM$ be
 a sequence that satisfies~\eqref{eq:max_princ}, 
 \eqref{eq:ipo_comp}, \eqref{eq:kinetic_th}. 
 Let~$J\in L^2(0, \, T; \, \mathscr{M}_2(M))$ be given by
 \eqref{eq:jacobian_limit}. Then, 
 for any continuous vector fields $X = X^k\frac{\partial}{\partial x^k}$ 
 and~$Y = f\frac{\partial}{\partial t}$ on~$Q$, there holds
 \begin{equation} \label{eq:product_curvo}
 \liminf_{\eps\to 0}\frac{1}{\abs{\log\eps}} \left(\int_{Q}\vert D_{X}u_\eps\vert_g^2 \Vg\d t\right)^{1/2}\left(\int_{Q}f^2\vert\partial_t u_\eps\vert^2_{g}\Vg\d t\right)^{1/2}\ge\abs{\frac{1}{2}\int_{Q}J[X,Y]}.
 \end{equation}
\end{prop}
\begin{proof}
 The proof again uses a localization argument and the analogous lower bound 
 in the Euclidean case. The details are as follows. 
 
 As in Proposition~\ref{prop:compactness}, we consider a finite cover
 $\mathfrak{B}=\{B_\delta^1, \, \ldots, \, B_\delta^N\}$ by geodesic balls
 and a partition of the unity $\{\eta_1^2, \, \ldots, \, \eta_N^2\}$
 subordined to~$\mathfrak{B}$. We represent~${u_\eps}_{|[0, T]\times B^k_\delta}$
 by a map~$v_\eps^k\colon [0, \, T]\times B_\delta(0)\to\R^2$,
 as in~\eqref{eq:normal_representation_space_time}.
 By equation~\eqref{pullbackJ} (Lemma~\ref{lemma:pullback}),
 we have
 \begin{equation} \label{pe1}
  (\Phi_{k,Q}^{-1})^*J(u_\eps) = \bar{J}(v_\eps^k) 
   + \d_t\mathfrak{R}_\eps^k +\d\mathfrak{R}_\eps^k 
 \end{equation}
 where~$\Phi_{k,Q}\colon [0, \, T]\times B^k_\delta\to [0, \, T]\to B_\delta(0)$
 is the coordinate map and
 \begin{equation} \label{pe2}
  \mathfrak{R}_\eps^k := \left(1 - |v_\eps^k|^2\right)
   (\Phi_{k,Q}^{-1})^*\mathcal{A}
  \xrightarrow{\eps\to 0} 0 \qquad \textrm{in } L^2([0, \, t]\times B_\delta(0))
 \end{equation}
 as in~\eqref{compstR}. By~\eqref{eq:comp_Jac} and~\eqref{compk}, 
 we may extract a subsequence in such a way that $J(u_\eps)\to J$
 in~$W^{-1,p}(Q)$ and~$\bar{J}(v_\eps)\to\bar{J}^k$ 
 in~$W^{-1,p}([0, \, T]\times B_\delta(0))$, for any~$k$.
 Due to~\eqref{pe1}, \eqref{pe2}, we must have
 \begin{equation}
  (\Phi_{k,Q}^{-1})^*J = \bar{J}^k 
 \end{equation}
 Then, for any continuous fields $X = X^k\frac{\partial}{\partial x^k}$ 
 and~$Y = f\frac{\partial}{\partial t}$ on~$Q$,
 \begin{equation} \label{eq:normal_rhs}
  \begin{split}
   \int_Q J[X, \, Y] \, \Vg \, \d t
   &= \sum_{k=1}^{N} \int_{[0, \, T]\times B^k_{\delta}} 
    \eta_k^2 \, J[X, \, Y] \, \Vg \, \d t \\
   &= \sum_{k=1}^{N} \int_{[0, \, T]\times B^k_{\delta}} 
    \bar{\eta}_k^2 \, \bar{J}^k[\bar{X}_k, \, \bar{Y}_k] \, \Vg \, \d t
  \end{split}
 \end{equation}
 where~$\bar{\eta}_k := \eta_k\circ\Phi_k^{-1}$ and
 $\bar{X}_k$, $\bar{Y}_k$ are the push-forward of~$X$, $Y$ 
 via the coordinate map~$\Phi_{k,Q}$.

 Let~$\lambda_\eps$ be any sequence of positive numbers.
 We have
 \begin{equation} \label{eq:product_1}
  \begin{split}
   &\frac{1}{2}\int_Q \lambda_\eps\vert D_X u_\eps\vert^2_g +\frac{1}{\lambda_\eps}\vert f\partial_t u_\eps\vert^2_g \, \Vg \, \d t \\
   &\qquad\qquad\qquad = \sum_{k=1}^{N}\frac{1}{2} 
   \int_{[0, \, T]\times B^k_{\delta}} \lambda_\eps\vert\eta_k D_X u_\eps\vert^2_g +\frac{1}{\lambda_\eps}\vert \eta_k f\partial_t u_\eps\vert^2_g \, \Vg \, \d t
  \end{split}
 \end{equation}
 For any~$k$, we have (see~\eqref{eq:energy_dens_normal})
 \begin{equation} \label{eq:local_product}
  \begin{split}
   &\frac{1}{2} \int_{[0, \, T]\times B^k_{\delta}} \lambda_\eps\vert\eta_k D_X u_\eps\vert^2_g +\frac{1}{\lambda_\eps}\vert \eta_k f\partial_t u_\eps\vert^2_g \, \Vg \, \d t \\
   &\qquad\qquad = \frac{1}{2} \left(1+\O(\delta^2)\right)
   \int_{[0, \, T]\times\bar{B}_{\delta}}\lambda_\eps\vert\bar{\eta}_k
   \nabla_{\bar{X}_k} v_\eps\vert^2 
   + \frac{1}{\lambda_\eps}\vert \bar{\eta}_k \bar{f}_k\partial_t v_\eps\vert^2 
   \, \d x \, \d t \\
   &\qquad\qquad\qquad + \left(1+\O(\delta^2)\right)
    \int_{[0, \, T]\times\bar{B}_{\delta}} \O(\vert v_\eps\vert^2) \, \d x \, \d t
  \end{split}
 \end{equation}
 where~$\bar{f}_k := \eta_k\circ\Phi_{k,Q}^{-1}$. The last integral
 in the right-hand side of~\eqref{eq:local_product} is~$\O(\delta^2)$,
 because of the assumption~\eqref{eq:max_princ}.
 Now, the product estimate of Sandier \& Serfaty
 \cite[Theorem 3]{SS-product}, applied on~$[0, \, T]\times B_\delta(0)$,
 implies
 \begin{equation}\label{eq:flat_product}
  \begin{split}
   &\liminf_{\eps\to 0}\frac{1}{2\abs{\log\eps}}
    \int_{[0, \, T]\times\bar{B}_{\delta}} \lambda_\eps\vert\bar{\eta}_k
    \nabla_{\bar{X}_k} v_\eps\vert^2 
    + \frac{1}{\lambda_\eps}\vert \bar{\eta}_k \bar{f}_k\partial_t v_\eps\vert^2 
    \, \d x \, \d t\\
   &\geq \liminf_{\eps\to 0}\frac{1}{\abs{\log\eps}}
   \left(\int_{[0, \, T]\times B_\delta(0)} \vert \bar{\eta}_k 
    \nabla_{\bar{X}_k} v_\eps\vert^2 \, \d x \, \d t\right)^{1/2}
   \left(\int_{[0, \, T]\times B_\delta(0)} \vert \bar{\eta}_k
   \bar{f}_k \partial_t v_\eps\vert^2 \, \d x \, \d t\right)^{1/2} \\
   &\geq \frac{1}{2}\int_{[0,T]\times \bar{B}_{\delta}}
    \bar{\eta}^2_k \, \bar{J}[\bar{X}_k, \, \bar{Y}_k] \, \Vg \, \d t 
  \end{split}
 \end{equation}
 As a result, combining \eqref{eq:flat_product}, 
 \eqref{eq:local_product} and \eqref{eq:product_1} we obtain 
 \begin{equation} \label{eq:local_lower_bound}
  \begin{split}
   &\liminf_{\eps\to 0}\frac{1}{2\abs{\log\eps}}
   \int_{[0, \, T]\times B^k_{\delta}}\lambda_\eps\vert\eta_k D_X u_\eps\vert^2_g 
   +\frac{1}{\lambda_\eps}\vert \eta_k f\partial_t u_\eps\vert^2_g \, \Vg \, \d t \\
   &\qquad\qquad \ge \left(1+ \O(\delta^2)\right)
   \int_{[0, \, T]\times B^k_{\delta}} \frac{1}{2}\eta_k^2 \, J[X, \, Y] \, \Vg \, \d t.
  \end{split}
 \end{equation}
 We sum over~$k$, apply~\eqref{eq:normal_rhs} and let~$\delta\to 0$.
 We obtain
 \begin{equation*}
  \begin{split}
   \liminf_{\eps\to 0}
   \frac{1}{2\abs{\log\eps}}\int_Q\lambda_\eps\vert D_X u_\eps\vert^2_g +\frac{1}{\lambda_\eps}\vert f\partial_t u_\eps\vert^2_g \, \Vg \, \d t\ge \frac{1}{2}\int_Q J[X, \, Y] \, \Vg \, \d t
  \end{split}
 \end{equation*}
 The same argument, applied with~$-X$ instead of~$X$, shows that
 \begin{equation} \label{eq:quasi_prod}
  \begin{split}
   \liminf_{\eps\to 0}
   \frac{1}{2\abs{\log\eps}}\int_Q\lambda_\eps\vert D_X u_\eps\vert^2_g 
   +\frac{1}{\lambda_\eps}\vert f\partial_t u_\eps\vert^2_g \, \Vg \, \d t
   \ge \abs{\frac{1}{2}\int_Q J[X, \, Y] \, \Vg \, \d t}
  \end{split}
 \end{equation}
 Finally, we take
 \[
  \lambda_\eps :=
  \left(\int_Q \vert D_X u_\eps\vert^2_g \, \Vg \, \d t\right)^{-1/2}
  \left(\int_Q \vert f\partial_t u_\eps\vert^2_g \, \Vg \, \d t\right)^{1/2}
 \]
 if both integrals are strictly positive; otherwise, 
 if either of the integrals is equal to zero,
 we choose~$\lambda_\eps>0$ in such a way that
 \[
  \int_Q\lambda_\eps\vert D_X u_\eps\vert^2_g
   +\frac{1}{\lambda_\eps}\vert f\partial_t u_\eps\vert^2_g \, \Vg \, \d t \leq \eps
 \]
 Then, \eqref{eq:product_curvo} follows from~\eqref{eq:quasi_prod}.
\end{proof}

\begin{prop}
\label{prop:kinetic}
Let $u_\eps:[0, \, T]\times M\to TM$ be a sequence of vector fields satysfying 
\begin{align*}
& \sup_{t\in [0,T]}F_\eps(u_\eps(t))\le C\vert \log\eps\vert,\\
\\
 &\int_{[0, \, T]\times M}\vert \partial_t u_\eps\vert^2_g \, \Vg \, \d t \le C\vert \log\eps\vert
\end{align*}
and 
\begin{equation}
\label{eq:ipovelocity2}
 \frac{1}{2} \sup_{t\in [0, \, T]} \int_{M}\vert D u_\eps(t)\vert^2_g \, \Vg \le \pi n\vert \log\eps\vert(1 + o_{\eps\to 0}(1)) 
\end{equation}
Assume that there exist a continuous curve $\a:[0,T]\to M^n$ and a vector 
$\db =(d_1,\ldots,d_n)$ with $d_j =\pm 1$ for $j=1,\ldots, n$ such that for any $t\in [0,T]$
\begin{equation}
\label{eq:ipovelocity1}
\omega(u_\eps(t))\to 2\pi \sum_{j=1}^{n}d_j \delta_{a_j(t)} \,\,\,\,\,\,\,\,\hbox{ in }\,\,(C^{0,\alpha}_{c}(M))',\,\,\,\,\,\forall \alpha\in (0,1).
\end{equation}
Let~$[t_1, \, t_2]\subset [0,T]$ be an interval such that $a_j(t)\neq a_k(t)$ for any~$t\in [t_1, \, t_2]$ and~$j\neq k$. Then,
we have ${\bf a}\in H^1(t_1, \, t_2; \, M^n)$ and 
\begin{equation}
\label{eq:velocity}
\liminf_{\eps\to 0}\frac{1}{\vert \log\eps\vert}\int_{t_1}^{t_2} \int_{M}\vert \partial_t u_\eps\vert^2_g \,  \Vg \, \d s\ge \pi \sum_{k=1}^n\int_{t_1}^{t_2} \vert a_k'\vert^2_{g}(s) \, \d s = \pi\int_{t_1}^{t_2}\vert {\bf a}'\vert^2_{g}(s) \, \d s.
\end{equation}
\end{prop}
\begin{proof}
The proof follows the one in the Euclidean case in \cite[Corollary 7]{SS-product}. 
Let $0\le t_1<t_2\le T$ as in the hypothesis, let $N\in \mathbb{N}$ and set $\tau:=\frac{t_2-t_1}{N}$ and $t^k:=t_1+\tau k$ for $k=0,1,\ldots,N$. 
In the whole interval $[0,T]$ there is a finite number of vortices. Moreover
in the interval $[t_1,t_2]$ 
the vortices $a_j$  ($j=1,\ldots,n$) remain distinct. Therefore,
taking~$N$ large enough,
for any interval $[t^{k},t^{k+1})$ we can find $n$ balls $B_j$ ($j=1,\ldots,n$) such that each of them contains exactly one vortice $a_j(t)$ for $t\in [t^{k},t^{k+1})$. 

Now, for any $k=1,\ldots, N$ and for any $j=1,\ldots, d$, we let $\psi\in C^1_{c}(B_j)$ with $\vert \nabla \psi\vert \le 1$. 
Then, \eqref{eq:product_curvo} with $X=\nabla^\perp \psi$, $Y = \frac{\partial}{\partial t}$, together with \eqref{eq:star_grad_ruotato} (with $\omega=V$) give 
\[
 \begin{split}
  &\liminf_{\eps\to 0}\left(\frac{1}{\vert \log\eps\vert^2}\int_{B_{j}\times[t^{k},t^{k+1}]}\vert D_{X}u_\eps\vert_g^2 \, \Vg \, \d t
  \int_{B_{j}\times[t^k,t^{k+1}]}\vert\partial_t u_\eps\vert^2_{g} \, \Vg \, \d t\right) \\
  &\hspace{4cm} 
  \ge\abs{\frac{1}{2}\int_{B_{j}\times [t^k,t^{k+1}]} V\wedge \d \psi}^2
 \end{split}
\]
Moreover, \eqref{eq:trasporto_2} gives that 
\[
\int_{B_{j}\times[t^{k},t^{k+1}]} V\wedge \d \psi = 2\pi d_j \left( \psi(a_j(t^{k+1}) -a_j(t^k)\right).
\]
Therefore, 
\[
 \begin{split}
  &\liminf_{\eps\to 0}\left(\frac{1}{\vert \log\eps\vert^2}\int_{B_{j}\times[t^k,t^{k+1}]}\vert D_{X}u_\eps\vert_g^2 \Vg\d t\int_{B_{j}\times[t^k,t^{k+1}]}\vert\partial_t u_\eps\vert^2_{g}\Vg\d t\right)\\
  &\hspace{4cm} \ge \pi^2\abs{\psi(a_j(t^{k+1})) -\psi_\delta(a_j(t^k))}^2.
 \end{split}
\]
Now, we apply Lemma \ref{lem:corollario4SS} (specifically \eqref{eq:polarization} with $X=Y=\nabla^\perp \psi_\delta$)
with $v_\eps = u_\eps(t)$ (with $t$ fixed in $[t^k,t^{k+1}]$) and we get 
\[
\lim_{\eps\to 0}\frac{1}{\vert \log\eps\vert}\int_{B_j}\vert D_X u_\eps(t)\vert^2_g\Vg = \pi \vert \nabla^\perp\psi(a_j(t))\vert^2_g=\pi \vert \nabla\psi(a_j(t))\vert^2_g,
\]
therefore, since $\vert \nabla \psi\vert\le 1$, 
\begin{align*}
&\liminf_{\eps\to 0}\frac{\tau\pi}{\vert \log\eps\vert}\int_{B_{j}\times[t^k,t^{k+1}]}\vert\partial_t u_\eps\vert^2_{g}\Vg\d t\ge\\
&\liminf_{\eps\to 0}\Big(\frac{1}{\vert \log\eps\vert^2}\int_{B_{j}\times[t^k,t^{k+1}]}\vert D_{X}u_\eps\vert_g^2 \Vg\d t\int_{B_{j}\times[t^k,t^{k+1}]}\vert\partial_t u_\eps\vert^2_{g}\Vg\d t\Big) \\
&\ge \pi^2\abs{\psi(a_j(t^{k+1})) -\psi(a_j(t^k))}^2.
\end{align*}
Thus, if we take the supremum over $\psi\in C^1_{c}(M)$ with $\norm{\nabla \psi}_{L^\infty}\le 1$, 
we get
\[
\liminf_{\eps\to 0}\frac{1}{\vert \log\eps\vert}\int_{B_{j}\times[t^k,t^{k+1}]}\vert\partial_t u_\eps\vert^2_{g}\Vg\d t \ge \pi\frac{\dist_g^2(a_j(t^{k+1}),a_j(t^k))}{\tau}, 
\]
and therefore, summing on $j=1,\ldots, n$, 
\begin{equation}
\label{eq:metric_deriv_1}
\liminf_{\eps\to 0}\frac{1}{\vert \log\eps\vert}\int_{\Omega\times[t^k,t^{k+1}]}\vert\partial_t u_\eps\vert^2_{g}\Vg\d t \ge \pi \sum_{j=1}^{n}\frac{\dist_g^2(a_j(t^{k+1}),a_j(t^k))}{\tau}.
\end{equation}
Given ${\bf a} = (a_1, \ldots,a_n)$, 
we define its piecewise constant interpolant $\bar{\a}_{\tau}$ by 
\[
\bar{{\bf a}}_{\tau}(t) := {\bf a}(t^{k+1})\,\,\,\,\,\,\,\hbox{ for }t \in [t^{k},t^{k+1}).
\]
Note that since ${\bf a}$ is continuous we have that $\bar{{\bf a}}_\tau\xrightarrow{\tau \to 0}a$ uniformly in $[t_1,t_2]$. 
Moreover, we define
\[
\abs{\bar{{\bf a}}_{\tau}'}_g(t):= \left(\sum_{j=1}^n\frac{\dist_g^2(a_j(t^{k+1}),a_j(t^{k}))}{\tau^2}\right)^{1/2},\,\,\,\,\hbox{ for }t\in [t^{k},t^{k+1}).
\]
Note that, as observed in \cite[Eq. 2.3.4.]{Amb-Gi-Sav}, the notation is justified since when one reduces to a linear framework, e.g. $M=\mathbb{R}^{2n}$, then $\abs{\bar{\a}_{\tau}'}_g(t)$ is the norm of the derivative of the piecewise interpolation of the values $\a(t^k)$, for $k=1,\ldots, N$. 

We set 
\[
L:= \liminf_{\eps\to 0}\frac{1}{\pi \vert \log\eps\vert}\int_{\Omega\times[t_1,t_2]}\vert\partial_t u_\eps\vert^2_{g}\Vg\d t,
\]
and thus \eqref{eq:metric_deriv_1} rewrites as 
\[
\int_{t_1}^{t_2}\vert \bar{{\bf a}}_{\tau}'\vert^2(t)\d t \le L. 
\]
We assume that $L<+\infty$ (otherwise there is nothing to prove). Therefore, there exists $A\in L^2(t_1,t_2)$ and a subsequence of $\tau$ (not relabelled) such that 
\[
\vert \bar{{\bf a}}_{\tau}'\vert \xrightarrow{\tau \to 0}A\,\,\,\,\,\hbox{ weakly in } L^2(t_1,t_2),
\]
and thus 
\begin{equation*}
\int_{t_1}^{t_2}\vert A\vert^2 \d t \le L.
\end{equation*}
Now, we show that ${\bf a}\in H^1(t_1,t_2;M)$ with metric derivative $\vert a'\vert_g(t)$ such that 
\[
\vert {\bf a}'\vert_g(t) \le A(t)\,\,\,\,\hbox{ for a.a. }t\in (t_1,t_2). 
\]
We fix $t_1 \le s<t\le t_2$ and we set
\[
\overline{s}_\tau:= \max\left\{t^k\colon t^k\le s\right\}\,\,\,\,\,\,\hbox{ and }\,\,\,\,\,\,\,\underline{t}_\tau:= \min\left\{t^k\colon t^k\ge t\right\}.
\]
Note that $\overline{s}_\tau\xrightarrow{\tau \to 0}s$ and $\underline{t}_\tau\xrightarrow{\tau \to 0}t$.
Then, the triangle inequality gives 
\[
\sum_{j=1}^n \dist_g(a_j(\overline{s}_\tau),a_j(\underline{t}_\tau))\le \int_{\overline{s}_\tau}^{\underline{t}_\tau}\vert \bar {\bf a}'\vert_g(r) \d r.
\]
Therefore, if we let $\tau\searrow 0$, we get 
\[
\sum_{j=1}^n \dist_g(a_j(s),a_j(t)) \le\int_{s}^t A(r)\d r, 
\]
and thus since $A\in L^2(t_1,t_2)$, thanks to Proposition \ref{prop:metric_derivative} we get that ${\bf a}\in H^1(t_1,t_2;M)$
and
\[
\vert {\bf a}'\vert_g(t) \le A(t)\,\,\,\,\,\hbox{ for a.a. } t\in (t_1,t_2).
\]
Therefore, we conclude
\[
\pi \int_{t_1}^{t_2}\vert {\bf a}'\vert_g^2(t)\d t \le \liminf_{\eps\to 0}\frac{1}{\vert \log\eps\vert}\int_{\Omega\times[t_1,t_2]}\vert\partial_t u_\eps\vert^2_{g}\Vg\d t. \qedhere
\]
\end{proof}

Given a~$1$-form $\omega\in L^2(M, \,  \T^*M)$, we denote by~$\P\omega$
the $L^2$-orthogonal projection of~$\omega$ onto the space of
harmonic $1$-forms~$\Harm^1(M)$. Equivalently, $\P\omega$ is the unique
element of~$\Harm^1(M)$ such that
\begin{equation} \label{harmproj}
 \int_M\left(\omega - \P\omega, \, \eta\right)_g \, \Vg = 0
 \qquad \textrm{ for any } \eta\in\Harm^1(M)
\end{equation}
Since the space~$\Harm^1(M)$ is finite-dimensional, we can choose
an orthonormal basis~$(\eta_\ell)_{\ell=1}^{2\gf}$ of~$\Harm^1(M)$
and write~$\P\omega$ as
\begin{equation} \label{harmproj-basis}
 \P\omega = \sum_{\ell=1}^{2\gf} \left( 
  \int_M(\omega, \, \eta_\ell)_g \, \Vg\right) \eta_\ell
 \qquad \textrm{for any } \omega\in L^2(M, \, \T^*M).
\end{equation}
The representation~\eqref{harmproj-basis} shows that~$\P$
extends to a linear, bounded operator~$L^1(M, \,  T'M)\to\Harm^1(M)$.

\begin{prop} \label{prop:harmproj}
 Let~$u_\eps:[0, \, T]\times M\to TM$ be a sequence 
 that satisfies the assumptions~\eqref{eq:max_princ},
 \eqref{eq:kinetic_th}, \eqref{eq:ipovelocity1} and
 \begin{equation} \label{eq:ipovelocity3}
  \sup_{t\in [0, \, T]} F_\eps(u_\eps(t)) \leq \pi n \abs{\log\eps} + C
 \end{equation}
 for some~$\eps$-independent constant~$C$.
 Then, there exists a map~$\xi\in H^1(0, \, T; \, \Harm^1(M))$ 
 and a (non-relabelled) subsequence such that
 \begin{equation} \label{harmproj-conv}
  \P j(u_\eps(t)) \xrightarrow{\eps\to 0} \xi(t) 
  \qquad \textrm{for any } t\in [0, \, T].
 \end{equation}
\end{prop}
As~$\Harm^1(M)$ is a finite dimensional space, the
convergence~\eqref{harmproj-conv} holds in any norm.
\begin{proof}[Proof of Proposition~\ref{prop:harmproj}]
 Due to the assumptions~\eqref{eq:max_princ}, \eqref{eq:ipovelocity1}
 and~\eqref{eq:ipovelocity3}, we can apply Corollary~\ref{cor:W1,pbounds}
 to the sequence~$u_\eps(t)$, at any fixed~$t\in [0, \, T]$. As
 a consequence, for any~$p\in (1, \, 2)$ there exists a constant~$C_p$
 such that
 \[
  \norm{\nabla u_\eps}_{L^\infty(0, \, T; \, L^p(M))} \leq C_p
 \]
 We deduce
 \[
  \norm{j(u_\eps)}_{L^\infty(0, \, T; \, L^p(M))}
  \leq \norm{u_\eps}_{L^\infty((0, \, T)\times M)}
   \norm{\nabla u_\eps}_{L^\infty(0, \, T; \, L^p(M))} \leq C_p
 \]
 and hence, we can extract a subsequence such that
 \[
  j(u_\eps) \rightharpoonup^* j \qquad \textrm{weakly}^*
  \quad \textrm{in } L^\infty(0, \, T; \, L^p(M))
 \]
 As the operator~$\P\colon L^1(M, \, T'M)\to\Harm^1(M)$
 is linear and bounded, we also have
 \begin{equation} \label{harmproj-convst}
  \P j(u_\eps) \rightharpoonup^* \P j =: \xi \qquad \textrm{weakly}^*
  \quad \textrm{in } L^\infty(0, \, T; \, \Harm^1(M))
 \end{equation}
 We claim that the limit~$\xi$ belongs to~$H^1(0, \, T; \, \Harm^1(M))$
 (and hence, is continuous in time).
 Let~$V_\eps$ be the $2$-form on~$Q := [0, \, T]\times M$ 
 associated with~$u_\eps$, as in~\eqref{eq:def_V}.
 Let~$\eta\in\Harm^1(M)$, and let~$\varphi\in C^\infty_{\mathrm{c}}(0, \, 1)$
 be a (scalar) test function. We have
 \begin{equation*} 
  \begin{split}
   I_\eps &:= \int_Q \varphi(t) \, \frac{\d}{\d t}
    (\P j(u_\eps), \, \eta)_g \, \Vg \, \d t
   \stackrel{\eqref{harmproj}}{=} \int_Q \varphi(t) \, \, \frac{\d}{\d t} 
    (j(u_\eps), \, \eta)_g \, \Vg \, \d t \\
   &= \int_Q \varphi(t) \, \, \d_t \left(j(u_\eps)\wedge \star\eta\right)
   = \int_Q \varphi(t) \, \, (\d_t j(u_\eps))\wedge \star\eta
  \end{split}
 \end{equation*}
 (we recall that~$\star$ denotes the Hodge dual operator in~$M$).
 Let~$\alpha_\eps := (\partial_t u_\eps, \, i u_\eps)_g$.
 By~\eqref{eq:def_V}, we know that
 \[
  V_\eps = \d(\alpha_\eps \, \d t) + \d_t j(u_\eps)
  = \d\alpha_\eps\wedge\d t + \d_t j(u_\eps)
 \]
 Therefore,
 \begin{equation*} 
  \begin{split}
   I_\eps  = \int_Q \varphi(t) \, \, V_\eps \wedge \star\eta
   - \int_Q \varphi(t) \, \, \d \alpha_\eps \wedge \d t \wedge \star\eta
  \end{split}
 \end{equation*}
 The form~$\eta$ is harmonic, hence co-closed, so~$\d(\star\eta) = 0$.
 This implies
 \[
  \d\alpha_\eps\wedge\d t\wedge\star\eta 
  = - \d\alpha_\eps\wedge\star\eta\wedge\d t
  = -\d(\alpha_\eps \star \eta) \wedge \d t
 \]
 and hence,
 \begin{equation} \label{harmproj-eq} 
  \begin{split}
   I_\eps  = \int_Q \varphi(t) \, \, V_\eps \wedge \star\eta
   + \int_0^T \varphi(t) \left( \int_M \d (\alpha_\eps \star\eta)\right) \d t 
   = \int_Q \varphi(t) \, \, V_\eps \wedge \star\eta
  \end{split}
 \end{equation}
 by Stokes theorem. Proposition~\ref{prop:compactness} states
 that there exists~$V\in L^2(0,T;\mathscr{M}_2(M))$ such that~$V_\eps\to V$
 in~$W^{-1,p}(Q)$, for any~$p\in (1, \, 2)$. On the other hand,
 the convergence~\eqref{harmproj-convst} implies
 $\partial_t \P j(u_\eps)\rightharpoonup^* \partial_t\xi$ 
 in~$\mathscr{D}^\prime(0, \, T; \, \Harm^1(M))$.
 Therefore, we can pass to the limit as~$\eps\to 0$ in~\eqref{harmproj-eq}:
 \begin{equation*}
  \begin{split}
   \int_Q \varphi(t) \, \frac{\d}{\d t} (\xi, \, \eta)_g \, \Vg \, \d t 
   = \int_Q \varphi(t) \, \, V \wedge \star\eta
  \end{split}
 \end{equation*}
 We deduce
 \begin{equation*}
  \begin{split}
   \int_Q \varphi(t) \, (\partial_t\xi, \, \eta)_g \, \Vg \, \d t 
   \leq \norm{V}_{L^2(0, \, T; \,\mathscr{M}_2(M))}
   \norm{\varphi}_{L^2(0, \, T)} \norm{\eta}_{L^\infty(M)}
  \end{split}
 \end{equation*}
 and, since~$\varphi$ and~$\eta$ are arbitrary, 
 $\partial_t\xi\in L^2(0, \, T; \, \Harm^1(M))$.
 
 It only remains to prove~\eqref{harmproj-conv}; we apply the arguments 
 from~\cite[Theorem 3]{SS-product}. Let~$t_0\in [0, \, T)$.
 Corollary~\ref{cor:W1,pbounds} implies, as above, that
 $\norm{j(u_\eps(t_0))}_{L^p(M)} \leq C_p$. Then,
 we can extract a subsequence (depending on~$t_0$)
 in such a way that $\P j(u_\eps(t_0))\to\zeta\in\Harm^1(M)$.
 We claim that~$\zeta = \xi(t_0)$. If we prove this claim,
 we will deduce that all the subsequences of~$\P j(u_\eps(t_0))$
 must converge to the same limit, and~\eqref{harmproj-conv} will follow.  
 Consider the sequence $\bar{u}_\eps\colon [-1, \, T]\times M \to \T M$,
 \[
  \bar{u}_\eps(t) := \begin{cases}
                      u_\eps(t_0) & \textrm{if } t \leq t_0 \\
                      u_\eps(t)   & \textrm{if } t > t_0
                     \end{cases}
 \]
 The sequence~$\bar{u}_\eps$ satisfies the
 assumptions~\eqref{eq:max_princ}, \eqref{eq:kinetic_th}
 and~\eqref{eq:ipovelocity3}. Proposition~\ref{prop:compactness}
 implies that $\omega(\bar{u}_\eps(t))\to\bar{\mu}(t)$ in~$W^{-1,p}(M)$
 for any~$t\in [-1, \, T]$, where
 \[
  \bar{\mu}(t) := \begin{cases}
                      2\pi \sum_{j=1}^n d_j \delta_{a_j(t_0)}
                       & \textrm{if } t \leq t_0 \\
                      2\pi \sum_{j=1}^n d_j \delta_{a_j(t)} 
                       & \textrm{if } t > t_0
                     \end{cases}
 \]
 Then, \eqref{eq:ipovelocity1} is satisfied, too.
 By the previous arguments, we can extract a subsequence
 in such a way that
 \[
  \P j(\bar{u}_\eps) \rightharpoonup^* \bar{\xi} \qquad 
  \textrm{weakly}^* \quad \textrm{in } L^\infty(0, \, T; \, \Harm^1(M)) 
 \]
 Using the definition of weak$^*$ convergence, we can check that
 \[
  \bar{\xi}(t) = \begin{cases}
                  \zeta  & \textrm{for a.e. } t < t_0 \\
                  \xi(t) & \textrm{for a.e. } t > t_0
                 \end{cases}
 \]
 However, the arguments above imply that both~$\xi$ and~$\bar{\xi}$
 are continuous functions of~$t$; therefore, we must have $\zeta=\xi(t_0)$,
 as claimed. In case~$t = T$, we define 
 $\bar{u}_\eps\colon [0, \, T+1]\times M \to \T M$,
 \[
  \bar{u}_\eps(t) := \begin{cases}
                      u_\eps(t) & \textrm{if } t < T \\
                      u_\eps(T) & \textrm{if } t \geq T
                     \end{cases}
 \]
 and proceed similarly.
\end{proof}

\subsection{Vortex dynamics: Proof of Theorem \ref{th:main1}}
\label{ssec:limit}
In this Subsection we finally prove our main result, Theorem \ref{th:main1}. 

The proof of Theorem \ref{th:main1} is based on the abstract scheme developed by Sandier \& Serfaty \cite[Theorem 1.4]{SS-GF}). The core of the Proof is contained in the next Propositions \ref{prop:main_limit_proc}, \ref{prop:liminfgrad} and \ref{prop:teorema}. 

We consider {\itshape well-prepared initial conditions}.
To ease the reading, we recall the definition.
Given~$n\in\Z$, $n\geq 1$, we 
consider~$(\a^0, \, \db)\in\mathscr{A}^n$ (see \eqref{admissible})
such that $d_j=\pm 1$ for any $j=1,\ldots, n$
and $\xi^0\in \mathcal{L}(\a^0, \, \db)$. 
We assume that the initial conditions
$u^0_{\eps}\in H^1_{\tang}(M)$ satisfy
\begin{align}
 & \omega (u^{0}_\eps)\xrightarrow{\eps \to 0} 2\pi \sum_{j=1}^n d_j \delta_{a_j^{0}} \qquad \hbox{ in } W^{-1,p}(M) \quad
 \textrm{for any } p\in (1, \, 2) \label{eq:initial_vorticity}\\
 & F_\eps(u_\eps^{0})\le \pi n\vert \log\eps\vert + W(\a^0, \, \db, \, \xi^0) + n\gamma + o(1)\label{eq:well_prepared}, \\
 &  \norm{u^0_\eps}_{L^\infty(M)} \leq 1.\label{eq:initial_Linfty}
\end{align} 

\medskip
By standard parabolic theory, we know that for any $\eps>0$ there exists a smooth solution $u_\eps$ of \eqref{eq:GL_intro}.

\begin{lemma}
\label{lem:lemma1}
Let $u_\eps$ be a solution of \eqref{eq:GL_intro} with $u_{\eps}^0$ satisfying 
\eqref{eq:well_prepared} and~\eqref{eq:initial_Linfty}.
Then, 
\begin{equation}
\label{eq:boundLinfty}
\| u_\eps\|_{L^\infty(M\times(0,T))}\le 1,
\end{equation}
\begin{equation}
\label{eq:energy_bound1_lemma}
F_\eps(u_\eps(t)) \le \pi n \vert \log\eps\vert + C, \,\,\,\,\hbox{ for a.a. }t\in [0,T].
\end{equation}
and   
\begin{equation}
\label{eq:boundut_lemma}
\int_{0}^{T} \int_{M}\vert \partial_t u_\eps\vert^2_g \, \Vg \d t \le C\vert\log\eps\vert.
\end{equation}
for some constant $C>0$ independent of $\eps$. 
\end{lemma}

\begin{proof}
The fact that $u_\eps$ verifies \eqref{eq:energy_bound1_lemma} is a consequence of the fact that
any solution of \eqref{eq:GL_intro} with the specified regularity satisfies the following energy estimate. 
\begin{equation}
\label{eq:energy_est}
\frac{1}{\vert \log\eps\vert}\int_{0}^T\int_{M}\vert\partial_t u_\eps\vert^2_g \, \Vg\d t + F_\eps(u_\eps(t)) = F_\eps(u_0)\le \pi n\vert \log\eps\vert + C,
\end{equation}
where $C$ denote a constant independent of $\eps$. 
Therefore \eqref{eq:energy_bound1_lemma} follows.
Set $v_\eps(x,t\vert \log\eps\vert) := u_\eps(x,t)$ for $(x,t)\in [0, \, T]\times M$. Then, $v_\eps$ solves 
\[
\begin{cases}
\displaystyle\partial_t v_\eps -\Delta_g v_\eps + \displaystyle\frac{1}{\eps^2}(\vert v_\eps\vert^2_g-1)v_\eps = 0,\,\,\,\,\hbox{ a.e. in }M\times (0,T)\\[.3cm]
\displaystyle v_\eps(x,0) = u_0\,\,\,\,\hbox{ a.e. in }M. 
\end{cases}
\]
Then, test (in the scalar product $(\cdot, \cdot)_g$) with $v_\eps$. 
We get 
\[
\partial_t \vert v_\eps\vert^2_g -\Delta_g \vert v_\eps\vert^2_g + 2 \vert \D v_\eps\vert^2_g + \frac{2}{\eps^2}\vert v_\eps\vert^2_g (\vert v_\eps\vert^2_g -1) =0. 
\] 
Therefore, if we set $w_\eps:=\vert v_\eps\vert^2_g -1$, we obtain that 
\[
\partial_t w_\eps -\Delta_g w_\eps + \frac{2}{\eps^2}w_\eps \le 0,
\]
and thus $w_\eps(t)\le w_\eps(0) \le 0$, namely $\vert v_\eps\vert_g \le 1$. 
By rescaling back with respect to time we obtain \eqref{eq:boundLinfty}.

Finally, we discuss the validity of \eqref{eq:boundut_lemma}. In particular, note that that it is sufficient to prove that there exists some $T_0>0$ such that 
\[
\int_{0}^{T_0}\int_{M}\vert \partial_t u_\eps\vert^2_g \, \Vg \d t \le \vert \log\eps\vert,
\]
but this has been already proved in \cite[Lemma 3.4]{SS-GF}. The argument uses a contradiction argument combined with the {\itshape product estimate} \eqref{eq:product_curvo} and can be replicated verbatim. 
%
\end{proof}

\begin{prop}
\label{prop:main_limit_proc}
Let $u_\eps $ be a solution of \eqref{eq:GL_intro} with $u_{\eps}^0$ satisfying 
\eqref{eq:initial_vorticity}, \eqref{eq:well_prepared} and~\eqref{eq:initial_Linfty}.
Then, there exist $T^*\in (0,T]$, a curve ${\bf a}:[0,T^*)\to M^n$
with ${\bf a}=(a_1,\ldots,a_n)\in H^1(0,T^*;M^n)$ integers~$\mathbf{d} = (d_1, \, \ldots, \, d_n)\in\Z^n$
such that~$(\a(t),\db)\in\mathscr{A}^n$ for any $t\in [0,T^*)$ and a curve $\xi\in H^1(0,T^*;\Harm^1(M))$ such that for any $t\in [0,T^*)$
\begin{align}
\label{eq:conv_vorticity_prop}
& \omega(u_\eps(t)) \to 2\pi\sum_{j=1}^n d_j\delta_{a_j(t)}\qquad \hbox{ in } W^{-1,p}(M) \quad
 \textrm{for any } p\in (1, \, 2),\\
& \P j(u_\eps(t)) \xrightarrow{\eps\to 0} \xi(t),\qquad\hbox{ with } \qquad\xi(t)\in \mathcal{L}(\a(t),\db),
\label{eq:convergence_xi}\\
 &\liminf_{\eps\to 0}\frac{1}{\vert \log\eps\vert}\int_{0}^t \int_{M}\vert \partial_t u_\eps\vert^2_g \,  \Vg\d t\ge \pi
  \int_{0}^t \vert {\bf a}'\vert^2_g(s) \, \d s,
  \label{eq:liminfut} 
 \end{align}
 Moreover, for almost any $t\in [0,T^*)$ there exists a subsequence $\eps_{h}(t)$ and a vector field $u^{*}(t)\in W^{1,p}_{\tang}(M)$
($p\in [1,2)$) such that $\abs{u^*(t)}_g =1$ a.e. in $M$ and 
\begin{align}
& u_{\eps_{h}(t)}(t)\weak u^*(t)\qquad \hbox{ in } W^{1,p}_{\tang}(M) \label{eq:convu} \quad
 \textrm{for any } p\in (1, \, 2),\\
&\d j(u^*(t)) = - \kappa \Vg + 2\pi \sum_{k=1}^n d_j\delta_{a_k(t)}\label{eq:vorticityustar}\\
&\d^* j(u^*(t)) = 0 \,\,\,\,\,\,\hbox{ in }M,\label{eq:hodge_canonical}
\end{align}
namely $u^*(t)$ is a canonical harmonic vector field for $(\a(t), \, \db, \, \xi(t))$.
\end{prop}
\begin{proof}
\setcounter{step}{0}
We divide the proof in several steps. 
\begin{step}[Compactness]
Thanks to Lemma \ref{lem:lemma1} and to Proposition \ref{prop:compactness}
(see in particular \eqref{eq:convforanyt}) we have that the exist 
$\mu\in C^{0,1/2}([0,T];\mathscr{M}_2(M))$ such that for any $t\in [0,T]$
\begin{equation}
\omega(u_\eps(t)) \to \mu(t)\qquad \hbox{ in } W^{-1,p}(M) \quad
 \textrm{for any } p\in (1, \, 2)
\end{equation}
Thanks to Proposition \ref{prop:comp_prod_est_space} for any $t\in (0,T)$ the measure $\mu(t)$ is of the form
\begin{equation}
\label{eq:limit_measure0}
\mu(t) = 2\pi \sum_{j=1}^{n(t)}d_j(t) \delta_{a_j(t)}, \,\,\,\hbox{ where } d_j\in \mathbb{Z},
\end{equation}
where for any $t\in [0,T]$ $n(t)\in \mathbb{N}$ and $a_j(t)\in M$ for $j=1,\ldots, n(t)$ with $a_j(t)\neq a_k(t)$ for $j\neq k$ and satisfies 
\begin{equation}
\label{eq:gamma_conv0}
\liminf_{\eps\to 0}\frac{F_\eps(u_\eps(t))}{\pi\vert\log\eps\vert}\ge \sum_{j=1}^{n(t)}\vert d_j(t)\vert.
\end{equation}
Therefore, since 
\[
F(u_\eps(t))\le F_\eps (u_\eps^{0})\le \pi n\vert\log\eps\vert +C,
\]
we conclude that for $t\in [0,T)$,
\begin{equation}
\sum_{j=1}^{n(t)} \vert d_j(t)\vert \le   \sum_{j=1}^{n}\vert d_j\vert.
\end{equation}
The fact that the evolution decreases the total charge combined with the continuity of $t\mapsto \int_{M}\mu(t)\zeta$, for $\zeta\in C^{1}_c(M)$, 
give (see \cite[Proposition 3.2]{SS-GF}) that there exists $T^*\in (0,T]$
\footnote{Note that since $u_\eps$ can be defined in the whole $(0,+\infty)$, we can always assume that $T^*\le T$.}
and $\a:[0,T^*)\to M^n$ with $\a\in H^1(0,T^*;M^n)$ such that $a_j(t)\neq a_k(t)$ for $j\neq k$ and for any $t\in [0,T^*)$. Moreover, $d_j(t) = d_j(0) =\pm 1$ for $t\in [0,T^*)$ and thus 
%
%
 the measure $\mu$ has the form 
\begin{equation}
\label{eq:limit_measure}
\mu(t) =2\pi \sum_{j=1}^n d_j\delta_{a_j(t)}
\end{equation}
 for any $t\in [0,T^*)$.
Therefore, we are in the position to apply Proposition \ref{prop:kinetic} to obtain \eqref{eq:liminfut}.

Finally, we consider the sequence $\xi_\eps:=\P(j(u_\eps))$. 
The sequence $u_\eps$ verifies all the hypothesis of Proposition \ref{prop:harmproj} and therefore there exists $\xi\in H^1(0,T^*;\Harm^1(M))$ such that   
\begin{equation}
\label{eq:conv_harmo1form}
\P(j(u_\eps(t)))\xrightarrow{\eps\to 0} \xi(t)\qquad \hbox{ for any }\quad t\in (0,T^*).
\end{equation}

\end{step}

\begin{step}[Convergence to a harmonic vector field]
 We have
\[
\int_{0}^T\int_{M}\abs{\Delta_g u_\eps + \frac{1}{\eps^2}(\vert 1- \vert u_\eps\vert^2)u_\eps}^2_{g} \Vg \d t = \frac{1}{\vert\log\eps\vert^2}\int_{0}^{T}\int_M\vert \partial_t u_\eps\vert^2_g \, \Vg\d t,
\]
and hence, thanks to \eqref{eq:boundut_lemma},
\[
\int_{0}^T \int_{M}\abs{\Delta_g u_\eps + \frac{1}{\eps^2}\left( 1- \abs{u_\eps}_g^2\right)u_\eps}_{g}^2 \, \Vg \d t \le \frac{C}{\vert \log\eps\vert}.
\]
Therefore 
\begin{equation}
\label{eq:Deltazero}
\Delta_g u_\eps + \frac{1}{\eps^2}\left(1-\abs{u_\eps}_g^2\right)u_\eps \to 0\,\,\,\,\hbox{ in }L^2(0,T;L^2(M)),
\end{equation}
and, up to subsequence, 
\begin{equation}
\label{eq:Deltazero_point}
\Delta_g u_\eps + \frac{1}{\eps^2}\left(1-\abs{u_\eps}_g^2\right)u_\eps \to 0 \qquad \hbox{ in } L^2_{\tang}(M)\quad \hbox{ and for almost any }t\in (0,T^*). 
\end{equation}
We let $C\subseteq (0,T^*)$ the set of those $t$ in which the above convergence holds. 
 We fix a $t\in C$.  
Thanks to Lemma \ref{lemma:H1bounds} and to Corollary \ref{cor:W1,pbounds},
there exists a subsequence $\eps_h(t)\xrightarrow{h\to +\infty} 0$,
which may depend on the chosen $t$, and a vector field $u^*(t)\in H^1_{\tang ,\text{loc}}(M\setminus\left\{a_1(t),\ldots,a_d(t)\right\})\cap W^{1,p}_{\tang}(M)$ such that $\abs{u^*(t)}_g=1$ and 
\begin{align}
& u_{\eps_h(t)}(t) \to u^*(t)\,\,\,\hbox{ weakly in } H^1_{\tang ,\text{loc}}(M\setminus\left\{a_1(t),\ldots,a_d(t)\right\}),
\label{eq:conv_ueps_weak2}\\
& u_{\eps_h(t)}(t)\to u^*(t)\,\,\,\hbox{ weakly in } W^{1,p}_{\tang}(M)\,\,\,\,p\in [1,2),\label{eq:conv_ueps_weakp}\\
& u_{\eps_h(t)}(t)\to u^*(t)\,\,\,\hbox{ strongly in } L^p_{\tang}(M)\,\,\,\,p\in [1,+\infty)\label{eq:conv_ueps_strong}.
\end{align}
Since $D u_{\eps_h(t)}(t)$ converges weakly in $L^{p}_{\tang}(M)$ for $p<2$ and $u_{\eps_h(t)}(t)$ converges strongly in $L^p_{\tang}(M)$ for any $p<+\infty$ we get that 
\begin{equation}
\label{eq:conv_current_t}
j(u_{\eps_h(t)}(t))\weak j(u^*(t))\qquad \hbox{ in }\quad L^p(M)\quad \hbox{ for any }p\in [1,2). 
\end{equation}
Therefore, thanks to \eqref{eq:conv_vorticity_prop}, we have that 
\begin{equation}
\label{eq:jacobians1}
\d j(u^*(t)) = -\kappa \Vg + 2\pi\sum_{j=1}^n d_j\delta_{a_j(t)},\,\,\,\,\,\,\hbox{ for a.a. }t\in [0,T^*).
\end{equation}
As a result, we have that
\[
0 \stackrel{\text{Stokes}}= \int_{M}\d j(u^*(t)) = 2\pi \sum_{j=1}^n d_j -\int_{M}\kappa \Vg\stackrel{\text{Gauss-Bonnet}}= 2\pi \sum_{j=1}^n d_j -2\pi \chi(M),
\]
and thus we have that $(\a(t), \db)\in \mathscr{A}^n$.

Now we prove that $\d^* j(u^*(t)) =0$. This will imply that $u^*(t)$ is indeed a canonical harmonic vector field. 
To this end we let $\psi:M\to \mathbb{R}$ be a smooth function. 
Then, since 
\[
\frac{1}{\eps^2}(1-\vert u_{\eps_h(t)}\vert^2)(u_{\eps_h}(t),i u_{\eps_h(t)})_g\psi = 0\qquad\hbox{ in } M
\]
we readily have that 
\[
\lim_{h\to +\infty}\int_{M}\left(\Delta_g u_{\eps_h(t)},iu_{\eps_h(t)}(t)\right)_g\psi \Vg = 0
\]
and thus, thanks to Lemma \ref{lem:ibp} below, that
\begin{equation}
\label{eq:harmonic1}
\lim_{h\to +\infty}\int_{M}\psi \,\d^*j(u_{\eps_h(t)}(t)) \Vg = 0.
\end{equation}
Thus, thanks to \eqref{eq:conv_current_t}, we have 
\[
0 = \lim_{h\to +\infty}\int_{M}\psi \,\d^*j(u_{\eps_h(t)}(t)) \Vg = 
\lim_{h\to +\infty}\int_{M}\d\psi \,j(u_{\eps_h(t)}(t)) \Vg =
\int_{M}\d \psi\,j(u^*(t))\Vg,
\]
namely
\begin{equation}
\label{eq:codiffjstar}
\d^*j(u^*(t)) = 0.
\end{equation}
Thus we conclude that $u^*(t)$ is a canonical harmonic vector field.
Since $t$ was indeed arbitrary on $C$ we have proved that for almost any $t\in (0,T^*)$  we have that $(\a(t),\db)\in \mathscr{A}^n$ with associated a canonical harmonic vector field  $u^*(t)$. 
Note that $u^*(t)$ is obtained as a limit of $u_\eps(t)$ along a selected subsequence.
The continuity of $\a: [0,T^*)\to M^n$ actually implies that $(\a(t), \db)\in \mathscr{A}$ for any $t\in [0,T^*)$. 
The fact that at time $t\in C$ the vector field $u^*(t)$ is a canonical harmonic vector field implies that the harmonic $1$-form $\xi(t)\in \mathcal{L}(\a(t), \db)$. Indeed, on the one hand the convergences \eqref{eq:conv_harmo1form} and \eqref{eq:conv_current_t} imply that 
\[
\xi(t) = \P j(u^*(t)).
\]
On the other hand
we let $\Phi_t=\Phi(\a(t),\db)$ be the $2$-form such that
\begin{equation}
\label{eq:PHI}
\begin{cases}
-\Delta \Phi = -\kappa \Vg + 2\pi\displaystyle\sum_{j=1}^n d_j\delta_{a_j(t)}\\
\displaystyle\int_{M}\Phi =0.
\end{cases}
\end{equation}
Then, 
\eqref{eq:jacobians1} and \eqref{eq:codiffjstar} 
imply that the $1$-form $j(u^*(t))- d^*\Phi_t \in \text{Harm}^1(M)$. Therefore \cite[Theorem 2.1]{JerrardIgnat_full}, and the uniqueness of the projection imply that 
\[
\xi(t) = j(u^*(t))- \d^*\Phi_t, \qquad \xi(t) \in \mathcal{L}(\a(t), \db)\quad \text{ for almost any }t\in [0,T^*).
\]
The continuity of $\xi:[0,T^*)\to \Harm^1(M)$ and of $\a:[0,T^*)\to M^n$ and the continuity of the lattice $\mathcal{L}$ with respect to $(\a,\db)$ (see \cite[Lemma 2.3]{JerrardIgnat_full}) give that 
\begin{equation}
\label{eq:xi_forany}
\xi(t)\in \mathcal{L}(\a(t),\db)\qquad\hbox{ for any }\quad t\in [0,T^*).
\end{equation}
Therefore, \cite[Theorem 2.1]{JerrardIgnat_full} gives that 
for any $t\in [0,T]$ there exists a canonical harmonic vector field $u^*(t)$. 
\qedhere

\end{step}

\end{proof}
\begin{remark}
\label{oss:campo_armonico}
Note that  \eqref{eq:conv_current_t} can be extended to the whole interval $[0,T^*)$ since $j(u^*(t))$ is uniquely identified by its differential, its co-differential and the harmonic $1$-form $\xi$, all of which are continuous with respect to time.
On the contrary, the vector field $u^*(t)$ may not be obtained as a limit of $u_\eps(t)$ for $t\in [0,T]\setminus C$. Moreover, we have no information regarding the regularity of $u^*$ with respect to time. Indeed the form $j(u^*(t))$ identifies $u^*(t)$ only up to an arbitrary global rotation.
\end{remark}

\begin{prop} \label{prop:liminfgrad}
 Let~$u_\eps $ be a solution of~\eqref{eq:GL_intro}, with~$u_{\eps}^0$ satisfying
 \eqref{eq:initial_vorticity}, \eqref{eq:well_prepared} and~\eqref{eq:initial_Linfty}.
 Let~$\a\colon [0, \, T^*)\to M^n$, $\xi\in H^1(0,T^*;\Harm^1(M))$ be
 as in Proposition~\ref{prop:main_limit_proc}. Then, for a.e.~$t\in (0, \, T^*)$, we have
 \begin{equation*} 
  \begin{split}
   \liminf_{\eps\to 0} \frac{\abs{\log\eps}}{2}
    \int_{M}\abs{ -\Delta_g u_\eps 
     + \frac{1}{\eps^2}(\vert u_\eps\vert^2-1)u_\eps}^2_{g}\,\Vg
     \geq  \frac{1}{2\pi} \sum_{j=1}^n
     \abs{\nabla_{a_j} W(\a(t), \, \mathbf{d}, \, \xi(t))}_g^2
  \end{split}
 \end{equation*} 
\end{prop}

\begin{proof}
 Let $t\in (0, \, T^*)$ be fixed, in such a way that 
 $\frac{1}{\abs{\log\eps}}\partial_t u_\eps(t)\to 0$ in $L^2(M)$ as $\eps\to 0$.
 First of all, we extract a subsequence~$\eps_h(t)\to 0$
 (possibly depending on~$t$) in such a way that
 $u_{\eps_h(t)}(t)$ converges to a canonical harmonic field $u^*(t)$ and
 \[
  \begin{split}
   &\liminf_{\eps\to 0} \frac{\abs{\log\eps}}{2}
    \int_{M}\abs{ -\Delta_g u_\eps 
     + \frac{1}{\eps^2}(\vert u_\eps\vert^2-1)u_\eps}^2_{g}\,\Vg \\
   &\qquad\qquad = \lim_{h\to +\infty} \frac{\abs{\log\eps_h(t)}}{2}
    \int_{M}\abs{ -\Delta_g u_{\eps_h(t)} 
     + \frac{1}{\eps_h(t)^2}(\vert u_{\eps_h(t)}\vert^2-1)
     u_{\eps_h(t)}}^2_{g}\,\Vg
  \end{split}
 \]
 We will then be able to extract further subsequences later on.
 We work at fixed~$t$ and drop the dependence on~$t$ from the notation;
 we write $\eps$ instead of~$\eps_h(t)$.
 Let~$j\in\{1, \, \ldots, \, n\}$ be fixed, and let~$\eta>0$
 be a small parameter. We choose an orthonormal basis
 $\{\hat{e}_1, \, \hat{e}_2\}$ of~$T_{a_j}M$. 
 We fix geodesic normal coordinates $\{x^1, \, x^2\}$ centered in $a_j$ in such a way that the fields 
 ~$e_k := \partial/\partial x^k$ satisfy
 \[
 e_k(a_j) =\hat{e}_k, \qquad   \div e_k(a_j)=0, \qquad \Gamma_{i,k}^m(a_j) =0 \quad \forall i,m,k,
 \]
 where $\Gamma_{ik}^m$ are the Cristhoffel symbols of the metric $g$. We let $B_\eta(a_j)$ be the geodesic ball centered at $a_j$ (to simplify the notation, we write $B_\eta$
 instead of~$B_\eta(a_j)$) and let~$\nu$ be the outward-pointing unit normal to ~$\partial B_\eta$.

%
%
 Let~$k\in\{1, \, 2\}$ be fixed. 
 We multiply the equation~\eqref{eq:GL_intro} by~$D_{k} u_\eps$ ($D_k$, as usual, stands for $D_{e_k}$)
 and integrate over~$B_\eta$. We obtain
 \begin{equation}
 \label{eq:stimafond}
  \begin{split}
   \int_{B_{\eta}}\left(\frac{\partial_t u_\eps}{\vert \log\eps\vert}, D_{k}  u_\eps\right)_g \Vg &+ \int_{B_\eta}\left(-\Delta_g u_\eps, D_{k} u_\eps\right)_g \Vg \\
   &+ \frac{1}{\eps^2}\int_{B_\eta}\left((\vert u_\eps\vert^2_g-1)u_\eps, D_{k} u_\eps\right)_g \Vg = 0. 
  \end{split}
 \end{equation}
 We integrate by parts in the second and in 
 the third term in the left hand side.
 The details of the computations are given in
 Lemma~\ref{lem:int_parti1}, Lemma~\ref{lem:int_parti2} 
 (see Appendix~\ref{app:int_byparts} below). We obtain
 \begin{equation}
 \label{eq:stimafond2}
  \begin{split}
   \frac{1}{\vert \log\eps\vert}\int_{B_\eta}(\partial_t u_\eps, D_{k} u_\eps)_g \Vg
   &= \frac{1}{4\eps^2}\int_{\partial B_\eta}(\vert u_\eps\vert_g^2 - 1)^2\, \d\H^1 
   + \frac{1}{2}\int_{\partial B_\eta}\vert D u_\eps\vert^2_g (\nu,e_k)_g \, \d\H^1 \\
   &\qquad- \int_{\partial B_\eta}(D_{k} u_\eps, D_{\nu}u_\eps)_g\, \d\H^1
   +\int_{B_\eta}(D u_\eps, R(\cdot,e_k)u_\eps)_g\Vg\\
   &  -\frac{1}{2}\int_{B_\eta(p)}\abs{D u_\eps}^2_g \div(e_k)\,\Vg
   -\frac{1}{2}\int_{B_\eta(p)}(\partial_k g^{ij})(D_i u_\eps, D_j u_\eps)_g \,\Vg
  \end{split}
 \end{equation}
 where~$R(\cdot, \, e_k)$ is the Riemann curvature tensor of~$M$.
 The left hand side  is controlled by 
 \[
  \begin{split}
   &\frac{1}{\vert \log\eps\vert}\int_{B_\eta}(\partial_t u_\eps, D_{k} u_\eps)_g \Vg \\
   &\qquad\qquad\le \left(\int_{B_\eta}\vert \log\eps\vert\abs{ -\Delta_g u_\eps + \frac{1}{\eps^2}(\vert u_\eps\vert_g^2-1)u_\eps}_g^2\Vg \right)^{1/2}\left(\int_{B_\eta}\frac{\vert D_{k} u_\eps\vert^2}{\vert \log\eps\vert}\Vg \right)^{1/2},
  \end{split}
 \]
 and thus, 
 \begin{equation} \label{eq:liminf_local_grad_ene}
  \begin{split}
  &\left(\int_{B_\eta}\vert \log\eps\vert\abs{ -\Delta_g u_\eps + \frac{1}{\eps^2}(\vert u_\eps\vert_g^2-1)u_\eps}_g^2\Vg\right) \left(\int_{B_\eta}\frac{\vert D_{k} u_\eps\vert_g^2}{\vert \log\eps\vert}\Vg\right)\\
  & \ge \abs{\frac{1}{4\eps^2}\int_{\partial B_\eta}(\vert u_\eps\vert_g^2 - 1)^2\, \d\H^1 + 
   \frac{1}{2}\int_{\partial B_\eta}\vert D u_\eps\vert_g^2 (\nu,e_k)_g\, \d\H^1 - \int_{\partial B_\eta}(D_{k} u_\eps, D_{\nu}u_\eps)_g\, \d\H^1\right.\\
  &\left.+\int_{B_\eta}(D u_\eps, R(\cdot,e_k)u_\eps)_g\Vg
  -\frac{1}{2}\int_{B_\eta(p)}\abs{D u_\eps}^2_g \div(e_k)\,\Vg
   -\frac{1}{2}\int_{B_\eta(p)}(\partial_k g^{ij})(D_i u_\eps, D_j u_\eps)_g \,\Vg}^2.
 \end{split}
 \end{equation}
 Thanks to \eqref{eq:polarization} (see Lemma~\ref{lem:corollario4SS}) we have that 
 \[
  \lim_{\eps\to 0}\frac{1}{\vert\log\eps\vert}\int_{B_\eta}\vert D_{k} u_\eps\vert_g^2\Vg = \lim_{\eps\to 0}\frac{1}{\vert \log\eps\vert}\int_{B_\eta(a_j)}(D u_\eps,e_k)_g^2\Vg  = \pi\vert e_k(a_j)\vert^2_g = \pi,
 \]
since $\left\{e_1, \, e_2\right\}$ are orthonormal in $a_j(t)$. 
Corollary~\ref{cor:W1,pbounds} implies that, up to extraction of a subsequence, 
$D u_\eps$ converges weakly in $L^p(M)$ for any $p\in [1,2)$ 
and $u_\eps$ converges strongly in $L^q(M)$ for any $q\in [1,+\infty)$
(except, possibly, for a negligible set of times~$t$). As a consequence, we have
\[
\lim_{\eps\to 0}\int_{B_\eta}(D u_\eps, R(\cdot,e_k)u_\eps)_g\Vg = \int_{B_\eta}(D u^*,R(\cdot,e_k)u^*)_g\Vg.
\]
On the other hand, Lemma~\ref{lem:lontano_sing} implies that 
\[
\lim_{\eps\to 0}\frac{1}{4\eps^2}\int_{\partial B_\eta}(\vert u_\eps\vert_g^2 - 1)^2 \, \d\H^1=0
\]
and  
\begin{align*}
\lim_{\eps\to 0}\frac{1}{2}\int_{\partial B_\eta}\vert D u_\eps\vert^2_g (\nu,e_k)_g\, \d\H^1
&- \int_{\partial B_\eta}(D_{k} u_\eps, D_{\nu}u_\eps)_g\, \d\H^1	\\
& =\frac{1}{2}\int_{\partial B_\eta}\vert D u^*\vert_g^2 (\nu,e_k)_g\, \d\H^1 
- \int_{\partial B_\eta}(D_{k} u^*, D_{\nu}u^*)_g\, \d\H^1
\end{align*}
for almost every~$\eta$. 
Finally, we concentrate on the last two terms in the right hand side of \eqref{eq:liminf_local_grad_ene}.
First of all, we observe that both terms are controlled by quantities of the form
\[
\int_{B_\eta}e_\eps(u_\eps) f \,\Vg, 
\]
where $e_\eps(u_\eps):=\frac{1}{2}\abs{D u_\eps}_g^2 + \frac{1}{4\eps^2}\left(\abs{u_\eps}_g^2 -1\right)^2$ is the density of the Ginzburg Landau energy and $f:B_\eta(p)\to \R$ is a smooth function such that $f(p)=0$ (in the first term $f= \div(e_k)$ and in the second $f= \partial_k g^{ij}$, see \eqref{eq:crist0} e \eqref{eq:div0}). 
We show that 
\begin{equation}
\label{eq:dyadicPallaeta}
\lim_{\eta \to 0}\,\sup_{\eps>0}\int_{B_\eta}e_\eps(u_\eps) f \,\Vg = 0. 
\end{equation}
To this end, we dyadically decompose $B_\eta$ to obtain that 
\[
\int_{B_\eta}e_\eps(u_\eps) f \,\Vg = \sum_{j=1}^{\infty}\int_{B_{2^{-j}\eta}\setminus B_{2^{-j-1}\eta}}e_\eps(u_\eps) f\,\Vg. 
\]  
On the one hand the smoothness of $f$ and the fact that $f(a_j)=0$ imply $\abs{f}\le \eta2^{-j}$ on each anulus $B_{2^{-j}\eta}\setminus B_{2^{-j-1}\eta}$. 
On the other hand, Lemma \ref{lemma:H1bounds} imply that, for some constant $c$ independent of $\eta$ and of $\eps$, 
\[
\int_{B_{2^{-j}\eta}\setminus B_{2^{-j-1}\eta}}e_\eps(u_\eps)\,\Vg \le c\left(\abs{\log\left(\eta 2^{-j-1}\right)} +1\right).
\]
Therefore,
\[
\abs{\int_{B_\eta}e_\eps(u_\eps) f \,\Vg}\le c\eta 
\sum_{j=1}^{\infty}2^{-j}\left(\abs{\log\left(\eta 2^{-j-1}\right) +1}\right). 
\]
Since the series in the right hand side converges, \eqref{eq:dyadicPallaeta} follows.

Therefore, taking the $\liminf$ in \eqref{eq:liminf_local_grad_ene} produces 
\begin{equation} \label{eq:liminfgrad}
 \begin{split}
  &\liminf_{\eps\to 0}\frac{\vert \log\eps\vert}{2}\int_{B_\eta}\abs{ -\Delta_g u_\eps + \frac{1}{\eps^2}(\vert u_\eps\vert_g^2-1)u_\eps}^2_{g}\,\Vg \\
  &\qquad\ge  \frac{1}{2\pi}\abs{\frac{1}{2}\int_{\partial B_\eta}\vert D u^*\vert^2_g (\nu,e_k)_g\, \d\H^1 - \int_{\partial B_\eta}(D_{e_k} u^*, D_{\nu}u^*)_g\right.\, \d\H^1
  \\
  &\qquad\qquad + \left.\int_{B_\eta}(D u^*,R(\cdot,e_k)u^*)_g\,\Vg  + o_{\eta\to 0}(1)
  }^2
 \end{split}
\end{equation}
Recall that the canonical harmonic vector field $u^*\in W^{1,p}_{\tang}(M)\cap L^\infty_{\tang}(M)$ for $p\in [1,2)$ and the curvature tensor is smooth. Therefore, as~$\eta\to 0$, we have that 
\[
\int_{B_\eta}(D u^*,R(\cdot,e_k)u^*)_g \Vg \to 0. 
\] 
Thanks to Proposition~\ref{prop:gradient_reno} 
we have
\begin{equation*}
 \begin{split}
  \liminf_{\eps\to 0}\frac{\vert \log\eps\vert}{2}\int_{B_{\eta}}\abs{ -\Delta_g u_\eps + \frac{1}{\eps^2}(\vert u_\eps\vert_g^2-1)u_\eps}^2_{g}\,\Vg
  \ge \frac{1}{2\pi} \abs{\left(\nabla_{a_j} W(\a, \, \db, \, \xi),
  \, {\hat{\be}}_k\right)_g}^2 + \mathrm{o}_{\eta\to 0}(1)
 \end{split}
\end{equation*}
As~$\{\hat{\be}_1, \, \hat{\be}_2\}$ is an arbitrary orthonormal basis
for~$\T_{a_j}(M)$, we also deduce
\begin{equation*}
 \begin{split}
  \liminf_{\eps\to 0}\frac{\vert \log\eps\vert}{2}\int_{B_{\eta}}\abs{ -\Delta_g u_\eps + \frac{1}{\eps^2}(\vert u_\eps\vert_g^2-1)u_\eps}^2_{g}\,\Vg
  \ge \frac{1}{2\pi} \abs{\left(\nabla_{a_j} W(\a, \, \db, \, \xi),
  \, {\hat{\be}}\right)_g}^2 + \mathrm{o}_{\eta\to 0}(1)
 \end{split}
\end{equation*}
for any unit vector~$\hat{\be}\in\T_{a_j}(M)$. By taking the
supremum over~$\hat{\be}$, we obtain
\begin{equation*}
 \begin{split}
  \liminf_{\eps\to 0}\frac{\vert \log\eps\vert}{2}\int_{B_{\eta}}\abs{ -\Delta_g u_\eps + \frac{1}{\eps^2}(\vert u_\eps\vert_g^2-1)u_\eps}^2_{g}\,\Vg
  \ge \frac{1}{2\pi} \abs{\nabla_{a_j} W(\a, \, \db, \, \xi)}_g^2 + \mathrm{o}_{\eta\to 0}(1)
 \end{split}
\end{equation*}
By taking the sum over~$j$, the proposition follows.
\end{proof}

We can finally prove our main result, Theorem \ref{th:main1},
which is an immediate consequence of the following proposition.
We recall that the constant~$\gamma > 0$ is the core energy,
implicitely defined in Equation~\eqref{eq:F-Gamma}
(see~\cite[Section~2.3]{JerrardIgnat_full} for more details).

\begin{prop}
\label{prop:teorema}
 Let $u_\eps $ be a solution of \eqref{eq:GL_intro} with $u_{\eps}^0$ satisfying 
 \eqref{eq:initial_vorticity}, \eqref{eq:well_prepared} and~\eqref{eq:initial_Linfty}.
 Let~$\a\colon [0, \, T^*)\to M^n$, $\xi\in H^1(0,T^*;\Harm^1(M))$ be
 as in Proposition~\ref{prop:main_limit_proc}. Then, the following statements hold.
 \begin{enumerate}[label=(\roman*), ref=\roman*]
  \item For any~$t\in (0,T^*)$, $u_\eps(t)$ is a recovery sequence for $W$, that is
  \begin{gather*}
   \omega(u_\eps(t))\to 2\pi \sum_{j=1}^n d_j \delta_{a_j(t)}
   \qquad \textrm{in } W^{-1,p}(M) \textrm{ for any } p in [1, \, 2) \\
   F_\eps(u_\eps(t))\leq \pi n\abs{\log\eps} + W(\a(t),\db,\xi(t)) + n\gamma + \mathrm{o}_{\eps\to 0}(1)
  \end{gather*}
  \item For any~$t\in (0,T^*)$, 
  \[
    \P j(u_\eps(t))\xrightarrow{\eps\to 0}\xi(t) 
    \qquad \text{ and }\quad \xi(t) \in \mathcal{L}(\a(t), \, \db)
  \]
  \item The vortex curve~$\a$ satisfies for any~$t\in (0,T^*)$
  \begin{equation*}
  \label{eq:gradflowW_propo}
   \frac{\pi}{2}\int_{0}^t\abs{\a'}_{g}^2(s)\d s + \frac{1}{2\pi}\int_{0}^t\abs{\nabla_{\a} W(\a(s),\db,\xi(s))}^2_{g}\d s   + W(\a(t),\db,\xi(t)) =  W(\a^0,\db,\xi_0)
  \end{equation*}
  Equivalently, $\a$ is a solution of the gradient flow of the renormalized energy
  \begin{equation}
    \label{eq:grad_flow_differential}
    \begin{cases}
     \displaystyle\frac{\d}{\d t} \a(t) = -\frac{1}{\pi}\nabla_{\a} W(\a(t),\d(t),\xi(t))\,\,\,\,\textrm{for any } t\in (0, \, T^*),\\
      \\
     \a(0) = \a^0.
    \end{cases}
  \end{equation}
 \end{enumerate}
\end{prop}
\begin{proof}
The proof is based on the abstract scheme developed by Sandier \& Serfaty (see \cite[Theorem 1.4]{SS-GF}).

Equation \eqref{eq:GL_intro} is the $L^2$-gradient flow of the Ginzburg-Landau energy $F_\eps$. 
More precisely, for any $\eps>0$ and for any $v\in H^1_{\tang}(M)$ we let 
\[
\displaystyle E_\eps(v):=
\begin{cases}
F_\eps(v) -\pi n\vert \log\eps\vert-n\gamma,\,\,\,\,\,\,&\hbox{ if }D v\in L^2_{\tang}(M),\,\,\,\frac{1}{\eps^2}(1-\vert v\vert^2)^2\in L^1(M)\\
\\
+\infty\,\,\,\,\,&\hbox{otherwise in }L^2_{\tang}(M). 
\end{cases}
\]
Recall that $E_\eps$ $\Gamma$-converges to $W$ and therefore, for any $t\in [0,T)^*$,
\begin{equation}
\label{eq:gamma_conv_foranyt}
\liminf_{\eps\to 0}E_\eps(u_\eps(t)) \ge W(\a(t),\,\db(t),\,\xi(t)).
\end{equation}
We set $X_\eps:= L^2_{\tang}(M)$. We endow $X_\eps$ with the norm $\norm{v}_{X_\eps}:=\frac{1}{\vert \log\eps\vert^{1/2}}\norm{v}_{L^2(M)}$.
It is easy to check that the subdifferential of $E_\eps$ (with respect to the scalar product in $X_\eps$) is singlevalued and is given by  
\[
\partial_{X_\eps} E_\eps(v) =\abs{\log\eps}\left(-\Delta v + \frac{1}{\eps^2}(\vert v\vert_g^2 -1)v\right).
\]
Therefore the Ginzburg-Landau evolution \eqref{eq:GL_intro} rewrites as the gradient flow
\[
\begin{cases}
\partial_t u_\eps(t) + \partial_{X_\eps}E_\eps(u_\eps) = 0\\
\\
u_\eps(0) = u_{\eps}^0.
\end{cases}
\]
Equivalently, $u_\eps$ is a {\itshape curve of maximal slope} for the energy $E_\eps$ with respect to the slope $\| \partial_{X_\eps}E_\eps\|_{X_\eps}$ and thus it satisfies
\begin{align}
\label{eq:gradfloweps}
\frac{1}{2}\int_{0}^t\norm{\partial_t u_\eps}_{X_\eps}^2 \d s
 + \frac{1}{2}\int_{0}^{t}\norm{\partial_{X_\eps}E_\eps(u_\eps(s))}^2_{X_\eps} \d s + E_\eps(u_\eps(t)) = E_\eps(u_\eps^{0})
\end{align}
for any~$t\in [0, \, T^*)$.
Thanks to Proposition \ref{prop:main_limit_proc} (see in particular \eqref{eq:liminfut}) 
we have 
\begin{equation}
\label{eq:kinetic}
\liminf_{\eps\to 0}\frac{1}{2}\int_{0}^t \norm{\partial_t u_\eps(s)}^2_{X_\eps}\d s \ge \frac{\pi}{2} \int_{0}^t \abs{\a'}^2_{g}(s)\d s. 
\end{equation}
On the other hand, Proposition~\ref{prop:liminfgrad} and Fatou's Lemma imply
\begin{equation}
\liminf_{\eps\to 0}\frac{1}{2}\int_{0}^t \norm{\partial_{E_\eps} u_\eps(s)}_{X_\eps}^2
\d s \ge \frac{1}{2\pi}\int_{0}^t \abs{\nabla_{\a}W(\a(s), \, \db, \, \xi(s))}_g^2 \d s.
\end{equation}
Therefore, for any $t\in [0,T^*)$ we obtain 
\begin{equation} \label{eq:maximal_slope1}
 \begin{split}
  E_\eps(u_\eps^0) -E_\eps(u_\eps(t)) &\ge  \frac{\pi}{2} \int_{0}^t \abs{\a'}^2_{g}(s)\d s + \frac{1}{2\pi}\int_{0}^t   \abs{\nabla_{\a}W(\a(s), \, \db, \, \xi(s))}_g^2 \d s +\mathrm{o}_{\eps\to 0}(1) \\
  & \ge\int_{0}^t \sum_{j=1}^n \left(\frac{\d a_j}{\d t},-\nabla_{a_j}W(\a, \, \db, \, \xi)\right)_g \d s+ \mathrm{o}_{\eps\to 0}(1)\\
& = W(\a(0),\,\db,\,\xi) - W(\a(t),\,\db,\,\xi(t)) + \mathrm{o}_{\eps\to 0}(1). 
 \end{split}
\end{equation}
On the other hand, since the initial conditions are well prepared according to 
\eqref{eq:initial_vorticity}, \eqref{eq:well_prepared} and we have the $\Gamma$-convergence of $E_\eps$ along $u_\eps(t)$ (see \eqref{eq:gamma_conv_foranyt}), 
we have 
\[
E_\eps(u_\eps^0) -E_\eps(u_\eps(t)) \le W(\a(0),\,\db,\,\xi) - W(\a(t),\,\db,\,\xi(t)) + \mathrm{o}_{\eps\to 0}(1).
\]
Therefore, the inequalities in \eqref{eq:maximal_slope1} become equalities. 
In particular, we obtain 
\[
 \frac{\pi}{2} \int_{0}^t \abs{\a'}^2_{g}(s)\d s + \frac{1}{2\pi}\int_{0}^t \abs{\nabla_{\a}W(\a(s), \, \db, \, \xi(s))}_g^2 \d s + W(\a(t),\,\db,\,\xi(t)) = W(\a(0),\,\db,\,\xi),
\]
that is \eqref{eq:gradflowW_propo}. 
Moreover, we also obtain that 
\[
E_\eps(u_\eps(t)) = W(\a(t),\,\db,\,\xi(t)) + \mathrm{o}_{\eps\to 0}(1)
\qquad \textrm{for any } t\in [0,T^*),
\]
namely that $u_\eps$ remains well prepared along the evolution, or, in other words, that $u_\eps(t)$ is a recovery sequence for the renormalized energy $W$ for any $t\in (0,T^*)$. 

Statement~(ii) is already contained in
Proposition~\ref{prop:main_limit_proc}.
Finally, the fact that \eqref{eq:gradflowW_propo} is indeed equivalent to the differential formulation of the gradient flow \eqref{eq:grad_flow_differential} is standard. More precisely, 
since the chain rule implies
\begin{align*}
\int_{0}^t \sum_{j=1}^n \left(\frac{\d a_j}{\d t},-\nabla_{a_j}W(\a, \, \db, \, \xi)\right)_g\d s
 = W(\a(0),\,\db,\,\xi) - W(\a(t),\,\db,\,\xi(t)),
\end{align*}
we obtain that for any $t\in (0,T^*)$ there holds
\begin{align*}
 \frac{\pi}{2} \int_{0}^t \abs{\a'}^2_{g}(s)\d s + \frac{1}{2\pi}\int_{0}^t \abs{\nabla_{\a}W(\a(s), \, \db, \, \xi(s))}_g^2 \d s = \int_{0}^t \sum_{j=1}^n \left(\frac{\d a_j}{\d t},-\nabla_{a_j}W(\a, \, \db, \, \xi)\right)_g\d s,
\end{align*}
which means
\[
\int_{0}^t\abs{\displaystyle\frac{\d}{\d t}\a(s) +\frac{1}{\pi}\nabla_{a_j}W(\a(s), \, \db, \, \xi(s))}_g^2 \d s= 0,
\]
and thus
\[
\displaystyle\frac{\d}{\d t}\a(t) = -\frac{1}{\pi}\nabla_{a_j}W(\a(t), \, \db, \, \xi(t))\qquad \text{ for any }\quad t\in (0,T^*). \qedhere
\]
\end{proof}

\section*{Acknowledgements}
{ \small GC \& AS are members of the GNAMPA (Gruppo Nazionale per l'Analisi Matematica, la Probabilit\`a e le loro Applicazioni)
group of INdAM.  
AS acknowledges the partial support of the MIUR-PRIN Grant 2017 ''Variational methods for stationary and evolution problems with singularities and interfaces''.
GC acknowledges partial support from the Leverhulme Trust, through the Research Project Grant ORPG-9787 ``Unravelling the Mysteries of Complex Nematic Solution Landscapes'',
from the Agence Nationale de la Recherche, through the project ANR-22-CE40-0006 ``Singularités d'applications \`a valeurs vectorielles minimisant une \'energie'',
and from the University of Verona, through RIBA~2019 No. RBVR199YFL ``Geometric Evolution of Multi-Agent Systems''.
}

\section*{Conflict of interest}
{\small The authors declare that they have no conflict of interest.}

\begin{appendix}
\section{Integration by parts}
\label{app:int_byparts}
 We collect some integration by parts results needed in our argument. 
  
\begin{lemma}
\label{lem:ibp}
 Let $(M, \, g)$ be a compact two dimensional Riemannian manifold. 
 Let $v$ be a smooth vector field. Then, for any smooth 
 $\psi\colon M\to \mathbb{R}$ there holds
 \begin{equation}
 \label{eq:ibp}
  \int_M (\Delta_g v, \, i v)_g \, \psi \, \Vg
  = -\int_{M} \psi \, \d^*j(v).
 \end{equation}
\end{lemma}
\begin{proof}
 We integrate by parts, and apply~\eqref{eq:ip_rough} 
 \begin{equation}
  \begin{split}
   - \int_M (\Delta_g v, \, i v)_g \, \psi \, \Vg
   &= \int_{M} (D v, \, D(\psi \, i v))_g \, \Vg \\
   &= \int_{M} (D v, \, iv\otimes\d\psi)_g \, \Vg 
     + \int_{M} (D v, \, D(i v))_g \, \psi \, \Vg 
  \end{split}
 \end{equation}
 Here, $iv\otimes\d\psi$ is the (fibre-wise linear) operator
 $TM\to TM$ that maps a vector field~$w$ to the vector field~$i v \, \d\psi(w)$.
 We claim that
 \begin{equation} \label{ibp1}
  (D v, \, D(i v))_g = 0 \qquad \textrm{at any point of } M.
 \end{equation}
 Indeed, let $\left\{\tau_1,\tau_2\right\}$ be a 
 moving frame on an open set~$U\subseteq M$.
 There holds that $\tau_2 = i\tau_1$ and thus 
 \begin{align*}
  v &=  v^1\tau_1 +iv^2\tau_1\\
  iv &= -v^2\tau_1 +v^1\tau_2.
 \end{align*}
 Therefore (see Section \ref{sssec:connection1} in this paper and \cite[Section 5.2]{JerrardIgnat_full}) for $j=1,2$, 
 \begin{align*}
  D_{\tau_j}(iv) &= -D_{\tau_j}(v^2\tau_1) + D_{\tau_j}(v^1\tau_2) \\
  &=
  -(\d v^2(\tau_j) -\mathcal{A}(\tau_j)v^2)\tau_1 + (\d v^1(\tau_j) -\mathcal{A}(\tau_j)v^1)\tau_2\\
  &= iD_{\tau_j} v\,\,\,\,\,\,\hbox{ in }U,
 \end{align*}
 from which~\eqref{ibp1} follows. 
 
 Now, we consider~$(Dv, \, iv\otimes\d\psi)_g$.
 In terms of the orthonormal frame~$\{\tau_1, \, \tau_2\}$,
 we have
 \[
  (Dv, \, iv\otimes\d\psi)_g 
  = (D_{\tau_1} v, \, i v)_g \,  \d\psi(\tau_1)
   + (D_{\tau_2} v, \, i v)_g \,  \d\psi(\tau_2)
 \]
 Recalling the definition of~$j(v)$, Equation~\ref{eq:current},
 we can write
 \[
  (Dv, \, iv\otimes\d\psi)_g 
  = j(v)(\tau_1) \,  \d\psi(\tau_1)
   + j(v)(\tau_2) \,  \d\psi(\tau_2)
  = (j(v), \, \d\psi)_g
 \]
 Therefore,
 \[
  - \int_M (\Delta_g v, \, i v)_g \, \psi \, \Vg
  = \int_{M} (j(v), \, \d\psi)_g \, \Vg 
  \stackrel{\eqref{eq:in_parts_form}}{=}
  \int_{M} \d^*j(v) \, \psi \, \Vg  \qedhere
 \]
\end{proof}

\begin{lemma}
\label{lem:int_parti1}
 Let~$p\in M$ and let~$u$, $v$ be smooth vector 
 fields defined in a neighbourhood of~$p$. Then, for any~$\eps>0$
 and any~$\eta>0$ small enough, there holds 
 \begin{equation}
 \label{eq:int_parti1}
  \begin{split}
   &\frac{1}{\eps^2}\int_{B_\eta(p)} ((\vert u\vert^2_g-1)u, \, D_v u)_g \Vg \\
   &\qquad\qquad = 
   -\frac{1}{4\eps^2} \int_{B_\eta(p)} (\vert u\vert^2_g-1)^2 
    \, \div(v) \, \Vg 
   + \frac{1}{4\eps^2} \int_{\partial B_\eta(p)}
    (\vert u\vert^2_g-1)^2 \, (\nu, \, v)_g \, \d\H^1,
  \end{split}
 \end{equation}
 where $\nu$ is the outer normal to $\partial B_\eta$.
\end{lemma}
\begin{proof}
 Thanks to the compatibility of the covariant
 derivative with the metric, we have 
 \begin{equation} \label{ibp11}
  (\nabla \abs{u}^2_g, \, v)_g 
  = 2(u, \, D_v u)_g
 \end{equation}
 Therefore,
 \[
  ((\abs{u}^2_g - 1)u, \, D_v u)_g
  = \frac{1}{2}(\nabla \abs{u}^2_g, \, (\abs{u}^2_g - 1)v)_g
  = \frac{1}{2}(\nabla(\abs{u}^2_g - 1), \, (\abs{u}^2_g - 1)v)_g
 \]
 We integrate by parts, by applying~\eqref{eq:ip_div}. We obtain
 \[
  \begin{split}
   &\int_{B_\eta(p)} ((\abs{u}^2_g - 1)u, \, D_v u)_g \, \Vg
   = - \frac{1}{2} \int_{B_\eta(p)}
    (\abs{u}^2_g - 1)^2 \, \div(v) \, \Vg \\
   &\qquad\qquad - \frac{1}{2} \int_{B_\eta(p)}
    (\abs{u}^2_g - 1) \, (\nabla\abs{u}^2_g, \, v)_g \, \Vg
    + \frac{1}{2} \int_{\partial B_\eta(p)}
    (\abs{u}^2_g - 1)^2 \, (\nu, \, v)_g \, \d\H^1
  \end{split}
 \]
 Taking~\eqref{ibp11} into account, we obtain~\eqref{eq:int_parti1}.
\end{proof}

\begin{lemma}
\label{lem:int_parti2}
 Let $u$ be a smooth vector field. Let~$\{x^1, \, x^2\}$
 be a local coordinate system, defined in a neighbourhood 
 of a point~$p\in M$. 
 We let $e_k := \partial/\partial x^k$. 
%
 Then, for $k=1, \, 2$, we have 
 \begin{equation}
 \label{eq:int_parti2}
  \begin{split}
   \int_{B_\eta(p)} (-\Delta u, \, D_k u)_g \Vg
   &= \frac{1}{2}\int_{\partial B_\eta(p)}
    \vert D u\vert^2_g \, (\nu, \, e_k)_g\, \d\H^1 
   - \int_{\partial B_\eta(p)} (D_\nu u, \, \D_k u)_g \, \d\H^1 \\
   & + \int_{B_\eta(p)} (D u, \, R(\cdot, \, e_k)u)_g \, \Vg  -\frac{1}{2}\int_{B_\eta(p)}\abs{D u}^2_g \div(e_k)\,\Vg\\
   &-\frac{1}{2}\int_{B_\eta(p)}(\partial_k g^{ij})(D_i u, D_j u)_d \,\Vg,
  \end{split}
 \end{equation}
 where~$R$ is the Riemann curvature tensor. 
\end{lemma}

\begin{proof}
 Recall that $\Delta u = -D^*D u$ and thus
 \begin{equation*}
  \begin{split}
   \int_{B_\eta(p)} (-\Delta u, \, D_k u)_g \, \Vg 
   &= \int_{B_\eta(p)} (D^*D u, \, D_{k} u)_g \, \Vg \\
   &= \int_{B_\eta(p)} (D u, \, D D_{k} u)_g \, \Vg 
    - \int_{\partial B_\eta(p)} (D_\nu u, \, D_{k} u)_g \, \d\H^1.
  \end{split}
 \end{equation*}
 Now we work on the first term in the right hand side. Recalling \eqref{eq:normDv} and the symmetry 
 of the metric we obtain
 \[
 \frac{1}{2}\partial_k \vert D u\vert^2_{g} =  g^{ij}(D_i u, D_{k}D_j u)_g + \frac{1}{2}\partial_k g^{ij}(D_i u,D_j u)_g
 \]
 
 Then, the very definition of curvature tensor gives 
 \[
  D_{j} D_{k} u = D_{k} D_{j} u + R(e_j, \, e_k) u,
 \]
 and thus
 \[
 \frac{1}{2}\partial_k \vert D u\vert^2_{g} = g^{ij}(D_i u, D_{j}D_k u)_g + \frac{1}{2}(\partial_k g^{ij})(D_i u,D_j u)_g
 -g^{ij}(D_i u, R(e_j,e_k) u)_g.
 \]

 Therefore
 
 \[
 \begin{split}
  \int_{B_\eta(p)} (D u, \, D D_{e_k} u)_g \, \Vg 
  &= \frac{1}{2} \int_{B_\eta(p)}\partial_k \vert D u\vert^2_g \, \Vg -\frac{1}{2}\int_{B_\eta(p)}(\partial_k g^{ij})(D_i u, D_j u)_d \,\Vg\\
  &+ \int_{B_\eta(p)} g^{ij}(D_i u, \, R(e_j, \, e_k)u)_g \, \Vg.
 \end{split}
 \]
 
 Then we
 integrate by parts in the first integral in the right hand side
 (see~\eqref{eq:ip_div}). 
We get
 \[
 \begin{split}
 \int_{B_\eta(p)} (D u, \, D D_{e_k} u)_g \, \Vg 
 &= \frac{1}{2}\int_{\partial B_\eta(p)}\abs{D u}^2_g (\nu, e_k)  \, \d\H^1
  -\frac{1}{2}\int_{B_\eta(p)}\abs{D u}^2_g \div(e_k)\,\Vg \\
  &-\frac{1}{2}\int_{B_\eta(p)}(\partial_k g^{ij})(D_i u, D_j u)_d \,\Vg 
+ \int_{B_\eta(p)} g^{ij}(D_i u, \, R(e_j, \, e_k)u)_g \, \Vg
 \end{split}
 \]
 
 and the lemma follows.
\end{proof}

\section{Proof of Lemma~\ref{lemma:indexflat}}
\label{app:flatindex}

The goal of this secton is to prove Lemma~\ref{lemma:indexflat}.
For convenience, we reproduce the statement of the lemma here.

\begin{lemma} \label{lemma:indexflatapp}
 Let~$u_\eps\in H^1_{\tang}(M)$ be a sequence of vector fields
 that satisfies
 \begin{equation} \label{hp:H1bounds-flatapp} 
  \begin{split}
   &\omega(u_\eps) \to 2\pi\sum_{j=1}^n d_j \delta_{a_j} \qquad 
   \textrm{in } W^{-1,1}(M) \quad  \textrm{as } \eps\to 0
  \end{split} 
 \end{equation}
 where~$a_1$, \ldots, $a_n$ are distinct points in~$M$ and~$d_1$,\ldots $d_n$
 are non-zero integers, and
 \begin{gather}
  \norm{u_\eps}_{L^\infty(M)} \leq C_0 \label{hp:H1bounds-Linftyapp} \\
  \int_{M} \left(\frac{1}{2}\abs{D u_\eps}^2_g 
   + \frac{1}{4\eps^2}(1 -  \abs{u_\eps}^2_g)^2\right)\Vg 
   \le \pi n \abs{\log\eps}  + C_0 \label{hp:H1bounds-energyapp}
 \end{gather}
 for some $\eps$-independent constant~$C_0$.
 Let~$\rho>0$ be small enough, so that the closed balls~$\bar{B}_\rho(a_j)$
 are pairwise disjoint. Given an index~$j\in\{1, \, \ldots, \, n\}$,
 we assume that there exist a constant~$C$ such that
 \begin{equation} \label{eq:upper_bound_bdapp}
  \int_{\partial B_\rho(a_j)} \left(\frac{1}{2}\abs{D u_\eps}^2_g
  + \frac{1}{4\eps^2}(1 - \abs{u_\eps}^2_g)^2\right) \d\mathscr{H}^1
  \lesssim \frac{C\abs{\log\eps}}{\rho}
 \end{equation}
 for any~$\eps$ small enough. Then, for any~$\eps$
 small enough we have~$\abs{u_\eps}_g \geq 1/2$ 
 on~$\partial B_\rho(a_j)$ and
 \[
  \ind(u_\eps, \, \partial B_\rho(a_j)) = d_j
 \]
\end{lemma}

If we knew that the sequence of measures~$\omega(u_\eps)$
has uniformly bounded mass, then we could extract a subsequence
such that~$\omega(u_\eps)\rightharpoonup^* 2\pi \sum_{j=1}^nd_j\delta_{a_j}$
weakly$^*$ in the sense of measures, as well as in~$W^{-1,1}(M)$.
As a consequence, we would deduce $\omega(u_\eps)(B_\rho(a_j))\to 2\pi d_j$
for a.e.~radius~$\rho$, and the lemma would follow
by the properties of~$\omega(u_\eps)$ (see Section~\ref{ssec:vorticity}).
Unfortunately, though, the assumptions~\eqref{hp:H1bounds-flatapp},
 \eqref{hp:H1bounds-Linftyapp}, 
\eqref{hp:H1bounds-energyapp} do \emph{not} imply that the
measures~$\omega(u_\eps)$ have uniformly bounded mass.
This makes the proof of Lemma~\ref{lemma:indexflatapp} more technical.
The strategy we adopt is quite classical in the theory of Ginzburg-Landau 
functionals (see, for instance, \cite{ABO2}): first, 
we construct a suitable partition of the domain (a `grid'); then,
we use this partition to construct a sequence of measures
that approximate~$\omega(u_\eps)$ in~$W^{-1,1}$ and have uniformly bounded mass.
This enable us to pass to the limit in the sense of measures.

\begin{proof}[Proof of Lemma~\ref{lemma:indexflatapp}]
 As a preliminary remark, we observe that the
 assumptions~\eqref{hp:H1bounds-flatapp}, \eqref{hp:H1bounds-Linftyapp},
 \eqref{hp:H1bounds-energyapp} and the
 estimate~\eqref{eq:conv_energy_palle_lemma} in 
 Lemma~\ref{lem:corollario4SS} imply
 \begin{equation} \label{ubH1-0}
  \int_{B_{2\rho}(a_j)\setminus B_\rho(a_j)}
  \left(\frac{1}{2}\abs{D u_\eps}^2_g
  + \frac{1}{4\eps^2}(1 - \abs{u_\eps}^2_g)^2\right) \Vg
  = \mathrm{o}_{\eps\to 0}(\abs{\log\eps})
 \end{equation}
 We split the proof in several steps.
 
 \setcounter{step}{0}
 \begin{step}[Reduction to the Euclidean setting]
  Let~$\{\tau_1, \, \tau_2\}$
  be an orthonormal tangent frame on~$B_{2\rho}(a_j)$, whose associated
  connection form~$\mathcal{A}$ satisfies~\eqref{eq:smallA}.
  By working in geodesic normal coordinates, we can represent
  $u_\eps$ by means of a Euclidean vector 
  field~$v_\eps\colon B_{2\rho}(0)\to\R^2$, as in~\eqref{uv}.
  Equation~\eqref{eq:energy_dens_normal} implies
  \begin{align}
   \int_{B_{2\rho}(0)} \bar{e}_\eps(v_\eps) \, \d x
   &\stackrel{\eqref{hp:H1bounds-energyapp}}{\leq } 
    C\abs{\log\eps} \label{ubH11-1} \\
   \int_{B_{2\rho}(0)\setminus B_\rho(0)}
     \bar{e}_\eps(v_\eps) \, \d x
   &\stackrel{\eqref{ubH1-0}}{=} 
    \mathrm{o}_{\eps\to 0}(\abs{\log\eps}) \label{ubH11-2} \\
   \int_{\partial B_\rho(0)}
     \bar{e}_\eps(v_\eps) \, \d\mathscr{H}^1
   &\stackrel{\eqref{eq:upper_bound_bdapp}}{\leq} 
    \frac{C\abs{\log\eps}}{\rho} \label{ubH11-3} 
  \end{align}
  where
  \[
   \bar{e}_\eps(v_\eps) := \frac{1}{2}\abs{\nabla v_\eps}^2
    + \frac{1}{4\eps^2}(1 - \abs{v_\eps}^2)^2
  \]
  is the Euclidean Ginzburg-Landau energy density.
  Let~$\Phi\colon B_{2\rho}(a_j)\to B_{2\rho}(0)$ be the coordinate map.
  The Euclidean vorticity of~$v_\eps$, i.e.~the $2$-form~$\d\bar{\jmath}(v_\eps)$
  (where~$\bar{\jmath}$ is defined by~\eqref{eq:flat_space_current}), satisfies 
  \begin{equation*}
   \d\bar{\jmath}(v_\eps) = (\Phi^{-1})^*\omega(u_\eps) - \d\mathfrak{R}_\eps
    \qquad \textrm{where } \mathfrak{R}_\eps
    := \left(1 - \abs{v_\eps}^2\right) (\Phi^{-1})^*\mathcal{A},
  \end{equation*}
  because of Lemma~\ref{lemma:pullback}. The energy 
  estimate~\eqref{ubH11-1} implies $\mathfrak{R}_\eps \to 0$ 
  strongly in~$L^2(B_{2\rho}(0))$ as~$\eps\to 0$,
  hence $\d\mathfrak{R}_\eps \to 0$ strongly in~$W^{-1,2}(B_{2\rho}(0))$ and
  \begin{equation} \label{ubH11-4}
   \d\bar{\jmath}(v_\eps)\to 2\pi d_j \delta_0
   \qquad \textrm{in } W^{-1,1}(B_{2\rho}(0))
   \quad \textrm{as } \eps\to 0.
  \end{equation}
 \end{step}
 
 \begin{step}[Construction of a suitable cover]
  From now on, we denote by~$C$ a generic positive constant,
  that may change from line to line; it may depend on~$\rho$
  (which is fixed), but does not depend on~$\eps$. Let
  \[
   h(\eps) := \abs{\log\eps}^{-2}
  \]
  We construct a finite collection~$\mathscr{G}_\eps$ of Lipschitz, closed sets
  that satisfies the following properties:
  \begin{itemize}
   \item[(i)] $\bar{B}_{2\rho}(0) = \bigcup_{K\in\mathscr{G}_\eps} K$
   and the interiors of~$K$ are pairwise disjoint;
   \item[(ii)] for any~$K\in\mathscr{G}_\eps$, there exists
   an invertible, Lipschitz map~$T_K\colon K\to [-h(\eps), \, h(\eps)]^2$,
   with Lipschitz inverse, such that
   \[
    \norm{\nabla T_K}_{L^\infty(K)} 
    + \norm{\nabla(T^{-1}_K)}_{L^\infty([-h(\eps), \, h(\eps)]^2)} \leq C
   \]
   for some constant~$C$ that does not depend on~$\eps$, $K$;
   \item[(iii)] for~$R_\eps := \cup_{K\in\mathscr{G}_\eps}\partial K$,
   there holds $\mathscr{H}^1(R_\eps) \leq Ch(\eps)^{-1}$ and
   \begin{equation} \label{ubH12-1}
    \int_{R_\eps} \bar{e}_\eps(v_\eps) \, \d\mathscr{H}^1
    \leq \frac{C\abs{\log\eps}}{h(\eps)}
   \end{equation}
   for some~$\eps$-independent constant~$C$;
   \item[(iv)] for any~$K\in\mathscr{G}_\eps$, either
   \[
    \mathrm{interior}(K)\subseteq B_{2\rho}(0)\setminus \bar{B}_\rho(0)
    \qquad \textrm{or} \qquad \mathrm{interior}(K)\subseteq B_\rho(0).
   \]
  \end{itemize}
  First, we define a suitable cover of~$B_\rho(0)$.
  Let~$T\colon\R^2\to\R^2$, $T(x) := \max(\abs{x^1}, \, \abs{x^2}) x/\abs{x}$.
  The map~$T$ is Lipschitz, invertible, has Lipschitz inverse, and maps
  the square~$[-\rho, \, \rho]^2$ to the disk~$\bar{B}_\rho(0)$.
  Let~$N_\eps\in\N$ be the integer part of~$2\rho/h(\eps)$.
  By Fubini theorem, we can select suitable points
  \[
   a_\eps^0 = -\rho < a_\eps^1 < \ldots < a_\eps^{N_\eps} = \rho
  \]
  such that, for any~$k\in\{1, \, \ldots, \, N_\eps - 1\}$,
  \begin{gather}
   \abs{a_\eps^k - \left(-\rho + \frac{2\rho k}{N_\eps}\right)}
    \leq \frac{\rho}{2N_\eps} \label{ubH12-2} \\
   \int_{\{a_\eps^k\}\times [-\rho, \, \rho]} 
    \bar{e}_\eps(v_\eps\circ T) \, \d\mathscr{H}^1 
   \leq \frac{N_\eps}{\rho}  \int_{[-\rho, \, \rho]^2}
    \bar{e}_\eps(v_\eps\circ T) \, \d x  \label{ubH12-3}
  \end{gather}
  We observe that~$2\rho/N_\eps \sim h(\eps)$ as~$\eps\to 0$
  and hence, \eqref{ubH12-3} implies
  \begin{equation} \label{ubH12-4}
   \int_{\{a_\eps^k\}\times [-\rho, \, \rho]} 
    \bar{e}_\eps(v_\eps\circ T) \, \d\mathscr{H}^1 
   \leq \frac{C}{h(\eps)} \int_{B_\rho(0)}
    \bar{e}_\eps(v_\eps) \, \d x  
   \stackrel{\eqref{ubH11-1}}{\leq} \frac{C\abs{\log\eps}}{h(\eps)}
  \end{equation}
  In fact, \eqref{ubH12-4} remains valid for~$k=0$ and~$k=N_\eps$,
  thanks to~\eqref{ubH11-3}. In a similar way, we find points
  \[
   b_\eps^0 = -\rho < b_\eps^1 < \ldots < b_\eps^{N_\eps} = \rho
  \]
  such that, for any~$h\in\{1, \, \ldots, \, N_\eps - 1\}$,
  \begin{gather}
   \abs{b_\eps^h - \left(-\rho + \frac{2\rho h}{N_\eps}\right)}
    \leq \frac{\rho}{2N_\eps} \label{ubH12-2bis} \\
   \int_{[-\rho, \, \rho]\times \{b_\eps^h\}} 
    \bar{e}_\eps(v_\eps\circ T) \, \d\mathscr{H}^1 
   \leq \frac{N_\eps}{\rho}  \int_{[-\rho, \, \rho]^2}
    \bar{e}_\eps(v_\eps\circ T) \, \d x \leq \frac{C\abs{\log\eps}}{h(\eps)}
    \label{ubH12-3bis}
  \end{gather}
  We consider the collection of sets~$K_{k,h} :=
  T([a^k_\eps, \, a^{k+1}_\eps]\times
  [b^h_\eps, \, b^{h+1}_\eps])$, for~$k\in\{0, \, \ldots, \, N_\eps - 1\}$
  and~$h\in\{0, \, \ldots, \, N_\eps -1\}$. These sets cover~$B_\rho(0)$,
  have pairwise disjoint interiors, satisfy the condition~(ii) above 
  (thanks to~\eqref{ubH12-2}, \eqref{ubH12-2bis}) and
  \[
   \int_{\bigcup_{k,h} \partial T_{k,h}} \bar{e}_\eps(v_\eps) \,\d\mathscr{H}^1
   \leq  \frac{C\abs{\log\eps}}{h(\eps)}
  \]
  We construct a cover of~$B_{2\rho}(0)\setminus\bar{B}_\rho(0)$
  by considering suitable points
  \[
   \rho_\eps^0 = \rho < \rho_\eps^1 < \ldots < \rho^{N_\eps} = 2\rho, \qquad
   \theta_\eps^0 = 0 < \theta_\eps^1 < \ldots < \rho^{N_\eps} = 2\pi
  \]
  and defining, for any~$k\in\{0, \, \ldots, \, N_\eps - 1\}$,
  $h\in\{0, \, \ldots, \, N_\eps -1\}$, 
  \[
   \tilde{K}_{k,h} := \left\{(\rho\cos\theta, \, \rho\sin\theta)\colon
   (\rho, \, \theta)\in[\rho^k_\eps, \, \rho^{k+1}_\eps]\times
   [\theta^h_\eps, \, \theta^{h+1}_\eps]\right\}
  \]
  For a suitable choice of~$\rho_\eps^k$ and~$\theta_\eps^h$
  (analogous to~\eqref{ubH12-2}, \eqref{ubH12-3}), we can make sure that
  the collection~$\mathscr{G}_\eps := \{K_{k,h}, \, \tilde{K}_{k,h}\}_{k,h}$
  satisfies all the properties~(i)--(v) above. 
 \end{step}
 
 \begin{step}[Construction of a suitable measure]
  An argument based on Sobolev embeddings (e.g.,
  \cite[Lemma~A.1]{JerrardIgnat_full})
  shows that 
  \begin{equation} \label{ubH13-1}
   \sup_{x\in R_\eps} \abs{1 - \abs{v_\eps(x)}} 
   \leq C\eps^{1/2}
   \left(\int_{R_\eps} \bar{e}_\eps(v_\eps) \, \d\mathscr{H}^1\right)^{1/2} 
   \stackrel{\eqref{ubH12-1}}{\leq}
   \frac{C\eps^{1/2}}{h(\eps)^{1/2}} \abs{\log\eps}^{1/2}  \to 0
  \end{equation}
  as~$\eps\to 0$. In particular, for any~$K\in\mathscr{G}_\eps$, the 
  projection~$v_\eps/\abs{v_\eps}\colon\partial K\to\S^1$ 
  is well-defined and continuous. We denote by~$\deg(v_\eps, \, \partial K)$
  the topological degree of~$v_\eps/\abs{v_\eps}$ on~$\partial K$.
  We define the measure
  \begin{equation} \label{ubH13-2}
   \mu_\eps := 2\pi\sum_{K\in\mathscr{G}_\eps} 
    \deg(v_\eps, \, \partial K) \, \delta_{b_K},
  \end{equation}
  where~$b_K$ is an arbitrary point in the interior of~$K$.
  Let~$\delta > 0$ be a small parameter. Classical results by
  Sandier~\cite{Sandier} and Jerrard~\cite{Jerrard}
  (see also \cite[Lemma~3.10]{ABO2}
  for a statement that is particular suited to our needs) imply
  \begin{equation} \label{ubH13-3}
   \begin{split}
    &\abs{\deg(v_\eps, \, \partial K)}
    \left(\log\frac{h(\eps)}{\eps} 
     - C_\delta \left(1 + \log\abs{\deg(v_\eps, \, \partial K)}\right) \right) \\
    &\qquad\qquad
    \leq C\int_K \bar{e}_\eps(v_\eps) \, \d x + C \delta h(\eps) 
     \int_{\partial K} \bar{e}_\eps(v_\eps) \, \d\mathscr{H}^1
   \end{split}
  \end{equation}
  for any~$K\in\mathscr{G}_\eps$, where~$C_\delta$ is a constant 
  that depends only on~$\delta$ and~$\rho$. By dividing both 
  sides of~\eqref{ubH13-3} by~$\abs{\log\eps}$, we deduce
  \begin{equation} \label{ubH13-4}
   \begin{split}
    &\abs{\deg(v_\eps, \, \partial K)} 
    \leq \frac{C}{\abs{\log\eps}} \int_K \bar{e}_\eps(v_\eps) \, \d x
    + \frac{C\delta h(\eps)}{\abs{\log\eps}} 
     \int_{\partial K} \bar{e}_\eps(v_\eps) \, \d\mathscr{H}^1
   \end{split}
  \end{equation}
  By applying~\eqref{ubH11-1}, \eqref{ubH11-2} and~\eqref{ubH12-1}, we deduce
  \begin{align}
   \abs{\mu_\eps}(B_{2\rho}(0)) 
    &= 2 \pi\sum_{K\in\mathscr{G}_\eps} 
     \abs{\deg(v_\eps, \, \partial K)} \leq C \label{ubH13-5} \\
   \abs{\mu_\eps}(B_{2\rho}(0)\setminus\bar{B}_\rho(0)) 
    &= 2 \pi\sum_{K\in\mathscr{G}_\eps\colon K
    \subseteq B_{2\rho}(0)\setminus B_\rho(0)}
    \abs{\deg(v_\eps, \, \partial K)} \leq  C\delta \label{ubH13-6}
  \end{align}
 \end{step}
 
 \begin{step}[Convergence of~$\mu_\eps$]
  We claim that
  \begin{equation} \label{ubH14-1}
   \mu_\eps \to 2\pi d_j \delta_0
   \qquad \textrm{in } W^{-1,1}(B_{2\rho}(0))
   \quad \textrm{as } \eps\to 0.
  \end{equation}
  Let~$\varphi\in W^{1,\infty}_0(B_{2\rho}(0))$ be a test function.
  We have
  \begin{equation} \label{ubH14-2}
   \begin{split}
    &\abs{\int_{B_{2\rho}(0)} \varphi \, \d\bar{\jmath}(v_\eps)
     - \int_{B_{2\rho}(0)} \varphi \, \d\mu_\eps} \\
    &\qquad \leq \underbrace{\abs{\sum_{K\in\mathscr{G}_\eps} 
     \int_{K} (\varphi - \varphi(b_K)) 
     \, \d\bar{\jmath}(v_\eps)}}_{=: I_1} 
     + \underbrace{\abs{ \sum_{K\in\mathscr{G}_\eps} \varphi(b_K) 
     \left(\int_{K} \d\bar{\jmath}(v_\eps)
     - 2\pi\deg(v_\eps, \, \partial K)\right)}}_{=:I_2}
   \end{split}
  \end{equation}
  To estimate the term~$I_1$, we observe that
  $\abs{\d\bar{\jmath}(v_\eps)} \leq \abs{\nabla v_\eps}^2 
  \leq 2\bar{e}_\eps(v_\eps)$ and hence,
  \begin{equation} \label{ubH14-3}
   \begin{split}
    I_1 \leq C h(\eps) \norm{\nabla\varphi}_{L^\infty(B_{2\rho}(0))}
     \int_{B_{2\rho}(0)} \bar{e}_\eps(v_\eps) \, \d x
    \stackrel{\eqref{ubH11-1}}{\leq} C h(\eps) \abs{\log\eps}
     \norm{\nabla\varphi}_{L^\infty(B_{2\rho}(0))}
   \end{split}
  \end{equation}
  We have $h(\eps)\abs{\log\eps} = \abs{\log\eps}^{-1}\to 0$ 
  as~$\eps\to 0$. To estimate~$I_2$, we apply Stokes' theorem
  and~\eqref{eq:index1}:
  \begin{equation*} 
   \begin{split}
    I_2 \leq \abs{\sum_{K\in\mathscr{G}_\eps}
     \int_{\partial K} \bar{\jmath}(v_\eps)
     \left(1 - \frac{1}{\abs{v_\eps}^2}\right)}
     \norm{\varphi}_{L^\infty(B_{2\rho}(0))}
   \end{split}
  \end{equation*}
  The~$L^\infty$-norm of~$1 - \abs{v_\eps}^{-2}$ on~$R_\eps$
  can be bounded from below by~\eqref{ubH13-1}. Moreover, we have
  $\abs{\bar{\jmath}(v_\eps)} \leq \abs{v_\eps} \abs{\nabla v_\eps}$
  and, since~$\abs{v_\eps}$ is bounded on~$R_\eps$ by~\eqref{ubH13-1}, we obtain
  \begin{equation*}
   \begin{split}
    I_2 &\leq \frac{C\eps^{1/2}\abs{\log\eps}^{1/2}}{h(\eps)^{1/2}}
     \left(\int_{R_\eps} \bar{e}_\eps(v_\eps)^{1/2}\d\mathscr{H}^1\right)
     \norm{\varphi}_{L^\infty(B_{2\rho}(0))} \\
    &\leq \frac{C\eps^{1/2}\abs{\log\eps}^{1/2}}{h(\eps)^{1/2}} 
     \left(\int_{R_\eps} \bar{e}_\eps(v_\eps)\d\mathscr{H}^1\right)^{1/2} 
     \mathscr{H}^1(R_\eps)^{1/2} \norm{\varphi}_{L^\infty(B_{2\rho}(0))}
   \end{split}
  \end{equation*}
  By construction, $\mathscr{H}^1(R_\eps) \leq C h(\eps)^{-1}$.
  By applying~\eqref{ubH12-1}, we deduce
  \begin{equation} \label{ubH14-4}
   \begin{split}
    I_2 \leq \frac{C\eps^{1/2}}{h(\eps)^{3/2}} \abs{\log\eps} 
     \norm{\varphi}_{L^\infty(B_{2\rho}(0))}
    = C \eps^{1/2} \abs{\log\eps}^4 \norm{\varphi}_{L^\infty(B_{2\rho}(0))}
   \end{split}
  \end{equation}
  Combining~\eqref{ubH14-2} with~\eqref{ubH14-3}, \eqref{ubH14-4},
  we obtain
  \begin{equation*}
   \begin{split}
    &\abs{\int_{B_{2\rho}(0)} \varphi \, \d\bar{\jmath}(v_\eps)
     - \int_{B_{2\rho}(0)} \varphi \, \d\mu_\eps}  \\
    &\qquad\qquad \leq C\left(\abs{\log\eps}^{-1} 
     + \eps^{1/2} \abs{\log\eps}^4\right)
    \left(\norm{\varphi}_{L^\infty(B_{2\rho}(0))} 
     + \norm{\nabla\varphi}_{L^\infty(B_{2\rho}(0))}\right)
   \end{split}
  \end{equation*}
  that is, $\d\bar{\jmath}(v_\eps) - \mu_\eps \to 0$ strongly
  in~$W^{-1,1}(B_{2\rho}(0))$. Now \eqref{ubH14-1}
  follows, due to~\eqref{ubH11-4}.
 \end{step}
 
 \begin{step}[Conclusion]
  Let~$\psi\in C^\infty_{\mathrm{c}}(B_{2\rho}(0)$ be such that
  $0 \leq \psi \leq 1$ and~$\psi = 1$ in~$B_\rho(0)$. By 
  construction of the measure~$\mu_\eps$ (see~\eqref{ubH13-2})
  and additivity of the degree, we have
  \begin{equation} \label{ubH15-1}
   \begin{split}
    2\pi\deg(v_\eps, \, \partial B_\rho(0))
     = \mu_\eps(B_\rho(0)) = \int_{B_{2\rho}(0)} \psi\,\d\mu_\eps
     - \int_{B_{2\rho}(0)\setminus\bar{B}_\rho(0)} \psi\,\d\mu_\eps
   \end{split}
  \end{equation}
  By~\eqref{ubH14-1}, we have
  \begin{equation} \label{ubH15-2}
   \begin{split}
    \int_{B_{2\rho}(0)} \psi\,\d\mu_\eps \to 2\pi d_j \psi(0)
    = 2\pi d_j
    \qquad \textrm{as } \eps\to 0.
   \end{split}
  \end{equation}
  On the other hand,
  \begin{equation} \label{ubH1-3}
   \abs{\int_{B_{2\rho}(0)\setminus\bar{B}_\rho(0)} \psi\,\d\mu_\eps}
   \leq \abs{\mu_\eps} (B_{2\rho}(0)\setminus\bar{B}_\rho(0))
   \stackrel{\eqref{ubH13-6}}{\leq} C\delta
  \end{equation}
  Therefore, for~$\eps$ sufficiently small we have
  \[
   \abs{\deg(v_\eps, \, \partial B_\rho(0)) - d_j} \leq C\delta
  \]
  We choose $\delta$ small enough, so that~$C\delta < 1$.
  Since both~$\deg(v_\eps, \, \partial B_\rho(0))$ and~$d_j$
  are integers, we conclude that $\deg(v_\eps, \, \partial B_\rho(0)) = d_j$
  for~$\eps$ small enough, and the lemma follows. \qedhere
 \end{step}
\end{proof}

\end{appendix}

\bibliographystyle{plain}
\bibliography{vorticesGL}

\end{document}